\pdfoutput=1
\documentclass[11pt]{article}
\usepackage[margin=1in]{geometry}
\usepackage{amsmath,amssymb,amsthm,mathtools}
\usepackage{booktabs}
\usepackage[expansion=false]{microtype}
\usepackage[hidelinks]{hyperref}
\usepackage{enumitem}
\usepackage{xcolor}

\numberwithin{equation}{section}
\newtheorem{theorem}{Theorem}[section]
\newtheorem{proposition}[theorem]{Proposition}
\newtheorem{lemma}[theorem]{Lemma}
\newtheorem{corollary}[theorem]{Corollary}
\theoremstyle{definition}
\newtheorem{definition}[theorem]{Definition}
\newtheorem{example}[theorem]{Example}
\newtheorem{problem}[theorem]{Question}
\newtheorem{observation}[theorem]{Observation}
\theoremstyle{remark}
\newtheorem{remark}[theorem]{Remark}

\newcommand{\Z}{\mathbb Z}
\newcommand{\Q}{\mathbb Q}
\newcommand{\F}{\mathbb F}
\newcommand{\N}{\mathbb N}
\newcommand{\fall}[2]{(#1)_{#2}}
\newcommand{\coeff}[2]{[x^{#1}]\,#2}
\newcommand{\OEIS}[1]{\href{https://oeis.org/#1}{\texttt{#1}}}

\newcommand{\dens}{\operatorname{dens}}
\title{A $p$-adic Valuation Calculus for Permutations with Restricted Cycle Lengths}
\author{Ahaan Kallat\\Northeastern University\\\texttt{kallat.a@northeastern.edu}}
\hypersetup{pdftitle={A p-adic Valuation Calculus for Permutations with Restricted Cycle Lengths},pdfauthor={Ahaan Kallat}}
\date{October 2026}

\begin{document}
\maketitle

\begin{abstract}
Let $a_A(n)$ be the number of permutations of $\{1,\dots,n\}$ all of whose cycle lengths lie in a set $A$.  We develop a valuation calculus for these numbers.  M\"obius inversion writes the exponential generating function as an Euler product, and at a prime $p$ the factors with non-$p$-integral exponents, the blocks, determine a finite knapsack problem.  Its optimal value gives a lower bound for $v_p(a_A(n))$, and the bound is attained exactly when one residue modulo $p$, the sum of the terms of the optimal choices, does not vanish.  For a fixed set $A$ and prime $p$ the bounds and residues for all $n\le N$ cost $O(N^2)$ operations.  When one block is strictly the most efficient, the valuation inside each residue class is that of a locally analytic $p$-adic function.  We use the calculus to determine the $2$-adic and $3$-adic valuations for cycle lengths at most $3$, $4$ and $5$, to prove the divisibility conjectures recorded in \OEIS{A378495}, \OEIS{A079128}, \OEIS{A193437} and \OEIS{A193438} and the $3$-adic statements recorded in \OEIS{A057693} and \OEIS{A070945}, and to treat $99$ OEIS sequences in all.  We also prove sharp divisibility and gcd theorems, show that at $n=111$ the gcd of the counts over all sets of excluded cycle lengths containing $\{1,2,3\}$ is one third of the gcd over the sets $\{1,2,3,k\}$, and determine eventual divisibility and tail gcds modulo composite integers.  Part of the paper is formalized in Lean~4.
\end{abstract}

\section{Introduction}\label{sec:intro}

\subsection{Where the questions come from}

The derangement numbers count the permutations of $\{1,\dots,n\}$ with no fixed point.  For $n=0,1,\dots,7$ they are
\[
1,\ 0,\ 1,\ 2,\ 9,\ 44,\ 265,\ 1854.
\]
From $n=2$ on, the $n$th term is divisible by $n-1$, as in $1854=6\cdot309$.  There is a one-line reason, and it works for many other families of permutations.  Look at the cycle containing the point $n+1$.  If its length is $m$, there are $n(n-1)\cdots(n-m+2)$ ways to choose and order its other points, and the remaining $n+1-m$ points carry a derangement.  A derangement has no cycle of length $1$, so $m\ge2$, and every one of these products is divisible by $n$.  Hence $n$ divides the number of derangements of $n+1$ points.

To state the general principle, let $S$ be a set of positive integers and let $a_S(n)$ be the number of permutations of $\{1,\dots,n\}$ with no cycle length in $S$, with $a_S(0)\coloneqq1$.  Derangements are the case $S=\{1\}$.  Involutions are the case $S=\{m:m>2\}$, and the permutations with all cycles of length at least $k$ are the case $S=\{1,\dots,k-1\}$.  By the exponential formula~\cite{Stanley,Wilf},
\begin{equation}\label{eq:egf}
\sum_{n\ge0}a_S(n)\frac{x^n}{n!}=\exp\Bigl(\sum_{m\ge1,\ m\notin S}\frac{x^m}{m}\Bigr).
\end{equation}
The argument above is the cycle of the point $n+1$ again.  Write $\fall{n}{j}\coloneqq n(n-1)\cdots(n-j+1)$ for the falling factorial.  If the cycle has length $m\notin S$, there are $\fall{n}{m-1}$ ways to choose and order its other points, and the remaining $n+1-m$ points carry a permutation counted by $a_S(n+1-m)$.  Hence
\begin{equation}\label{eq:last0}
a_S(n+1)=\sum_{m\notin S}\fall{n}{m-1}\,a_S(n+1-m).
\end{equation}
If $\{1,\dots,r\}\subseteq S$ then every allowed $m$ exceeds $r$, so each $\fall{n}{m-1}$ is divisible by $\fall{n}{r}$, and we learn that $\fall{n}{r}$ divides $a_S(n+1)$.  This is the first statement of Theorem~\ref{thm:general} below.

Divisibility questions of this kind appear as conjectures in the On-Line Encyclopedia of Integer Sequences (OEIS)~\cite{OEIS}, and the answers to them are sharper than the argument above.  Take the triangle \OEIS{A378495}, where $T(n,k)=a_{\{1,k\}}(n)$ is the number of derangements of $\{1,\dots,n\}$ with no $k$-cycle ($2\le k\le n$, the entry's column $k=1$ being zero).  Row $7$ is $1140,1224,1434,1350,1854,1134$.  Every entry is divisible by $6=7-1$, and the quotients $190,204,239,225,309,189$ have gcd $1$, so $n-1$ is exactly the gcd of this row.

The same question can be asked for primes.  Write $v_p(m)$ for the exponent of the prime $p$ in $m$, and let $A_r(n)$ be the number of permutations of $\{1,\dots,n\}$ all of whose cycles have length at most $r$, which is $a_S(n)$ for $S=\{m:m>r\}$.  For $r=3$ (\OEIS{A057693}) the numbers are $A_3(n)=1,1,2,6,18,66,276,1212,5916,31068,\dots$, and their $2$-adic valuations for $n=2,\dots,13$ are $1,1,1,1,2,2,2,2,3,3,3,3$.  The values come in blocks of four, and each block is one larger than the last.  The $3$-adic valuations for $n=3,\dots,14$ are $1,2,1,1,1,1,2,2,2,3,3,3$, and the value $2$ at $n=4$ stands out.  Does the first pattern continue forever, and which $n$ break the second one?  Derangements already show what a good answer looks like for a prime.  Write $D(n)$ for the number of derangements of $\{1,\dots,n\}$.  Then $3\mid D(n)$ exactly when $n\equiv1\pmod3$, as in $D(4)=9$ and $D(7)=1854=3\cdot618$.  Section~\ref{sec:euler} explains why, for every set of allowed cycle lengths and every prime.

\subsection{Main results}\label{sec:main}

Throughout, $A\subseteq\N$ is a set of allowed cycle lengths and $a_A(n)$ is the number of permutations of $\{1,\dots,n\}$ all of whose cycle lengths lie in $A$, so that $a_A(n)=a_S(n)$ for $S=\N\setminus A$.  Its generating function is $F_A(x)=\sum_na_A(n)x^n/n!=\exp\bigl(\sum_{m\in A}x^m/m\bigr)$.  We state six results here, in the order of the methods, and list their consequences for OEIS entries in Section~\ref{sec:oeistable}.

\paragraph{The valuation calculus.}  The first theorem explains the divisibility of the derangement numbers by $3$, and it computes the $p$-adic valuation of $a_A(n)$ for every set $A$ and every prime $p$, apart from an explicit set of $n$.  M\"obius inversion factors the generating function as an Euler product $F_A=\prod_{m\ge1}(1-x^m)^{-c_m}$ with $c_m=\frac1m\sum_{d\mid m,\,d\in A}\mu(m/d)$.  Fix a prime $p$.  A factor whose exponent $c_m$ is a $p$-adic integer has $p$-integral coefficients and plays no role.  The other factors are called \emph{blocks}.  They have $c_m=\alpha'_m/p^{j_m}$ with $j_m\ge1$ and $\alpha'_m$ a $p$-adic unit, and the coefficient of $y^\lambda$ in $(1-y)^{-c_m}$ has valuation exactly $-j_m\lambda-v_p(\lambda!)$.  We call $j_m$ the \emph{depth} of the block $m$ and the number $j_m\lambda+v_p(\lambda!)$ the \emph{cost} of $\lambda$ copies of it.  The cost is the exponent of $p$ in the denominator that these copies contribute, and the copies use the degree $m\lambda$.  Let $\kappa_A(n)$ be the largest total cost of a choice of copies whose degrees add up to at most $n$, and call a choice \emph{optimal} when it attains $\kappa_A(n)$.  No term of the expansion of $a_A(n)/n!$ has a larger denominator, and the terms of the optimal choices are the only ones that can cancel at this level.  For derangements at $p=3$ the factor with exponent $c_3=\frac13$ is a block of depth $1$.  The coefficients of $y$, $y^2$ and $y^3$ in $(1-y)^{-1/3}$ are $\frac13$, $\frac29$ and $\frac{14}{81}$, with valuations $-1$, $-2$ and $-4=-3-v_3(3!)$.

\begin{theorem}[Valuation calculus]\label{thm:introcalc}
For every set $A$, every prime $p$ and every $n\ge0$,
\[
v_p\bigl(a_A(n)\bigr)\ \ge\ v_p(n!)-\kappa_A(n),
\]
with equality if and only if the terms of the optimal choices do not cancel modulo $p$.  Their sum, multiplied by $p^{\kappa_A(n)}$ and reduced modulo $p$, is a residue $\Xi_A(n)\in\F_p$, and equality means $\Xi_A(n)\ne0$.  Only finitely many blocks matter, and for fixed $A$ and $p$ the numbers $\kappa_A(n)$ and $\Xi_A(n)$ for all $n\le N$ are computed with $O(N^2)$ operations.
\end{theorem}

The residue $\Xi_A(n)$ is a sum over the optimal choices, built from the residues of the numbers $\alpha'_m$ and from the base-$p$ digits of the multiplicities (Theorem~\ref{thm:calculus}, Corollary~\ref{cor:calc} and Lemma~\ref{lem:essential}).  For derangements at $p=3$ the optimal choice is $\lfloor n/3\rfloor$ copies of the block $3$, and the theorem says that $3\mid D(n)$ exactly when $n\equiv1\pmod3$ (Example~\ref{ex:der}).  The same statement holds for the quotient $a_S(n)/\fall{n-1}{r}$ when $S\supseteq\{1,\dots,r\}$ (Theorem~\ref{thm:quotient}).  The set of $n$ where the inequality is strict is the \emph{exceptional set}.  In the simplest case one block is both the smallest and the most efficient, the optimal choice uses as many copies of it as fit, and the exceptional set is a union of residue classes modulo that block (Theorem~\ref{thm:dominant}).  For $(p,k)\ne(2,1)$ and $\max(1,p^k-p)\le r\le p^k-1$ the block $p^k$ plays this role for $A_r$, and the Harmonic Theorem names the exceptional classes when $p^k\le2r+2$.  They are the classes $\nu$ modulo $p^k$ with $r<\nu<p^k$ and $\sum_{r<i\le\nu}1/i\equiv1\pmod p$ (Theorem~\ref{thm:harmonic}).  The smallest instance is $p=3$ and $r=6$, where $v_3(A_6(n))=\lfloor n/3\rfloor-\lfloor n/9\rfloor$ for every $n$ except $n\equiv7\pmod9$ (Corollary~\ref{cor:plateau}).

\paragraph{Short cycles.}  For cycle lengths at most $3$, $4$ and $5$ two or three blocks compete, and the calculus becomes a small exact computation.  It answers the two questions about $A_3$ asked above.  The blocks of four continue for ever, and the $3$-adic pattern breaks exactly on the class $n\equiv4\pmod{27}$.

\begin{theorem}[Short cycles]\label{thm:introshort}
Let $e(n)=\lfloor n/9\rfloor+\lfloor(n+6)/9\rfloor$, and let $\gamma(0),\dots,\gamma(7)$ be $0,0,1,1,3,3,3,3$.
\begin{enumerate}[label=(\alph*),itemsep=0pt]
\item $v_2(A_3(n))=1+\lfloor(n-2)/4\rfloor$ for all $n\ge2$.
\item For $r\in\{3,4,5\}$ and all $n\ge0$ one has $v_3(A_r(n))\ge e(n)$, with equality except for $n\equiv4\pmod{27}$ when $r=3$ and for $n\equiv14,23\pmod{27}$ when $r=4$.  For $r=5$ equality holds for every $n$.
\item For $r\in\{4,5\}$ and $n=8K+\nu$ with $0\le\nu\le7$ one has $v_2(A_r(n))\ge4K+\gamma(\nu)$, with equality except for $n\equiv5,12\pmod{16}$ when $r=4$ and for $n\equiv12,13\pmod{16}$ when $r=5$.
\item $v_2\bigl(A_4(2n+1)-A_4(2n)\bigr)\ge n+1$ for every $n\ge1$.
\end{enumerate}
\end{theorem}

Part~(a) is Corollary~\ref{cor:plateau}(a), the case $p^k=4$ of the Harmonic Theorem.  Part~(b) is Theorem~\ref{thm:p3}, part~(c) is Theorems~\ref{thm:p2r4} and~\ref{thm:p2r5}, and part~(d) is Corollary~\ref{cor:bala2}.  Together they prove the statements of Bala recorded in \OEIS{A057693} and \OEIS{A070945}, with the $3$-adic ones in the form of Corollary~\ref{cor:bala}, and part~(d) is the third inequality of \OEIS{A070945} with $a(2n)$ in place of the printed $a(n)$.  For $A_3$ one has $A_3(4)=18$, so $v_3(A_3(4))=2>e(4)=1$, and $n=4$ is the first member of the exceptional class modulo $27$.

\paragraph{Divisibility and gcds.}  The one-line argument for derangements at the start of this section is the case $r=1$ of part~(a) below, and part~(b) settles the row-gcd conjecture of the triangle \OEIS{A378495}.

\begin{theorem}[Divisibility and gcds]\label{thm:introgcd}
\begin{enumerate}[label=(\alph*),itemsep=0pt]
\item If $\{1,\dots,r\}\subseteq S$, then $(n-1)(n-2)\cdots(n-r)$ divides $a_S(n)$ for all $n\ge1$.
\item For $n\ge3$ the gcd of the numbers $a_{\{1,k\}}(n)$ with $2\le k\le n$ is $n-1$, and so is the gcd of the numbers $a_{\{1,\dots,k-1\}}(n)$ with $1\le k\le n$.
\item The gcd of the numbers $a_S(111)$ over all sets $S\supseteq\{1,2,3\}$ is one third of their gcd over the sets $S=\{1,2,3,k\}$.
\end{enumerate}
\end{theorem}

Part~(a) is Theorem~\ref{thm:general}(i), and part~(b) is Theorems~\ref{thm:A378}(b) and~\ref{thm:A333726}.  Part~(c) is the surprising one.  One might expect the sets $\{1,\dots,r,k\}$ to determine the gcd over all sets containing $\{1,\dots,r\}$, and they do at every prime above $r$, for every $r$, and at every prime when $r\le2$ (Theorem~\ref{thm:gcdfamily}).  At the prime $3$ the set $S=\N\setminus\{4,9\}$, whose permutations have all cycles of length $4$ or $9$, gives a number $a_S(111)$ with one factor $3$ fewer than the gcd over the family, and $n=111$ is the smallest case found (Theorem~\ref{thm:counter} and Section~\ref{sec:globalgcd}).

\paragraph{Cycle gcd.}  Counting permutations by the gcd of their cycle lengths reduces, by M\"obius inversion over the divisors of $n$, to explicit products.  The permutations of $\{1,\dots,n\}$ with relatively prime cycle lengths number $1,1,4,15,96,455$ for $n=1,\dots,6$ (\OEIS{A079128}), and indeed $15=4^2-1$ and $455=13\cdot(6^2-1)$.

\begin{theorem}[Cycle gcd]\label{thm:introcyc}
Let $n\ge4$, let $g\mid n$ with $n/g\ge3$, and let $C(n,g)$ be the number of permutations of $\{1,\dots,n\}$ whose cycle lengths have gcd $g$.  Then $n^2-1$ divides $C(n,g)$.  Moreover $C(n,1)$ is prime to $n$.
\end{theorem}

This proves the two conjectures recorded in the entry \OEIS{A079128} (Theorem~\ref{thm:cyclegcd} and Corollary~\ref{cor:A079128}).  The factor $n+1$ comes from Theorem~\ref{thm:cyclegcd}(b), which writes $C(n,g)$ modulo $n+1$ as $n!$ times a short M\"obius sum, and this sum vanishes when $n/g\ge3$.  For the permutations whose cycle lengths are all prime to $n$ (\OEIS{A335088}), the same idea reduces the count modulo a prime $n+1$ to a short prefix sum, with closed forms for Fermat primes and for safe primes, and $n+1$ divides the count for $n=708$ and $n=92202$ and for no other $n\le3\cdot10^6$ with $n+1$ prime (Theorem~\ref{thm:A335088} and Section~\ref{sec:coprime}).

\paragraph{Composite moduli.}  The calculus works one prime at a time, and a composite modulus combines primes of two kinds.  The residues of $A_3(n)$ modulo $5$ repeat $1,1,2,1,3$ for ever, so $5$ never divides $A_3(n)$, while every power of $2$ and of $3$ divides $A_3(n)$ for all large $n$.  The difference is that $2,3\le r$ and $5>r$.  Conjugating by a $p$-cycle explains this (Lemma~\ref{lem:A}).  It gives $a_A(n+p)\equiv([1\in A]-[p\in A])\,a_A(n)\pmod p$, where $[P]$ is $1$ if the statement $P$ holds and $0$ otherwise.  If $1$ and $p$ are both allowed or both forbidden, the factor is $0$ and $p$ divides $a_A(n)$ for every $n\ge p$.  Otherwise the residues repeat with period $p$ up to sign, as they do for $A_3$ modulo $5$.  Call a prime $p$ of \emph{type~I} for $A$ if $[1\in A]=[p\in A]$ and of \emph{type~II} otherwise, and write $m=m_{\mathrm I}m_{\mathrm{II}}$, where $m_{\mathrm I}$ collects the prime powers of $m$ at primes of type~I.  For $A=A_{\le r}$ the primes of type~I are the primes $p\le r$.

\begin{theorem}[Composite moduli]\label{thm:introcomp}
Let $A\subseteq\N$ and $m=m_{\mathrm I}m_{\mathrm{II}}$.  Then $m_{\mathrm I}$ divides $a_A(n)$ for all large $n$, the set $\{n:m_{\mathrm{II}}\mid a_A(n)\}$ is purely periodic, and $\{n:m\mid a_A(n)\}$ agrees with it for large $n$ and has density $\prod_{p^e\,\|\,m_{\mathrm{II}}}\dens\{n:p^e\mid a_A(n)\}$.  In particular $m\mid a_A(n)$ for all large $n$ if and only if $m_{\mathrm{II}}=1$.
\end{theorem}

The threshold from which $m_{\mathrm I}\mid a_A(n)$, the pre-period and the least period of $a_A(n)\bmod m$ are exact, and the integers that divide $a_A(n)$ for all $n\ge n_0$ are the divisors of a tail gcd, which a certified finite computation finds (Theorem~\ref{thm:composite} and Proposition~\ref{thm:compositeA}).  For \OEIS{A057693} this gives $\gcd_{n\ge n_0}A_3(n)=2^{1+\lfloor(n_0-2)/4\rfloor}3^{\lfloor n_0/9\rfloor+\lfloor(n_0+6)/9\rfloor}$ for every $n_0\ge0$, and $A_3(n)\bmod10$ is periodic from $n=2$ with period $5$ (Corollary~\ref{cor:A057693}).

\paragraph{Valuations inside a residue class.}  The calculus gives only a lower bound on the exceptional set.  When one block is strictly the most efficient, an analytic statement takes over.  The \emph{efficiency} of a block $m$ of depth $j_m$ is $\varepsilon_m=(j_m+\frac1{p-1})/m$.  By Legendre's formula $v_p(\lambda!)$ is about $\lambda/(p-1)$, so $\varepsilon_m$ is the cost per unit of degree of many copies of $m$.

\begin{theorem}[Valuations inside a residue class]\label{thm:introinterp}
Let $p$ be a prime and assume that $A$ has blocks at $p$ and that one block $m_*$, of depth $j$, satisfies $\varepsilon_{m_*}>\varepsilon_m$ for every other block $m$.  For each $0\le\rho<m_*$ there is a locally analytic function $F_\rho\colon\Z_p\to\Q_p$ such that
\[
v_p\bigl(a_A(n)\bigr)=v_p(n!)-jq-v_p(q!)+v_p\bigl(F_\rho(q)\bigr)\qquad\text{for every }n=m_*q+\rho\text{ with }q\ge0.
\]
In column $2$ of \OEIS{A378495} let $p$ be an odd prime, $Q_n=T(n,2)/((n-1)(n-2))$, and $0\le\nu<p$ with $v_p(Q_{\nu+3})=1$.  Then either there is a $p$-adic integer $q^*\notin\N$ with $v_p(Q_{\nu+3+pq})=1+v_p(q-q^*)$ for every $q\in\N$, or $v_p(Q_{\nu+3+pq})=1$ for every $q\in\N$, and an explicit number $\zeta$ decides which case occurs.
\end{theorem}

The first part is Theorem~\ref{thm:interp}, where $F_\rho$ is an explicit series in the coefficients of $F_A(1-x^{m_*})^{c_{m_*}}$.  The second part is Theorem~\ref{thm:qstar}, which extends the second-order Theorem~\ref{thm:col2} to every power of $p$.  At $p=19$ one has $Q_7=38=2\cdot19$, and $v_{19}(Q_{7+19q})\ge2$ exactly for $q\equiv9\pmod{19}$.

\subsection{Consequences for OEIS entries}\label{sec:oeistable}

The next table lists the entries whose recorded conjectures or statements the theorems prove, with the place of the proof.  The last column says what the entry itself records, and the name Bala refers to statements of Peter Bala in the entry.  The theorems apply to $99$ OEIS sequences in all, and the supplementary catalogue described in Section~\ref{sec:computations} records the statement for each of them.

\begin{center}\small
\renewcommand{\arraystretch}{1.15}
\begin{tabular}{@{}p{1.55cm}p{6.1cm}p{3.3cm}p{3.6cm}@{}}
\toprule
entry & what is proved & where & in the entry\\
\midrule
\OEIS{A378495} & row gcd $n-1$ for $n\ge3$, the differences $T(n,n-2)-T(n,n)$ and $T(n,n-1)-T(n,n-2)$, the $2$-adic valuations of the columns & Theorem~\ref{thm:A378} & the three statements as conjectures\\
\OEIS{A333726} & row gcd $n-1$ for $n\ge3$ & Theorem~\ref{thm:A333726} & ---\\
\OEIS{A079128} & $(n^2-1)\mid a(n)$ for $n\ge4$, and $\gcd(a(n),n)=1$ & Corollary~\ref{cor:A079128} & both as conjectures\\
\OEIS{A193437}, \OEIS{A193438} & $p^{\lfloor n/p\rfloor}\mid a(n)$ for every prime $p\equiv1\pmod3$, respectively $p\equiv1\pmod4$ & Theorem~\ref{thm:saturation} and the remark after it & both as conjectures\\
\OEIS{A335088} & $(n-1)\mid a(n)$, $\gcd(a(n),n)=1$, $(n+1)\mid a(n)$ for composite $n+1$, residue modulo a prime $n+1$ & Theorem~\ref{thm:A335088}, Section~\ref{sec:coprime} & ---\\
\OEIS{A057693} & $v_2(a(n))$ and $v_3(a(n))$ exactly, tail gcds, sharp thresholds for $2^i3^j\mid a(n)$, pre-period and period of $a(n)\bmod k$ & Corollaries~\ref{cor:plateau}, \ref{cor:bala} and~\ref{cor:A057693} & $v_2$ stated, $3$-adic values as conjectures, eventual periodicity (Bala)\\
\OEIS{A070945} & $v_2(a(n))$ and $v_3(a(n))$ exactly, $v_2(a(2n+1)-a(2n))\ge n+1$, tail gcds & Theorem~\ref{thm:p2r4}, Corollaries~\ref{cor:bala2}, \ref{cor:bala}, \ref{cor:A070945} & two inequalities with proof, a third misprinted, $3$-adic values as conjectures (Bala)\\
\OEIS{A070946} & $v_2$ and $v_3$ of $a(n)$, the $2$- and $3$-parts of the tail gcds, sharp thresholds for $2^i3^j\mid a(n)$ & Theorem~\ref{thm:p2r5}, Corollaries~\ref{cor:bala} and~\ref{cor:A070946} & eventual divisibility by $2^i3^j5^k$ (Bala)\\
\OEIS{A000085} & $v_2(a(n))$ exactly & Theorem~\ref{thm:invol} & known \cite{Ochiai,KimKim}\\
\OEIS{A038205} & $v_2(a(n))$ for all $n$, $v_p(a(n))$ at $p=3,5,7$, a $p$-adic integer $q^*$ in each class with $v_p(Q_{\nu+3})=1$ & Theorems~\ref{thm:A378}(f), \ref{thm:qstar}, Proposition~\ref{prop:col2} & ---\\
\bottomrule
\end{tabular}
\end{center}

\subsection{Overview of the proofs}\label{sec:overview}

Most of the ideas of the paper can already be seen on the derangement numbers $D(n)=1,0,1,2,9,44,265,1854$ ($n=0,\dots,7$) and on the triangle \OEIS{A378495}, so we follow these two examples through.

\emph{The last point.}  Removing the cycle through the last point gives the recurrence~\eqref{eq:last0}, and when $S\supseteq\{1,\dots,r\}$ every one of its terms is divisible by $\fall{n-1}{r}$ (Lemma~\ref{lem:rec} and Theorem~\ref{thm:general}(i)).  For derangements this is $n-1\mid D(n)$, as in $1854=6\cdot309$, and for column $2$ of the triangle it is $(n-1)(n-2)\mid T(n,2)$.  The quotient $Q_S(n)=a_S(n)/\fall{n-1}{r}$ is again a short sum over the cycle of the last point (Lemma~\ref{lem:quot}), and it carries the rest of the arithmetic.

\emph{A sign rule.}  Conjugation by a $q$-cycle on the last $q$ points permutes the permutations with no cycle length in $S$, and its orbits have size $1$ or $q$.  Counting the fixed points gives $a_S(n+q)\equiv-a_S(n)\pmod q$ for every prime $q\notin S$ when $1\in S$ (Lemma~\ref{lem:A}), and the same rule holds for the quotient (Lemma~\ref{lem:quot}).  Row $7$ of the triangle is $1140,1224,1434,1350,1854,1134$, all divisible by $6$, and the quotients $190,204,239,225,309,189$ have gcd $1$.  The sign rule applied to the quotients gives the gcd theorems, among them that $n-1$ is exactly the gcd of row $n$ (Theorem~\ref{thm:A378}(b)).  In column $2$ it gives $T(26,2)\equiv-T(7,2)\pmod{19}$, and $T(7,2)=1140=19\cdot60$, so $19$ divides $T(26,2)$ although it does not divide $25\cdot24$.

\emph{Blocks and the packing problem.}  For derangements the Euler exponents are $c_m=-\mu(m)/m$ for $m\ge2$.  At $p=3$ the factors whose exponent is not a $3$-adic integer are the squarefree multiples of $3$, all of depth $1$, and the smallest is the block $3$, with $c_3=\frac13$.  The coefficient of $x^{3\lambda}$ in $(1-x^3)^{-1/3}$ has valuation $-\lambda-v_3(\lambda!)$, the cost of $\lambda$ copies of $3$.  Trading a block $3d$ for $d$ copies of $3$ raises the cost, so for $n=3q+\nu$ with $0\le\nu\le2$ the only optimal choice is $q$ copies of $3$, and its leading coefficient is $D(\nu)/\nu!$ (Example~\ref{ex:der}).  Since $D(0),D(1),D(2)=1,0,1$, the residue vanishes exactly for $\nu=1$, and $3\mid D(n)$ exactly when $n\equiv1\pmod3$, as in $D(4)=9$ and $D(7)=1854=3\cdot618$.  In general the packing problem is a knapsack, which a dynamic program solves, and an exchange argument, trading a large block for several copies of the smallest one, shows that only finitely many blocks matter (Lemma~\ref{lem:essential}).

\emph{Competing blocks.}  When two or three blocks share the top of the packing problem, several optimal choices tie, and the residue $\Xi$ adds up their leading coefficients.  For $A_3$, $A_4$ and $A_5$ at $p=3$ the blocks that matter are $6$ and $9$, of depths $1$ and $2$, and a comparison of base-$3$ digit sums lists the optimal pairs (Lemma~\ref{lem:max3}).  For $A_3$ the residue vanishes exactly on the class $n\equiv4\pmod{27}$, which gives Theorem~\ref{thm:introshort}(b).

\emph{Column $2$.}  The sign rule gives $Q_{n+p}\equiv-Q_n\pmod p$ for the quotient $Q_n=T(n,2)/((n-1)(n-2))$ and every odd prime $p$, so $p$ divides no $Q_n$ as soon as it divides none of $Q_3,\dots,Q_{p+2}$ (Proposition~\ref{prop:col2}).  Here $Q_3,\dots,Q_7=1,1,2,8,38$, and $38=2\cdot19$, so the prime $19$ breaks the formula $v_p(T(n,2))=v_p((n-1)(n-2))$ (Example~\ref{ex:19}).  The $2$-adic valuations of $T(n,2)$ for $n=3,\dots,10$ are $1,1,3,5,2,2,5,7$, and they follow the four-class pattern of Theorem~\ref{thm:A378}(f).  The sign rule, the $2$-adic pattern and the growth of $Q_n$ together show that infinitely many odd primes divide some $Q_n$ (Theorem~\ref{thm:bad}).

\emph{Interpolation.}  The calculus applies to $Q_n$ as well, since $Q_{k+3}/k!$ is the coefficient of $x^k$ in $G=(1-x)^{-2}e^{-x-x^2/2}$ (Theorem~\ref{thm:quotient}).  At an odd prime $p$ the block $p$ has depth $1$, and every other block has depth $1$ and degree at least $3p$.  So in the class $n=\nu+3+pq$ only the copies of the block $p$ matter, and Theorem~\ref{thm:interp} writes $Q_n$ as $k!\,\beta_qF_\nu(q)$, where $k=\nu+pq$, $\beta_q$ is an explicit binomial coefficient and $F_\nu$ is locally analytic.  When $v_p(Q_{\nu+3})=1$ the reduction of $F_\nu/p$ modulo $p$ is a polynomial $\xi+\zeta z$ of degree at most one, and Weierstrass preparation lifts its root to a $p$-adic integer $q^*$ (Theorem~\ref{thm:qstar}).  For $p=19$ and $\nu=4$ the root is $q^*\equiv9\pmod{19}$, and the exact values give $q^*\equiv332\pmod{19^2}$ and $q^*\equiv5747\pmod{19^3}$.

\emph{Composite moduli.}  The sign rule also sorts the primes into two kinds.  If $1$ and $p$ are both allowed or both forbidden, conjugation by a $p$-cycle gives $p\mid a_A(n)$ for every $n\ge p$, and the shift congruences of~\cite{KallatShift} raise this to every power of $p$ for large $n$ (Theorem~\ref{thm:dichotomy}).  Otherwise $p$ is the single dominant block, and the residues of $a_A(n)$ modulo $p^e$ repeat up to sign with period $p^e$.  The Chinese remainder theorem combines the primes (Section~\ref{sec:composite}).  For derangements every prime is of the second kind, and at $p=3$ the class $n\equiv1\pmod3$ found above carries $v_3(D(n))\ge2$ throughout.

\emph{The global gcd and the cycle gcd.}  The counterexample at $n=111$ comes from a knapsack over cycle types, the same packing problem as above with all parts and not only the blocks (Proposition~\ref{prop:globalgcd}).  For the cycle gcd, M\"obius inversion writes every count as a sum of explicit products, whose residues modulo $n\pm1$ and whose valuations can be read off term by term (Section~\ref{sec:coprime}).

\subsection{Relation to earlier work}\label{sec:prior}

Recursions for the number of involutions, that is the solutions of $x^2=1$ in the symmetric group, go back to Chowla, Herstein and Moore~\cite{CHM}, and Chowla, Herstein and Scott~\cite{CHS} treat the solutions of $x^d=1$, which are the permutations whose cycle lengths divide $d$.  The shift congruence $a(n+k)\equiv a(n)a(k)\pmod k$ arose in the proof of Bala's conjecture on \OEIS{A028342} in~\cite{KallatBala}, and~\cite{KallatShift} proves it for arbitrary integer weights, and for rational weights modulo integers coprime to their denominators.  The proof in~\cite{KallatBala} (Proposition~2.2) lets $\Z/k$ act on coloured permutations.  For integer weights~\cite{KallatShift} proves the congruence with the modulus $\operatorname{lcm}(n,k)$ (Theorem~1 there), and it gives the exact modulus (Theorem~34), the lift at every prime (Theorem~39), and closed formulas modulo prime powers (Corollaries~46 and~50, Theorem~48, Corollary~51).\footnote{Numbers refer to the current version of~\cite{KallatShift}.}  We use these in Theorem~\ref{thm:dichotomy}, which restates them for $0/1$ weights, and Lemma~\ref{lem:A} is the case of a prime modulus, whose short proof we include for completeness.  The present paper continues the study of weighted permutation sequences begun in~\cite{KallatBala} and~\cite{KallatShift}, which prove congruences between the terms of such sequences.  Here we study the $p$-adic valuations of the counts of permutations with restricted cycle lengths, which we determine exactly outside explicit exceptional sets, and the shift congruences of~\cite{KallatShift} give the dichotomy at a prime.

The entries \OEIS{A057693}, \OEIS{A070945} and \OEIS{A070946} record statements of Bala on the $2$-adic and $3$-adic valuations of these sequences and on their eventual divisibility.  We describe them, together with what is proved here, after Corollary~\ref{cor:bala} and after Theorem~\ref{thm:p2r5}.

The $p$-adic valuations of the coefficients of exponentials of power series have been studied by Krattenthaler and M\"uller~\cite{KM}, whose truncated versions of Dwork's lemma give lower bounds and equality criteria.  Their Section~8 applies them to permutations with restricted cycle lengths, and for $r=p^k-1$ the lower bound in Part (a) of the Harmonic Theorem is the case of cycle lengths below $p^k$ of~\cite[Thm.~28 and the paragraph after it]{KM}.  For $p^k-p\le r\le p^k-1$ it is the case $\ell=k$ of~\cite[Cor.~4]{KM}, which gives the lower bound and an equality criterion at some $n$.  As far as we can tell, the new content here is the explicit base coefficient $1-H_r(\nu)$, and $1+H_r(\nu)$ for the complementary family $A_{>r}$, which identifies the exceptional classes, together with equality for every $n$.  In Corollary~\ref{cor:plateau}, the lower bounds in (a) and (b) are cases of~\cite[Cor.~4]{KM}, equality for every $n$ is proved here from the Harmonic Theorem, and the formula in (a) was stated by Bala on \OEIS{A057693}.  The exact $2$-adic valuation of the number of involutions is known~\cite{Ochiai,KimKim}, and Theorem~\ref{thm:invol} recovers it as the one-block instance of the calculus.  The $p$-adic behaviour of the derangement numbers in Example~\ref{ex:der} is treated in~\cite{Miska}.  In the dichotomy of Theorem~\ref{thm:dichotomy}, case (B) is the $p$-adic continuity statement of~\cite[Thm.~1.5]{ODR} for $0/1$ weights, the explicit period $p^e$ comes from the shift congruences of~\cite{KallatShift}, and the case $e=1$ appears in~\cite{KM} (the remark after Theorem~2 and Corollary~11).  The block decomposition of the Euler product and the resulting exact valuation criterion for arbitrary sets $A$ extend the one-level analysis of~\cite[Thm.~2]{KM}.  The $p$-adic valuations of the numbers of solutions of $x^p=1$ are studied further in~\cite{IOTY}, and those of the numbers of permutations whose cycles all have one given length $d$ in~\cite{MiskaUlas}.  Congruences for Bell polynomials and for the cycle indicator of the symmetric group are given in~\cite{Carlitz} and~\cite{BO}, and they are related to the case of a prime shift $k=p$ of the congruence above.

The entry \OEIS{A378495} states three conjectures, the row gcd and the identities (c) and (d) of Theorem~\ref{thm:A378}, where in (d) the index of \OEIS{A238474} is $n-3$ in place of the entry's $n-4$.  The row-gcd formula and the two neighbouring-entry identities for \OEIS{A378495} also appear in~\cite{Su}.

\subsection{Notation and organization}

Throughout, $p$ is a prime, $s_p(n)$ is the sum of the base-$p$ digits of $n$, $\coeff{n}{F}$ is the coefficient of $x^n$ in a power series $F$, and $a_j=0$ for $j<0$.  We write $A_{\le r}=\{1,\dots,r\}$ and $A_{>r}=\{m:m>r\}$, and $A_r(n)=a_{A_{\le r}}(n)$ is the number of permutations of $\{1,\dots,n\}$ all of whose cycles have length at most $r$.  Symbols that are specific to one argument are defined where they are first used.  A result whose proof includes a finite computation is called computer-assisted, and its scripts are described in Section~\ref{sec:computations}.

Sections~\ref{sec:tools} to~\ref{sec:dichotomy} prove the divisibility, gcd and cycle gcd theorems (Theorems~\ref{thm:introgcd} and~\ref{thm:introcyc}) and the dichotomy at a prime.  A reader interested in valuations can go directly to Section~\ref{sec:euler}, which uses only Lemmas~\ref{lem:quot} and~\ref{lem:binom} from these sections.  Section~\ref{sec:euler} develops the calculus (Theorem~\ref{thm:introcalc}) and the Harmonic Theorem, Sections~\ref{sec:p3} and~\ref{sec:p2} prove Theorem~\ref{thm:introshort}, and Section~\ref{sec:globalgcd} treats the global gcd.  Section~\ref{sec:interp} describes the valuations inside a residue class (Theorem~\ref{thm:introinterp}), Section~\ref{sec:composite} proves Theorem~\ref{thm:introcomp} and the tail gcds of the OEIS entries, Section~\ref{sec:open} treats column $2$ of \OEIS{A378495} and the open questions, and Section~\ref{sec:computations} describes the computations and the formalization.  The appendix, Sections~\ref{app:AH} to~\ref{app:data} at the end of the paper, contains the proofs not given in the main text and the computational details.

\section{Four tools}\label{sec:tools}

The divisibility theorems rest on two facts about a single point and a single cycle.  Removing the cycle through the last point gives a recurrence (Lemma~\ref{lem:rec}), and conjugating by a $q$-cycle gives a sign rule modulo $q$ (Lemma~\ref{lem:A}), which survives division by $\fall{n-1}{r}$ (Lemma~\ref{lem:quot}).  Lemma~\ref{lem:binom}, on binomial series, is the source of the costs of Section~\ref{sec:main}, and it is the engine of the valuation calculus of Section~\ref{sec:euler}.

\begin{lemma}[Recurrences]\label{lem:rec}
Write $a_n=a_S(n)$, with $a_j=0$ for $j<0$.  For every set $S$ and $n\ge0$,
\begin{equation}\label{eq:last}
a_{n+1}=\sum_{m\ge1,\ m\notin S}\fall{n}{m-1}\,a_{n+1-m}.
\end{equation}
Now let $r\ge1$ and $\{1,\dots,r\}\subseteq S$.  Let $\theta_i=1$ if $r+1+i\in S$ and $\theta_i=0$ otherwise, and put $\delta_i=[i=0]-\theta_i+\theta_{i-1}$ (with $\theta_{-1}=0$), so that $\delta_i$ is the coefficient of $x^{r+i}$ in $(1-x)\sum_{m\notin S}x^{m-1}$.  Then for all $n\ge0$
\begin{equation}\label{eq:rec}
a_{n+1}=n\,a_n+\sum_{i\ge0}\delta_i\,\fall{n}{r+i}\,a_{n-r-i}.
\end{equation}
For $S=\{1,k\}$ with $k\ge2$ and $n\ge0$ this reads
\begin{equation}\label{eq:Rk}
a_{n+1}=n\,a_n+n\,a_{n-1}-\fall{n}{k-1}a_{n-k+1}+\fall{n}{k}a_{n-k},
\end{equation}
and for $k=2$ it reads
\begin{equation}\label{eq:R2}
a_{n+1}=n\,a_n+n(n-1)\,a_{n-2}.
\end{equation}
For $S=\{1,2,3\}$ it reads
\begin{equation}\label{eq:R3}
a_{n+1}=n\,a_n+n(n-1)(n-2)\,a_{n-3}.
\end{equation}
\end{lemma}

The idea is that $A=\exp(L)$ satisfies $A'=L'A$, and here $L'$ is a rational function whose denominator $1-x$ can be cleared.  The proof of Lemma~\ref{lem:rec} is in Appendix~\ref{app:moved}.

\begin{lemma}[Conjugation by a $q$-cycle]\label{lem:A}
Let $S$ be a set of positive integers, $q$ a prime, and $w_m=[m\notin S]$.  Then for all $n\ge0$
\[
a_S(n+q)\equiv\bigl(w_1+(q-1)w_q\bigr)\,a_S(n)\pmod q.
\]
In particular, if $1\in S$ and $q\notin S$, then $a_S(n+q)\equiv-a_S(n)\pmod q$, and if $1\in S$ and $q\in S$, then $a_S(n+q)\equiv0\pmod q$.  Consequently, for $n=tq+\rho$ with $t\ge0$ and $1\in S$, $q\notin S$, we have $a_S(n)\equiv(-1)^ta_S(\rho)\pmod q$.
\end{lemma}

Fix a $q$-cycle $\chi$ on the last $q$ points of $\{1,\dots,n+q\}$.\footnote{This is the standard fixed-point argument for a group of prime order acting on a finite set, and it is the argument of~\cite[Proposition~2.2]{KallatBala}.}  Conjugating by $\chi$ preserves cycle type, so it permutes the permutations we count.  The group $\langle \chi\rangle$ has prime order, so every orbit has size $1$ or $q$, and the count is congruent mod $q$ to the number of fixed points.  Fixed points are permutations commuting with $\chi$, which are easy to list.  The proof of Lemma~\ref{lem:A} is in Appendix~\ref{app:moved}.

Every allowed cycle length $m$ of a set $S\supseteq\{1,\dots,r\}$ exceeds $r$, so $\fall{n-1}{m-1}=\fall{n-1}{r}\fall{n-r-1}{m-r-1}$, and dividing \eqref{eq:last}, with $n-1$ in place of $n$, by $\fall{n-1}{r}$ gives the following quotient.

\begin{lemma}[The quotient by $\fall{n-1}{r}$]\label{lem:quot}
Let $r\ge1$, let $S\supseteq\{1,\dots,r\}$, and write $a=a_S$.  For $n>r$ put
\[
Q_S(n)=\sum_{\substack{r<m\le n\\ m\notin S}}\fall{n-r-1}{m-r-1}\,a(n-m).
\]
Then $a(n)=\fall{n-1}{r}\,Q_S(n)$, so $Q_S(n)$ is an integer, and for every prime $q\notin S$ and every $n>r$,
\[
Q_S(n+q)\equiv-Q_S(n)\pmod q .
\]
\end{lemma}

The proof of Lemma~\ref{lem:quot} is in Appendix~\ref{app:moved}.

For $e$ in a field of characteristic $0$ put $(1-y)^{-e}\coloneqq\exp\bigl(-e\log(1-y)\bigr)$ as a formal power series.  Let $\Z_{(p)}$ be the rationals with denominator prime to $p$.

\begin{lemma}[Binomial series]\label{lem:binom}
For every $c\in\Q$ and $\lambda\ge0$, the coefficient of $y^\lambda$ in $(1-y)^{-c}$ is $\binom{c+\lambda-1}{\lambda}=c(c+1)\cdots(c+\lambda-1)/\lambda!$.  If $c\in\Z_{(p)}$ this number lies in $\Z_{(p)}$.  If $v_p(c)=-j<0$ then its valuation is exactly $-j\lambda-v_p(\lambda!)$, and writing $c=\alpha/p^j$ with $\alpha\in\Z_{(p)}^\times$, the quantity $p^{j\lambda}\lambda!\binom{c+\lambda-1}\lambda$ equals $\prod_{i<\lambda}(\alpha+p^ji)\equiv\alpha^\lambda\pmod p$.
\end{lemma}

\begin{proof}
The coefficient of $y^\lambda$ in $\exp(-c\log(1-y))$ and the number $c(c+1)\cdots(c+\lambda-1)/\lambda!$ are polynomials in $c$ with rational coefficients, and they agree for $c=1,2,3,\dots$, where $(1-y)^{-c}$ is the $c$-th power of the geometric series.  So they are equal.  The polynomial $f(c)=\binom{c+\lambda-1}{\lambda}$ takes integer values at positive integers, and every $c\in\Z_{(p)}$ is a $p$-adic limit of positive integers $c_i$ while $f$ is continuous, so $f(c)=\lim f(c_i)\in\Z_p\cap\Q=\Z_{(p)}$.  For the last statement, $c+i=(\alpha+p^ji)/p^j$ and $\alpha+p^ji$ is a unit for every $i$.
\end{proof}

\section{Divisibility for general excluded sets}\label{sec:general}

This section proves the following theorem, whose part~(i) is Theorem~\ref{thm:introgcd}(a).

\begin{theorem}[Divisibility for excluded sets]\label{thm:general}
Let $S$ be a set of positive integers and $a(n)=a_S(n)$.
\begin{enumerate}[label=(\roman*),itemsep=0pt]
\item \emph{(Falling factorials)} If $\{1,\dots,r\}\subseteq S$, then $(n-1)(n-2)\cdots(n-r)\mid a(n)$ for all $n\ge1$.
\item \emph{(Congruence modulo $t!\,(n-1)$)} If $t\ge2$, $1\in S$ and $S\cap\{2,\dots,t+1\}=\emptyset$, then $a(n)\equiv D(n)\pmod{t!\,(n-1)}$ for all $n\ge2$.  In particular, if $1\in S$ and $2,3\notin S$, then $a(n)/(n-1)$ is odd for all $n\ge2$.
\item \emph{(Bound at a prime $p$)} If $\{1,\dots,p\}\subseteq S$, then for all $n\ge p+1$
\[
v_p\bigl(a(n)\bigr)\ \ge\ v_p\bigl((n-1)(n-2)\cdots(n-p)\bigr)+\Bigl\lfloor\frac{n-1}{2p}\Bigr\rfloor .
\]
Consequently $p^{\lceil n/(2p)\rceil}\mid a(n)$ for all $n\ge1$, and if $\{1,2\}\subseteq S$ then $2^{\lfloor(n-1)/4\rfloor}(n-1)(n-2)\mid a(n)$ for all $n\ge3$.
\end{enumerate}
\end{theorem}

The hypothesis $1\in S$ cannot be dropped from (i), since $S=\{2\}$ gives $a(3)=3$, which is not divisible by $n-1=2$.  The bound in (iii) is sharp for $p=2$, as the end of this section shows.

\subsection{Falling factorials}

\begin{proof}[Proof of Theorem~\ref{thm:general}(i)]
Let $n\ge r$ and apply \eqref{eq:last}.  Every $m\notin S$ satisfies $m\ge r+1$, so $\fall{n}{m-1}=\fall{n}{r}\fall{n-r}{m-1-r}$ is an integer multiple of $\fall{n}{r}=n(n-1)\cdots(n-r+1)$.  Hence $\fall{n}{r}\mid a_S(n+1)$, which is the claim for $n+1>r$.  For $1\le n\le r$ the statement is $a_S(n)=0$, which holds because every cycle length at most $n$ lies in $S$.  This completes the proof of Theorem~\ref{thm:general}(i).
\end{proof}

For $r=1$ and $r=2$, (i) also follows from the shift congruence of~\cite{KallatShift}, with $k=n-1$, $a(1)=0$, respectively $k=n-2$, $a(2)=0$ and the coprime moduli $n-1$, $n-2$.  For $r\ge3$ the lcm of the moduli $n-1,\dots,n-r$ can be smaller than their product, as for $r=3$ and $n=5$, where $24\mid a_S(5)$ is the claim and the lcm is $12$.

\subsection{Congruence modulo \texorpdfstring{$t!\,(n-1)$}{t!(n-1)}}

If $1\in S$, then \eqref{eq:rec} with $r=1$ says that $a_S(n+1)/n$ is an integer given by the previous two values and corrections $\delta_i\fall{n-1}{i}a_{n-1-i}$.  If $S$ avoids $\{2,\dots,t+1\}$, the corrections only come from $m\ge t+2$ and carry a factor $\fall{n-1}{i}$ with $i\ge t$ consecutive integers, so they vanish modulo $t!$, and what remains is the recurrence of the derangement numbers.  The proof of Theorem~\ref{thm:general}(ii) is in Appendix~\ref{app:moved}.

\begin{remark}\label{rem:t}
The hypotheses of (ii) cannot be weakened.  For $S=\{1,3\}$ ($3\in S$) the congruence fails at $n=3$ already for $t=2$, since $a_S(3)=0$ and $D(3)=2$ while $2(n-1)=4$.  For $S=\{1,4\}$ it holds for $t=2$ but not for $t=3$, at $n=4$, where $a_S(4)=3$, $D(4)=9$ and $3!(n-1)=18$.
\end{remark}

\subsection{The bound at a prime \texorpdfstring{$p$}{p}}

The same last-point recurrence as in Lemma~\ref{lem:quot} gives a bound that is uniform in the prime, by an induction on $n$ in which the term with the longest cycle supplies the factorial.  The proof of Theorem~\ref{thm:general}(iii) is in Appendix~\ref{app:moved}.

For $p=2$ the bound is attained at $n=4k$ with $S$ the set of all positive integers except $4$, for which $a_S(4k)=(4k)!/(4^kk!)$ has $2$-adic valuation $(4k-s_2(k))-2k-(k-s_2(k))=k$.  Theorem~\ref{thm:gcdfamily} shows that the bound for $S\supseteq\{1,2\}$ is also attained by a set of the form $\{1,2,k\}$ at every $n\ge4$.

\section{The triangle of derangements without a \texorpdfstring{$k$}{k}-cycle}\label{sec:A378}

This section proves the following theorem on the triangle \OEIS{A378495} of Section~\ref{sec:intro}, where $T(n,k)=a_{\{1,k\}}(n)$ counts the derangements of $\{1,\dots,n\}$ with no $k$-cycle.

\begin{theorem}[Derangements with no $k$-cycle, \OEIS{A378495}]\label{thm:A378}
\begin{enumerate}[label=(\alph*),itemsep=0pt]
\item For $2\le k\le n$, $(n-1)\mid T(n,k)$.
\item For $n\ge3$, $\gcd\bigl(T(n,2),\dots,T(n,n)\bigr)=n-1$.
\item For $n\ge5$, $T(n,n-2)-T(n,n)=\tfrac12(n-3)!\,(n-4)(n-1)$.
\item For $n\ge5$, $T(n,n-1)-T(n,n-2)=\dfrac{n!}{2(n-2)}$.
\item For every $k\ge4$ and $n\ge k$, $v_2\bigl(T(n,k)\bigr)=v_2(n-1)$.
\item For every $m\ge0$, the $2$-adic valuations of $T(n,2)$ at $n=4m+3,\,4m+4,\,4m+5,\,4m+6$ are
\[
m+1,\qquad m+1,\qquad m+3+v_2(m+1),\qquad m+5+v_2(m+1).
\]
\end{enumerate}
\end{theorem}

Part (b) is the row-gcd conjecture of the \OEIS{A378495} entry, and (c) and (d) are two further conjectures of that entry.\footnote{The entry states (d) as $|\text{A238474}(n-4)|$ for $n\ge4$.  This is correct only at $n=4$, where the difference is $3$.  For $n\ge5$ the identity holds with the index $n-3$, since $n!/(2(n-2))=|\text{A238474}(n-3)|$, and this is why (d) is stated for $n\ge5$.  For example, at $n=5$ we have $T(5,4)-T(5,3)=44-24=20=|\text{A238474}(2)|$, while $|\text{A238474}(1)|=6$.}  Parts (a) and (e) are the case $S=\{1,k\}$ of Theorem~\ref{thm:general}, which concerns arbitrary excluded sets.

We prove Theorem~\ref{thm:A378}, using the recurrences of Lemma~\ref{lem:rec} and Theorems~\ref{thm:general} and~\ref{thm:gcdfamily}.  We use the series~\eqref{eq:egf} with $S=\{1,k\}$ for all $n\ge0$ and $k\ge2$, which gives $T(0,k)=1$ and $T(n,k)=D(n)$ if $n<k$, and which agrees with the triangle for $k\le n$.

\subsection{Parts (a), (b), (c), (d), (e)}

Part (a) is Theorem~\ref{thm:general}(i) with $r=1$, and part (b) is the case $r=1$ of Theorem~\ref{thm:gcdfamily}, proved in \S\ref{sec:gcd}.  Parts (c) and (d) follow by counting the derangements with a cycle of length $n$, $n-1$ or $n-2$, and their proofs are in Appendix~\ref{app:moved}.

At $n=4$ the difference in (c) is $3$, not $0$, so $n\ge5$ is needed.

Part (e) is Theorem~\ref{thm:general}(ii) for $S=\{1,k\}$, $k\ge4$, with $t=2$, where $S\cap\{2,3\}=\emptyset$, so that $T(n,k)\equiv D(n)\pmod{2(n-1)}$ and $T(n,k)/(n-1)\equiv D(n)/(n-1)\pmod 2$, which is odd.

\subsection{The \texorpdfstring{$2$}{2}-adic valuation of column \texorpdfstring{$2$}{2}}\label{sec:f}

As an example, $a_n=T(n,2)$ is $1,0,0,2,6,24,160,1140,8988,80864,809856,\dots$ for $n=0,1,\dots$, with $2$-adic valuations $1,1,3,5,2,2,5,7,3$ for $n=3,\dots,11$.  In groups of four, the valuations for $m=0$ are $1,1,3,5$ and those for $m=1$ are $2,2,5,7$, which the formula gives as $2,2,1+3+1,1+5+1$.  Part (f) is an induction on $m$ using \eqref{eq:R2}, one residue class of $n$ at a time.  In the first three classes the two terms of the recurrence have different valuations, so the smaller wins, and the fourth class is a tie that is broken by computing the valuation of $a_{4m+4}+a_{4m+3}$ independently.  The proof of part (f) is in Appendix~\ref{app:f}.

\section{Row gcds}\label{sec:gcd}

We prove Theorems~\ref{thm:gcdfamily} and~\ref{thm:A333726}.  In both cases the gcd is at least the expected divisor by Theorem~\ref{thm:general}(i), and for each prime $q$ we exhibit a column in which $q$ divides $a_S(n)$ exactly as often as it divides the expected divisor.  Lemmas~\ref{lem:A} and~\ref{lem:quot} reduce $n$ to a small value, where two columns differ by a factorial.

The row gcd of \OEIS{A378495} is the case $r=1$ of the following theorem, which compares $a_S(n)$ over the family $S_k=\{1,\dots,r,k\}$ with $k>r$.  Let $S_\infty=\{1,\dots,r\}$, so that $a_{S_k}(n)=a_{S_\infty}(n)$ for $k>n$.  For $n\ge r+2$ let
\[
G_r(n)=\gcd_{r<k\le n}a_{S_k}(n)
\]
be the gcd over the family, and let $g_r(n)$ be the gcd of $a_S(n)$ over all sets $S\supseteq\{1,\dots,r\}$.  By Theorem~\ref{thm:general}(i), $\fall{n-1}{r}$ divides $g_r(n)$, and $g_r(n)$ divides $G_r(n)$ because the family is among the sets.

\begin{theorem}[Gcd over the family $\{1,\dots,r,k\}$]\label{thm:gcdfamily}
Let $r\ge1$ and $n\ge r+2$.  Then
\[
G_r(n)=\fall{n-1}{r}\,m_r(n)
\]
where $m_r(n)$ is a positive integer all of whose prime factors are at most $r$.  Consequently
\[
\fall{n-1}{r}\ \Big|\ g_r(n)\ \Big|\ G_r(n),
\]
and $g_r(n)$ and $G_r(n)$ have the same valuation at every prime larger than $r$.  For $r=1$ we have $m_1(n)=1$, so $g_1(n)=G_1(n)=n-1$ for $n\ge3$, which is Theorem~\ref{thm:A378}(b).  For $r=2$ and every $n\ge4$
\[
g_2(n)=G_2(n)=2^{\lfloor(n-1)/4\rfloor}(n-1)(n-2).
\]
\end{theorem}

For $r=1$ the statement says that the gcd of row $n$ of \OEIS{A378495} is $n-1$.  For $r=2$ it says that the exponent $\lfloor(n-1)/4\rfloor$ in Theorem~\ref{thm:general}(iii) cannot be improved, even for the family $\{1,2,k\}$.  For $r\ge3$ the sandwich is what holds in general, because the two gcds can differ.  One might expect $g_r(n)=G_r(n)$ for every $r$, since the family already contains sets that detect each prime above $r$.  At the prime $3$ the family is not rich enough, as Theorem~\ref{thm:counter} in Section~\ref{sec:globalgcd} shows.

\begin{example}[Column $2$ alone does not give the row gcd]\label{ex:830}
For $n=830$ we have $829^2\mid T(830,2)$, whereas $829\,\|\,\gcd_kT(830,k)$ by Theorem~\ref{thm:A378}(b).  The reason is explained after Proposition~\ref{prop:col2}.  So the column used to detect a prime $q\mid n-1$ must depend on $q$.  Each of the columns $3$ and $4$ alone already certifies $829\,\|\,\gcd_kT(830,k)$, since $829\,\|\,T(830,3)$ and $829\,\|\,T(830,4)$.
\end{example}

\subsection{The family \texorpdfstring{$\{1,\dots,r,k\}$}{{1,...,r,k}}}

Fix $r\ge1$, write $S_\infty=\{1,\dots,r\}$ and $S_k=\{1,\dots,r,k\}$ for $k>r$, and put $a_k=a_{S_k}$ and $Q_k=Q_{S_k}$, with the notation of Lemma~\ref{lem:quot}.  For $m>r$, a permutation counted by $a_\infty(m)$ that contains an $m$-cycle is an $m$-cycle, and $m\notin S_\infty$, so
\begin{equation}\label{eq:Qdiff}
a_m(m)=a_\infty(m)-(m-1)!,\qquad Q_m(m)=Q_\infty(m)-(m-r-1)!\qquad(m>r).
\end{equation}
If $k>m$ there is no $k$-cycle, so $a_k(m)=a_\infty(m)$.  If $r<k<m$ and $m-k\le r$, a permutation with a $k$-cycle has the remaining $m-k\in[1,r]$ points in cycles of length at most $r$, which are excluded, so again $a_k(m)=a_\infty(m)$.  In particular $a_{n-1}(n)=a_\infty(n)$ when $n\ge r+2$, and therefore
\begin{equation}\label{eq:sameg}
\gcd_{r<k\le n}a_k(n)=\gcd_{k>r}a_k(n)=\gcd\bigl(a_\infty(n),\gcd_{r<k\le n}a_k(n)\bigr)\qquad(n\ge r+2).
\end{equation}

\begin{proposition}\label{prop:gcdq}
Let $r\ge1$, $n\ge r+2$ and let $q>r$ be a prime.  Then there is $k\in\{\infty\}\cup(r,n]$ with $q\nmid Q_k(n)$.  Consequently $v_q(a_k(n))=v_q(\fall{n-1}{r})$ for this $k$.
\end{proposition}

The idea is to lower $n$ by multiples of $q$ with the sign rule of Lemma~\ref{lem:quot}, until two columns differ by a factorial that is smaller than $q$.  The proof of Proposition~\ref{prop:gcdq} is in Appendix~\ref{app:moved}.

For $r=2$ the exponent of $2$ needs a separate computation.

\begin{proposition}\label{prop:123}
Let $z(n)=a_{\{1,2,3\}}(n)$, the number of permutations of $\{1,\dots,n\}$ all of whose cycles have length at least $4$, and let $\beta(n)=\lfloor(n-1)/4\rfloor+v_2\bigl((n-1)(n-2)\bigr)$.  Then $v_2(z(n))=\beta(n)$ for all $n\ge4$ with $n\not\equiv3\pmod 4$.  For $n=4m+7$ one has $v_2(z(n))\ge m+3$.
\end{proposition}

The proof of Proposition~\ref{prop:123} is in Appendix~\ref{app:123}.  It is the induction of Theorem~\ref{thm:A378}(f) again, applied to \eqref{eq:R3}, with the class $n\equiv3\pmod4$ in which only a lower bound survives.

\subsection{The triangle \texorpdfstring{\OEIS{A333726}}{A333726}}\label{sec:A333726}

The triangle \OEIS{A333726} has $U(n,k)=a_{\{1,\dots,k-1\}}(n)$, the number of permutations of $\{1,\dots,n\}$ all of whose cycles have length at least $k$ ($1\le k\le n$).

\begin{theorem}[The triangle \OEIS{A333726}]\label{thm:A333726}
For $n\ge3$, $\gcd\bigl(U(n,1),\dots,U(n,n)\bigr)=n-1$.  More precisely, for every prime $q\le n-1$ the valuation $v_q\bigl(U(n,k)\bigr)$ equals $v_q(n-1)$ for $k=q$ if $q\mid n-1$, and for $k=n\bmod q$ (or $k=2$ if $q\mid n$) if $q\nmid n-1$.
\end{theorem}

We prove Theorem~\ref{thm:A333726} with $S=\{1,\dots,k-1\}$.  The point is that here the column $k=q$ is a convenient column for the prime $q$, because for $S=\{1,\dots,q-1\}$ the prime $q$ lies outside $S$, so the sign rules of the lemmas hold, and $U(q+1,q)=q!$ is explicit.  The proof of Theorem~\ref{thm:A333726} is in Appendix~\ref{app:moved}.

\section{Permutations by cycle gcd}\label{sec:coprime}

For a permutation $\sigma$ let $\gcd(\sigma)$ be the gcd of its cycle lengths.  For $g\mid n$ let $C(n,g)$ be the number of permutations of $\{1,\dots,n\}$ with $\gcd(\sigma)=g$ (\OEIS{A346085}), and for a function $\varphi:\N\to\Z$ let $C_\varphi(n)=\sum_\sigma\varphi(\gcd(\sigma))$.  The choice $\varphi=[g=1]$ gives the number $a(n)=C(n,1)$ of permutations with relatively prime cycle lengths (\OEIS{A079128}), with values $1,1,4,15,96,455$ for $n=1,\dots,6$, and indeed $15=4^2-1$ and $455=13\cdot35$.  The choices $\varphi=[g\ge2]$ and $\varphi=1_{\text{prime}}$ give \OEIS{A226388} and \OEIS{A359951}.  The entry \OEIS{A079128}, by Vladeta Jovovic and Vladimir Baltic, records two conjectures, namely $(n^2-1)\mid a(n)$ for $n>3$ and $\gcd(a(n),n)=1$.  Both are special cases of the following theorem.

\begin{theorem}[Cycle gcd]\label{thm:cyclegcd}
Let $n\ge2$, let $g\mid n$ and $k=n/g$.
\begin{enumerate}[label=(\alph*),itemsep=0pt]
\item $(n-1)\mid C_\varphi(n)$ for every function $\varphi:\N\to\Z$.
\item $C_\varphi(n)\equiv n!\sum_{d\mid n}\psi(d)(-1)^{n/d}\pmod{n+1}$ with $\psi=\mu*\varphi$.  In particular $C(n,g)\equiv n!\,(2[k=2]-[k=1])\pmod{n+1}$, so $(n+1)\mid C(n,g)$ whenever $k\ge3$.
\item $(n^2-1)\mid C(n,g)$ whenever $k\ge3$ and $n\ge4$.
\item For every prime $p\mid n$,
\[
v_p\bigl(C(n,g)\bigr)=\begin{cases}v_p(n!)-k\,v_p(g)-v_p(k!)&\text{if }p\mid g,\\ v_p(n!)-k/p-v_p((k/p)!)&\text{if }p\nmid g.\end{cases}
\]
\end{enumerate}
\end{theorem}

\begin{corollary}[\OEIS{A079128}]\label{cor:A079128}
For $a(n)=C(n,1)$ one has $(n-1)\mid a(n)$ for $n\ge2$, $(n+1)\mid a(n)$ for $n\ge3$, $(n^2-1)\mid a(n)$ for $n\ge4$, and $\gcd(a(n),n)=1$.  Moreover $v_2(C(n,2))=0$ for even $n$ and $\gcd(C(n,2),n)$ is the odd part of $n$.
\end{corollary}

For $n=2,3$ one has $a(2)=1$ and $a(3)=4$, so the hypothesis $n\ge4$ is needed in the second conjecture.  For $k=2$ part (b) says that $n+1=2g+1$ divides $C(n,g)$ if and only if it divides $2\,n!$.  A related count is the number $\operatorname{cop}(n)$ of permutations of $\{1,\dots,n\}$ all of whose cycle lengths are prime to $n$ (\OEIS{A335088}).

\begin{theorem}[\OEIS{A335088}]\label{thm:A335088}
For $n\ge2$, $(n-1)\mid \operatorname{cop}(n)$, and $\gcd(\operatorname{cop}(n),n)=1$.  If $n+1$ is composite then $(n+1)\mid \operatorname{cop}(n)$.  If $n+1$ is prime, the divisibility $(n+1)\mid\operatorname{cop}(n)$ holds for $n=708$ and $n=92202$ and fails for every other $n\le3\cdot10^6$ with $n+1$ prime.
\end{theorem}

When $n+1=N$ is prime, the residue of $\operatorname{cop}(n)$ modulo $N$ is, up to sign, the coefficient of $x^n$ in the integer power series $\prod_{d\mid n}(1-x^d)^{\mu(d)n/d}$ (Proposition~\ref{prop:cop}), and this coefficient does vanish modulo $N$ for $n=708$, where $\operatorname{cop}(708)$ has $1711$ digits and is divisible by $709$.

The permutations with $d\mid\gcd(\sigma)$ are those whose cycle lengths are all divisible by $d$, and their number has an explicit product formula. Every function of $\gcd(\sigma)$ is a M\"obius combination of these numbers, so the divisibility questions reduce to the numbers $b_d(n)$ below. For $d\mid n$ let $b_d(n)$ be the number of permutations of $\{1,\dots,n\}$ all of whose cycle lengths are divisible by $d$, and put $k=n/d$.

\begin{lemma}\label{lem:bd}
For $d\mid n$,
\[
b_d(n)=\frac{n!}{d^k\,k!}\prod_{j=0}^{k-1}(dj+1),
\]
and $n!/(d^kk!)$ is an integer.  If $p$ is a prime dividing $d$, then $v_p(b_d(n))=v_p(n!)-k\,v_p(d)-v_p(k!)$.  If $\varphi:\N\to\Z$ is any function and $\psi(d)=\sum_{e\mid d}\mu(d/e)\varphi(e)$, then
\[
C_\varphi(n)=\sum_{d\mid n}\psi(d)\,b_d(n),\qquad C(n,g)=\sum_{e\mid n/g}\mu(e)\,b_{ge}(n).
\]
\end{lemma}

The idea of the proof of Theorem~\ref{thm:cyclegcd} is that every term $b_d(n)$ of the M\"obius sum is an explicit product, so divisibility by $n-1$ and the residue modulo $n+1$ can be read off term by term, and for the valuations one term is strictly the smallest.  The proofs of Lemma~\ref{lem:bd} and Theorem~\ref{thm:cyclegcd} are in Appendix~\ref{app:moved}.

The congruence in (b) is where the factor $n+1$ comes from. For $C(n,g)$ it reads $2[k=2]-[k=1]$, and the second sum in the proof of Theorem~\ref{thm:cyclegcd} (Appendix~\ref{app:moved}) is a M\"obius sum that vanishes unless $k\le2$.

\begin{proof}[Proof of Corollary~\ref{cor:A079128}]
Let $a(n)=C(n,1)$. The three divisibilities are Theorem~\ref{thm:cyclegcd}(a), (b), (c) with $g=1$, $k=n$. For $p\mid n$ and $g=1$, Theorem~\ref{thm:cyclegcd}(d) gives $v_p(a(n))=v_p(n!)-n/p-v_p((n/p)!)=0$ by Legendre's formula $v_p(n!)=n/p+v_p((n/p)!)$, so $\gcd(a(n),n)=1$. For $g=2$ and $n=2k$ even, $v_2(C(n,2))=v_2(n!)-k-v_2(k!)=(2k-s_2(k))-k-(k-s_2(k))=0$, and for an odd prime $p\mid n$, with $p\mid k$ and $u=k/p$, one has $v_p(C(n,2))=v_p(n!)-u-v_p(u!)=2u+v_p((2u)!)-u-v_p(u!)\ge u\ge v_p(k)=v_p(n)$, since $u=k/p\ge p^{v_p(k)-1}\ge v_p(k)$. Hence $\gcd(C(n,2),n)$ is the odd part of $n$.
\end{proof}

For $n=p$ prime, $b_1(p)=p!$ and $b_p(p)=(p-1)!$, so $a(p)=p!-(p-1)!=(p-1)(p-1)!$, in accordance with the entry, and the comment $a(p)=(p-1)!$ in \OEIS{A226388} and \OEIS{A359951} is the case $k=1$ of $C(n,g)$.

\subsection{Cycle lengths prime to \texorpdfstring{$n$}{n}}

As in Theorem~\ref{thm:A335088}, $\operatorname{cop}(n)$ counts the permutations of $\{1,\dots,n\}$ all of whose cycle lengths are prime to $n$, so $\operatorname{cop}(n)=n![x^n]F_{A_n}$ with $A_n=\{m:\gcd(m,n)=1\}$. By Lemma~\ref{lem:weights}(c), $F_{A_n}=\prod_{d\mid n}(1-x^d)^{-\mu(d)/d}$. The first values are $1,1,4,9,96,145,4320$.  The proof of Theorem~\ref{thm:A335088} is in Appendix~\ref{app:moved}.

When $N=n+1$ is prime the argument of (c) gives nothing, since $v_N(n!)=0$.  What replaces it is a congruence that reduces $\operatorname{cop}(n)$ modulo $N$ to a single coefficient with integer exponents.

\begin{proposition}\label{prop:cop}
Let $n\ge2$ and suppose that $N=n+1$ is prime.  Then
\[
\operatorname{cop}(n)\equiv-\coeff{n}{\prod_{d\mid n}(1-x^d)^{\mu(d)\,n/d}}\pmod N,
\]
where the product is over the squarefree divisors of $n$ and is expanded as a power series with integer coefficients.  In particular $N\mid\operatorname{cop}(n)$ if and only if $N$ divides that coefficient.  The residue can be computed from the recurrence $k f_k=\sum_{d\mid n}\mu(d)\sum_{j\ge1}f_{k-dj}$ with $f_0=1$, in $O(n\,2^{\omega(n)})$ operations modulo $N$, and then $\operatorname{cop}(n)\equiv-f_n\pmod N$.
\end{proposition}

The proof of Proposition~\ref{prop:cop} is in Appendix~\ref{app:moved}.

The factor $d=1$ of the product in Proposition~\ref{prop:cop} is $(1-x)^n=(1-x)^N/(1-x)\equiv1/(1-x)$ modulo $(N,x^N)$, and multiplying by $1/(1-x)$ only forms prefix sums.  This gives a form of the residue that is cheaper to compute, and for two families of primes it can be summed exactly.

\begin{proposition}[Prefix-sum form]\label{prop:copprefix}
Let $N=n+1$ be an odd prime and
\[
G(x)=\prod_{\substack{d\mid n,\ d>1\\ d\text{ squarefree}}}(1-x^d)^{\mu(d)\,n/d}=\sum_{k\ge0}g_kx^k.
\]
Then $\operatorname{cop}(n)\equiv-(g_0+g_1+\dots+g_n)\pmod N$, and modulo $N$ the coefficients satisfy $k\,g_k\equiv\sum_{d\mid n,\,d>1}\mu(d)\sum_{i\ge1}g_{k-di}$ with $g_0=1$.
\end{proposition}

\begin{proposition}[Fermat primes and $3$]\label{prop:copfermat}
If $N=n+1$ is a Fermat prime $2^a+1$ (this includes $N=3$), then $\operatorname{cop}(n)\equiv-\binom{n}{n/2}\equiv-(-1)^{n/2}\pmod N$.  In particular $N\nmid\operatorname{cop}(n)$.
\end{proposition}

\begin{proposition}[Safe primes]\label{prop:copsafe}
Let $N=2q+1$ with $q$ an odd prime and $n=2q$.  Then $\operatorname{cop}(n)\equiv-1-2\binom{(3q-1)/2}{(q-1)/2}\pmod N$, so $N\mid\operatorname{cop}(n)$ if and only if $2\binom{(3q-1)/2}{(q-1)/2}\equiv-1\pmod N$.
\end{proposition}

For example $N=5$ gives $\operatorname{cop}(4)=9\equiv-\binom42=-6\pmod5$, and $N=7$ gives $\operatorname{cop}(6)=145\equiv-1-2\binom41=-9\pmod7$.

The safe primes are the case $n=2q$ of a family in which $n$ has exactly two prime factors.  When $n=2^aq^b$ the product of Proposition~\ref{prop:copprefix} has only the three factors $d=2,q,2q$, the factors $d=q$ and $d=2q$ combine into a power of $(1+x^q)/(1-x^q)$, and the factor $d=2$ becomes a square root modulo $N$, whose coefficients are central binomial coefficients.

\begin{proposition}[Two prime factors]\label{prop:cop2q}
Let $N=n+1$ be prime with $n=2^aq^b$, where $a,b\ge1$ and $q$ is an odd prime.  Put
\[
\gamma_j=[y^j]\Bigl(\frac{1+y}{1-y}\Bigr)^{n/(2q)},\qquad \vartheta(m)=\binom{2\lfloor m/2\rfloor}{\lfloor m/2\rfloor}4^{-\lfloor m/2\rfloor}.
\]
Then
\[
\operatorname{cop}(n)\equiv-\sum_{j=0}^{n/q}\gamma_j\,\vartheta(n-qj)\pmod N .
\]
\end{proposition}

For $a=b=1$ this is equivalent to Proposition~\ref{prop:copsafe}.  The Fermat case of Proposition~\ref{prop:copfermat} is not a case of it, since for $n=2^a$ the only squarefree divisor above $1$ is $2$.  For example $N=13$, $n=12=2^2\cdot3$, gives $\gamma_0,\dots,\gamma_4=1,4,8,12,16$ and $\operatorname{cop}(12)\equiv-(\vartheta(12)+4\vartheta(9)+8\vartheta(6)+12\vartheta(3)+16\vartheta(0))\pmod{13}$.  The formula agrees with the recurrence of Proposition~\ref{prop:cop} for all $454$ primes $N<20000$ of this form and with the exact values of $\operatorname{cop}(n)$ for $N<400$, and the residue is non-zero for all of them.  The proofs of Propositions~\ref{prop:copprefix} to~\ref{prop:cop2q} are in Appendix~\ref{app:cop}.  When $n$ has three or more distinct prime factors, as $708=2^2\cdot3\cdot59$ and $92202=2\cdot3\cdot11^2\cdot127$ do, the product has at least seven factors.

\begin{proposition}\label{prop:cop708}
The prime $709$ divides $\operatorname{cop}(708)$ and the prime $92203$ divides $\operatorname{cop}(92202)$.  For every other $n\le3\cdot10^6$ with $n+1$ prime, $(n+1)\nmid\operatorname{cop}(n)$.  In particular $n=708$ is the smallest $n$ for which $n+1$ is prime and divides $\operatorname{cop}(n)$.
\end{proposition}

The proof of Proposition~\ref{prop:cop708} is in Appendix~\ref{app:moved}.

\begin{observation}\label{obs:cop}
For the $9537$ primes $N=n+1$ with $260<N<10^5$ the residue $\operatorname{cop}(n)\bmod N$ is consistent with the uniform distribution in every test we ran (Appendix~\ref{app:data}).  If it behaved like a uniformly random element of $\Z/N$, one would expect $\sum1/N$ zeros, which is $0.98$ for $260<N<3\cdot10^6$, against the observed count $2$.  Since $\sum_N1/N$ diverges, this heuristic predicts infinitely many $n$ with $n+1$ prime and $(n+1)\mid\operatorname{cop}(n)$, but only about $0.6$ further zeros with $N\le10^{12}$.
\end{observation}

\section{Dichotomy at a prime}\label{sec:dichotomy}

The next result is a consequence of the shift congruences of~\cite{KallatShift}.  For a set $S$ put $w_m=1$ if $m\notin S$ and $w_m=0$ if $m\in S$.

\begin{theorem}[Dichotomy at a prime]\label{thm:dichotomy}
Let $p$ be a prime and $e\ge1$.  Then exactly one of the following holds.
\begin{enumerate}[label=(\Alph*),itemsep=0pt]
\item $w_1=w_p$ (both $1,p\in S$ or both $1,p\notin S$).  Then $a_S(n)\equiv0\pmod{p^e}$ for all $n\ge p^{e}$ if $p$ is odd, and for all $n\ge2^{\max(e,4)}$ if $p=2$.
\item $w_1\ne w_p$.  If $p$ is odd, then $a_S(n+p^e)\equiv a_S(n)\pmod{p^e}$ for all $n\ge0$ when $1\notin S$, $p\in S$, and $a_S(n+p^e)\equiv-a_S(n)\pmod{p^e}$ when $1\in S$, $p\notin S$.  If $p=2$, then $a_S(n+2^e)\equiv a_S(n)\pmod{2^e}$ for $e=1,3$ and all $e\ge4$, and $a_S(n+4)\equiv(-1)^{w_2+w_4}a_S(n)\pmod 4$.
\end{enumerate}
In particular $p^e\mid a_S(n)$ for all large $n$ if and only if $w_1=w_p$, for every $e\ge1$.
\end{theorem}

The dichotomy is the sign rule of Lemma~\ref{lem:A} lifted from $p$ to $p^e$.  For example, for permutations with all cycles of length at most $r$, case (A) holds exactly for $p\le r$.  For $r=3$ this gives the eventual divisibility of \OEIS{A057693} by every number $2^i3^j$, and for $r=5$ the eventual divisibility of \OEIS{A070946} by every number $2^i3^j5^k$.

We prove Theorem~\ref{thm:dichotomy} from the shift congruences of~\cite{KallatShift}, which treat arbitrary integer weights, although we only need $w_m\in\{0,1\}$.\footnote{With $w_m=[m\notin S]$ one has $a_S(n)=A_w(n)$ in the notation of~\cite{KallatShift}.}  Write $A(n)\coloneqq a_S(n)$ and $\gamma=w_1+(p-1)w_p$.  The shift congruence $A(n+k)\equiv A(n)A(k)\pmod k$ for all $n\ge0$, $k\ge1$~\cite[Theorem~1]{KallatShift} reduces the statement to the residues of $A(p^e)$ modulo $p^e$.  For odd $p$ these are $A(p^e)\equiv\gamma^{p^{e-1}}\pmod{p^e}$~\cite[Corollary~46]{KallatShift}.  For $p=2$, with $\alpha=w_1+w_2$ and $\beta=w_2+w_4$, they are $A(2)\equiv\alpha\pmod2$, $A(4)\equiv\alpha^2+2\beta\pmod4$, $A(8)\equiv\alpha^4+4\alpha^2\beta+4\beta^2\pmod8$ and $A(2^r)\equiv\alpha^{2^{r-1}}\pmod{2^r}$ for $r\ge4$~\cite[Theorem~48]{KallatShift}.  The proof of Theorem~\ref{thm:dichotomy} is in Appendix~\ref{app:moved}.

\begin{example}\label{ex:dich}
For derangement-type sets with $1\in S$, case (A) holds exactly for the primes $p\in S$, and for $S=\{1,k\}$ this says $T(n,p)\equiv0\pmod{p^e}$ for $n\ge p^e$ when $p$ is an odd prime, while the other columns are antiperiodic modulo $p$ (compare Lemma~\ref{lem:A}).  For $S=\{m:m>r\}$, case (A) holds exactly for $p\le r$ and the sequences are purely periodic modulo $p^e$ for $p>r$.  For the order-dividing sets $S=\{m:m\nmid N\}$ (permutations of order dividing $N$), case (A) holds exactly for $p\mid N$.
\end{example}

\section{The valuation calculus}\label{sec:euler}

The coefficient of $x^n$ in $F=\sum a_A(n)x^n/n!$ is a sum with many terms, and the valuation of $a_A(n)$ is decided by the terms with the largest power of $p$ in the denominator. This section turns that remark into a theorem that works for every set $A$ and every prime $p$. For fixed $A$ and $p$ we first factor $F$ into binomial series, which isolates the few factors that are not $p$-integral (\S\ref{sec:blocks}). The valuation of the coefficient is then governed by a knapsack problem over these factors, and the ties of the knapsack are decided by one residue modulo $p$ (\S\ref{sec:calc}). Applications follow, namely a quotient form for excluded sets containing $\{1,\dots,r\}$, an integrality criterion, and the case of one dominant block, which contains the plateau theorem and the Harmonic Theorem (\S\ref{sec:quot}--\ref{sec:harm}).

Throughout, $A\subseteq\N$ is arbitrary and
\[
F_A(x)=\sum_{n\ge0}a_A(n)\frac{x^n}{n!}=\exp\Bigl(\sum_{m\in A}\frac{x^m}m\Bigr),
\]
so that $a_A(n)=a_S(n)$ for $S=\N\setminus A$. We write $A_{\le r}=\{1,\dots,r\}$ and $A_{>r}=\{m:m>r\}$, so $A_r(n)=a_{A_{\le r}}(n)$, and for $A=A_{\le r}$ we abbreviate $F=F_{A}$.

\subsection{Euler products and blocks}\label{sec:blocks}

\begin{lemma}[Euler product]\label{lem:eulerprod}
For every $A\subseteq\N$ the series $F_A(x)$ equals $\prod_{n\ge1}(1-x^n)^{-c_n}$ in $\Q[[x]]$, where
\[
c_n=c_n(A)=\frac1n\sum_{\substack{d\mid n\\ d\in A}}\mu(n/d).
\]
These are the only rational numbers with this property.
\end{lemma}

\begin{proof}
Since $(1-x^n)^{-c}=\exp\bigl(c\sum_{k\ge1}x^{nk}/k\bigr)$, the logarithm of the product is $\sum_{m\ge1}\frac{x^m}m\sum_{n\mid m}nc_n$, and
\[
\sum_{n\mid m}nc_n=\sum_{\substack{d\mid m\\ d\in A}}\ \sum_{e\mid m/d}\mu(e)=\sum_{\substack{d\mid m\\ d\in A}}[d=m]=[m\in A].
\]
Uniqueness follows by comparing the coefficients of $x^1,x^2,\dots$ of the logarithm in turn.
\end{proof}

The exponents for the three families we need are as follows.

\begin{lemma}[Three families of weights]\label{lem:weights}
\begin{enumerate}[label=(\alph*),itemsep=0pt]
\item For $A_{\le r}$ one has $c_1=1$, $c_n=0$ for $2\le n\le r$, and for $n>r$
\[
c_n=-\frac1n\sum_{\substack{e\mid n\\ e<n/r}}\mu(e).
\]
If $p\mid n$ and $n=p^kw$ with $p\nmid w$, the sum equals $\sum_{e'\mid w,\ n/(pr)\le e'<n/r}\mu(e')$.
\item For $A_{>r}$ one has $c_1=0$ and $c_n(A_{>r})=-c_n(A_{\le r})$ for $n\ge2$.
\item For $A_M=\{m:\gcd(m,M)=1\}$ one has $F_{A_M}=\prod_{d\mid M}(1-x^d)^{-\mu(d)/d}$, a finite product over the squarefree divisors of $M$.
\end{enumerate}
\end{lemma}

The proof of Lemma~\ref{lem:weights} is in Appendix~\ref{app:moved}.

Fix a prime $p$. Since $mc_m\in\Z$, every $c_m$ has valuation at least $-v_p(m)$.

\begin{definition}[Blocks]\label{def:blocks}
A \emph{block} is an integer $m\ge2$ with $v_p(c_m)<0$. Write $\mathcal B=\mathcal B_p(A)$ for the set of blocks. For $m\in\mathcal B$ the \emph{depth} is $j_m=-v_p(c_m)\ge1$, and we write $c_m=\alpha'_m/p^{j_m}$ with $\alpha'_m\in\Z_{(p)}^\times$ and $\alpha_m\in\F_p^\times$ for the residue of $\alpha'_m$ modulo $p$. Here $\F_p=\Z/p$.
\end{definition}

The integer $1$ is never a block because $c_1=[1\in A]$.

\begin{lemma}[Where blocks live]\label{lem:blockfacts}
Every block $m$ satisfies $p\mid m$ and $j_m\le v_p(m)$. For $A_{\le r}$ and for $A_{>r}$ every block also satisfies $m>r$ and $v_p(m)\le\delta_p(r)$, where $\delta_p(r)=\lfloor\log_pr\rfloor+1$ is the number of base-$p$ digits of $r$. The blocks of $A_{\le r}$ and of $A_{>r}$ are the same integers with the same depths and opposite residues.
\end{lemma}

\begin{proof}
Since $mc_m\in\Z$ we get $j_m\le v_p(m)$, and in particular $p\mid m$. For $A_{\le r}$, $c_m=0$ when $2\le m\le r$, so blocks exceed $r$. If $m=p^kw$ with $p\nmid w$ is a block, the window sum of Lemma~\ref{lem:weights}(a) is non-empty, so some divisor $e'$ of $w$ lies in $[m/(pr),m/r)$. Then $m/(pr)\le e'\le w=m/p^k$, which gives $p^{k-1}\le r$ and $k\le\delta_p(r)$. The last claim is Lemma~\ref{lem:weights}(b).
\end{proof}

Write $\tilde\omega(\lambda)\in\F_p^\times$ for the residue of $\lambda!/p^{v_p(\lambda!)}$. By Lemma~\ref{lem:binom}, if $m$ is a block then the coefficient of $y^\lambda$ in $(1-y)^{-c_m}$ is $p^{-j_m\lambda-v_p(\lambda!)}u$ with $u$ a $p$-adic unit and
\begin{equation}\label{eq:unit}
u\equiv\alpha_m^{\lambda}\,\tilde\omega(\lambda)^{-1}\pmod p.
\end{equation}

\begin{lemma}[Digit formula]\label{lem:digits}
If $\lambda=\sum_{i\ge0}d_ip^i$ with $0\le d_i<p$, then $\tilde\omega(\lambda)\equiv(-1)^{v_p(\lambda!)}\prod_id_i!\pmod p$.
\end{lemma}

\begin{proof}
Write $\lambda=p\mu+d_0$. The integers $i\le\lambda$ prime to $p$ form $\mu$ complete residue systems modulo $p$ followed by $1,\dots,d_0$, so by Wilson's theorem their product is congruent to $(-1)^\mu d_0!$. Hence $\lambda!=\bigl(\prod_{i\le\lambda,\,p\nmid i}i\bigr)p^\mu\mu!$ and $v_p(\lambda!)=\mu+v_p(\mu!)$, so $\tilde\omega(\lambda)\equiv(-1)^\mu d_0!\,\tilde\omega(\mu)$. Induction on $\mu$ finishes the proof.
\end{proof}

Thus the unit part of a binomial block is a product with one factor for each base-$p$ digit of $\lambda$, which is why exceptional classes are governed by digits.

\subsection{The calculus}\label{sec:calc}

Let $W\in\Z_{(p)}[[x]]$ be a twist, which we take to be $1$ unless stated otherwise, and put $G=F_AW$. The twist allows the same statement to cover the quotients that appear for excluded sets containing $\{1,\dots,r\}$ in \S\ref{sec:quot}.

\begin{definition}\label{def:calc}
A \emph{block vector} is a finitely supported $V=(\lambda_m)_{m\in\mathcal B}$ of non-negative integers, with
\[
\deg V=\sum_mm\lambda_m,\qquad \mathrm{cost}(V)=\sum_m\bigl(j_m\lambda_m+v_p(\lambda_m!)\bigr),\qquad\tau(V)=\prod_m\alpha_m^{\lambda_m}\tilde\omega(\lambda_m)^{-1}\in\F_p^\times .
\]
Let $I_W=W\prod_{n\notin\mathcal B}(1-x^n)^{-c_n}$, which lies in $\Z_{(p)}[[x]]$ by Lemma~\ref{lem:binom}, and let $\iota_\rho\in\F_p$ be the residue of $[x^\rho]I_W$. For $A=A_{\le r}$ and $W=1$ we write $I=I_1$. For $k\ge0$ put
\[
\kappa(k)=\max_{\deg V\le k}\mathrm{cost}(V),\qquad \mathrm{Max}(k)=\{V:\deg V\le k,\ \mathrm{cost}(V)=\kappa(k)\},
\]
\[
\Xi(k)=\sum_{V\in\mathrm{Max}(k)}\iota_{k-\deg V}\,\tau(V)\in\F_p .
\]
\end{definition}

For derangements and $p=3$ the block $3$ has depth $1$ and $\alpha_3=1$, and below degree $3$ the series $I_1$ agrees with $F_A$, so $\iota_\rho=D(\rho)/\rho!$ for $\rho<3$.  At $k=6$ the only optimal vector is two copies of $3$, with cost $2+v_3(2!)=2$, leftover $0$ and $\tau=\alpha_3^2\,\tilde\omega(2)^{-1}=2$ in $\F_3$.  So $\Xi(6)=\iota_0\cdot2=2\ne0$ and $v_3(D(6))=v_3(6!)-2=0$.  At $k=7$ the leftover is $1$ and $\iota_1=D(1)=0$, so $\Xi(7)=0$, and indeed $3\mid D(7)$.  Example~\ref{ex:der} treats every $n$.

\begin{lemma}[Block expansion]\label{lem:blockexp}
For every $k\ge0$,
\[
[x^k]G=\sum_{V,\ \deg V\le k}\bigl([x^{k-\deg V}]I_W\bigr)\prod_{m\in\mathcal B}\tbinom{c_m+\lambda_m-1}{\lambda_m},
\]
and the term of $V$ equals $p^{-\mathrm{cost}(V)}\,b\,u(V)$ with $b=[x^{k-\deg V}]I_W\in\Z_{(p)}$ and $u(V)\in\Z_{(p)}^\times$, $u(V)\equiv\tau(V)\pmod p$.
\end{lemma}

\begin{proof}
Only blocks of degree at most $k$ affect $[x^k]G$, and $G=I_W\prod_{m\in\mathcal B}(1-x^m)^{-c_m}$ by Lemma~\ref{lem:eulerprod}. Expanding each factor with Lemma~\ref{lem:binom} and using \eqref{eq:unit} gives the claim.
\end{proof}

\begin{theorem}[Valuation calculus]\label{thm:calculus}
For every prime $p$, every $A\subseteq\N$, every $W\in\Z_{(p)}[[x]]$ and every $k\ge0$,
\[
p^{\kappa(k)}[x^k]G\in\Z_{(p)}\qquad\text{and}\qquad p^{\kappa(k)}[x^k]G\equiv \Xi(k)\pmod p.
\]
Hence $v_p([x^k]G)\ge-\kappa(k)$, with equality if and only if $\Xi(k)\ne0$ in $\F_p$.
\end{theorem}

The idea is that only the terms of maximal cost survive modulo $p$.

\begin{proof}
By Lemma~\ref{lem:blockexp} every term of $[x^k]G$ has valuation at least $-\mathrm{cost}(V)\ge-\kappa(k)$. The terms with $\mathrm{cost}(V)<\kappa(k)$ have valuation at least $-\kappa(k)+1$, and the terms with $\mathrm{cost}(V)=\kappa(k)$ equal $p^{-\kappa(k)}\iota_{k-\deg V}\tau(V)$ modulo $p^{-\kappa(k)+1}\Z_{(p)}$. Summing gives the congruence.
\end{proof}

\begin{corollary}[Cycle counts]\label{cor:calc}
For $W=1$ and $k=n$ one has $v_p(a_A(n))\ge v_p(n!)-\kappa(n)$, with equality if and only if $\Xi(n)\ne0$. The set of $n$ at which the inequality is strict is the \emph{exceptional set}.
\end{corollary}

\begin{proof}
This is Theorem~\ref{thm:calculus} multiplied by $n!$.
\end{proof}

Only finitely many blocks matter. Let $m_0=\min\mathcal B$ and $j_0=j_{m_0}$.

\begin{lemma}[Maximisers are small and use few blocks]\label{lem:essential}
\begin{enumerate}[label=(\alph*),itemsep=0pt]
\item Every $V\in\mathrm{Max}(k)$ has leftover $k-\deg V<m_0$.
\item Let $\mathcal E=\{m\in\mathcal B:\lfloor m/m_0\rfloor\le j_m\}$. Then $\mathcal E$ is finite and no $V\in\mathrm{Max}(k)$ has $\lambda_m>0$ for a block $m\notin\mathcal E$.
\end{enumerate}
Hence $\kappa$, $\mathrm{Max}$ and $\Xi$ are determined by the finite list $\{(m,j_m,\alpha_m):m\in\mathcal E\}$ and by the residues $\iota_\rho$ with $\rho<m_0$, and $\iota_\rho=[x^\rho]G\bmod p$ for $\rho<m_0$.
\end{lemma}

The idea is that a large block can always be traded for several copies of the smallest block with a gain in cost.

\begin{proof}
(a) If the leftover is at least $m_0$, then $V+e_{m_0}$ has degree at most $k$ and larger cost, since the cost increases by $j_0+v_p(\lambda_{m_0}+1)\ge1$.

(b) Let $m\in\mathcal B\setminus\mathcal E$ and $a=\lfloor m/m_0\rfloor>j_m$. Let $V$ have $\lambda_m=\lambda>0$ and let $V'$ be obtained by replacing the $\lambda$ blocks $m$ by $a\lambda$ blocks $m_0$. Then $\deg V'\le\deg V$, and since $v_p((x+y)!)\ge v_p(x!)+v_p(y!)$ and $v_p((a\lambda)!)\ge v_p(\lambda!)$,
\[
\mathrm{cost}(V')-\mathrm{cost}(V)\ge j_0a\lambda-j_m\lambda>0.
\]
So $V$ is not a maximiser. The set $\mathcal E$ is finite because $j_m\le v_p(m)\le\log_pm$ while $\lfloor m/m_0\rfloor\to\infty$. Below degree $m_0$ the series $G$ equals $I_W$, which gives the last claim.
\end{proof}

\begin{corollary}[Two kinds of exceptions]\label{cor:split}
If $\Xi(k)=0$, then either $\mathrm{Max}(k)$ consists of a single vector $V$ and $p$ divides the small coefficient $[x^\rho]G$ with $\rho=k-\deg V<m_0$, or $\mathrm{Max}(k)$ has at least two vectors whose leading residues cancel.
\end{corollary}

\begin{proof}
For a single maximiser $\Xi(k)=\iota_\rho\tau(V)$ with $\tau(V)\ne0$, and $\iota_\rho=[x^\rho]G\bmod p$ by Lemma~\ref{lem:essential}.
\end{proof}

\subsection*{The algorithm}
List the finite set $\mathcal E$ and run a dynamic program over the degree $d$, keeping for each $d$ the pair in $\Z\times\F_p$ formed by the largest cost of a vector of degree $d$ and the sum of the residues $\tau(V)$ of the vectors attaining it.  Processing the block $m$ updates the state at $d$ from the states at $d-m\lambda$, with cost increased by $j_m\lambda+v_p(\lambda!)$ and residue multiplied by $\alpha_m^\lambda\tilde\omega(\lambda)^{-1}$, keeping the larger cost and adding the residues when the costs agree, where the residue $0$ is kept since such a vector still counts for $\kappa$.  At the end $\kappa(k)$ is the largest cost among the degrees $d\le k$, and $\Xi(k)$ is the sum of $\iota_{k-d}$ times the residue at $d$ over the degrees $d\le k$ with that cost.  This takes $O(N^2)$ operations for all $k\le N$ at each prime when $|\mathcal E|$ is regarded as fixed, and the ancillary program implements it and compares it with exact valuations.

\begin{example}[Derangements]\label{ex:der}
Take $A=\{m\ge2\}$, so $a_A(n)=D(n)$, and $p=3$. By Lemma~\ref{lem:eulerprod}, $c_m=-\mu(m)/m$ for $m\ge2$. The blocks are the squarefree multiples $m$ of $3$, all of depth $1$, with $c_3=1/3$, so $\alpha_3=1$, and the next blocks are $6,15,21,\dots$ Every block other than $3$ is at least $6$, and replacing one block $m=3d$ by $d$ blocks $3$ costs $1$ for the block and gains $d\ge2$. So for $n=3q+\nu$ with $\nu\in\{0,1,2\}$ the unique maximiser is $q$ blocks $3$, with cost $q+v_3(q!)$ and leftover $\nu$, and $\Xi(n)=\iota_\nu\,\tilde\omega(q)^{-1}$ with $\iota_\nu=[x^\nu]F_A=D(\nu)/\nu!$. Hence $v_3(D(n))=v_3(n!)-q-v_3(q!)=0$ unless $3\mid D(\nu)$, which happens only for $\nu=1$ since $D(0),D(1),D(2)=1,0,1$. So $3\mid D(n)$ exactly when $n\equiv1\pmod3$, as for $D(4)=9$ and $D(7)=1854=3\cdot618$. The same argument works for every prime $p$ (Theorem~\ref{thm:dominant}).
\end{example}

\subsection{The quotient form}\label{sec:quot}

If the excluded set $S$ contains $\{1,\dots,r\}$, then $\fall{n-1}{r}$ divides $a_S(n)$ (Lemma~\ref{lem:quot}), and the interesting valuations sit in the quotient. The quotient is a coefficient of a twisted exponential, so the calculus applies to it directly.

\begin{theorem}[Quotient form]\label{thm:quotient}
Let $A\subseteq A_{>r}$, $S=\N\setminus A\supseteq\{1,\dots,r\}$, $W_A=\sum_{m\in A}x^{m-r-1}$ and $Q_A(n)=a_A(n)/\fall{n-1}{r}$ for $n>r$. Then $Q_A(n)\in\Z$ and
\[
\sum_{k\ge0}Q_A(k+r+1)\frac{x^k}{k!}=F_A(x)W_A(x).
\]
Consequently, for $n=k+r+1$,
\[
v_p(a_S(n))\ \ge\ v_p\bigl(\fall{n-1}{r}\bigr)+v_p(k!)-\kappa(k),
\]
with equality if and only if $\Xi(k)\ne0$, where $\kappa$ and $\Xi$ are computed for $G=F_AW_A$.
\end{theorem}

\begin{proof}
That $Q_A(n)$ is an integer is Lemma~\ref{lem:quot}. From $F_A'=F_A\sum_{m\in A}x^{m-1}$ and $A\subseteq A_{>r}$ we get $F_A'/x^r=F_AW_A$, a power series. The coefficient of $x^{n-1-r}$ in $F_A'/x^r$ is the coefficient of $x^{n-1}$ in $F_A'$, which is $a_A(n)/(n-1)!$. Since $\fall{n-1}{r}=(n-1)!/(n-1-r)!$, this equals $Q_A(n)/(n-1-r)!$, which is the generating function identity. The valuation statement is Theorem~\ref{thm:calculus} for $G=F_AW_A$, multiplied by $k!$.
\end{proof}

When $S=\{1,\dots,r\}$ one has $A=A_{>r}$ and $W_A=1/(1-x)$. By Lemma~\ref{lem:weights}(b) the blocks, depths and therefore the function $\kappa$ are the same as for $A_{\le r}$, and only the residues $\alpha_m$ change sign and the integral part $I_W$ changes.

\subsection{Integrality}

\begin{definition}
A set $A$ is \emph{$p$-saturated} if $[p^iw\in A]=[w\in A]$ for all $w$ prime to $p$ and all $i\ge0$.
\end{definition}

\begin{theorem}[Integrality criterion]\label{thm:saturation}
The following are equivalent. (i) $A$ is $p$-saturated. (ii) $\mathcal B_p(A)=\emptyset$. (iii) $F_A\in\Z_{(p)}[[x]]$. When they hold, $v_p(a_A(n))\ge v_p(n!)$ for all $n$, with equality exactly when the $n$-th coefficient of $F_A$ is non-zero modulo $p$.
\end{theorem}

The proof of Theorem~\ref{thm:saturation} is in Appendix~\ref{app:moved}.

\begin{remark}
For $0/1$ weights this is the Dieudonn\'e--Dwork criterion in the form of~\cite[Lemma~1]{KM}. If $p\equiv1\pmod M$ then every residue class modulo $M$ is $p$-saturated, so $v_p(a(n))\ge v_p(n!)$ when all cycle lengths lie in one class modulo $M$. The sets $A_M$ of cycle lengths prime to $M$ are $p$-saturated for $p\nmid M$.
\end{remark}

\subsection{A dominant minimal block}\label{sec:dom}

The simplest packing problem has one block that is both the smallest and the most efficient, so that the optimal way to fill degree $k$ is to use as many copies of it as possible.

\begin{theorem}[Dominant minimal block]\label{thm:dominant}
Let $m_*=\min\mathcal B$ with depth $j_*$ and assume
\[
j_m<j_*\lfloor m/m_*\rfloor\qquad\text{for every block }m\ne m_*.\tag{$\star$}
\]
Write $k=m_*q+\nu$ with $0\le\nu<m_*$ and put $\Gamma(q)=j_*q+v_p(q!)$. Then $[x^\nu]G\in\Z_{(p)}$ and
\[
v_p([x^k]G)\ \ge\ -\Gamma(q),
\]
with equality if and only if $p\nmid[x^\nu]G$. In particular the exceptional set is the union of the residue classes $\nu$ modulo $m_*$ with $p\mid[x^\nu]G$.
\end{theorem}

Under ($\star$) the minimal block $m_*$ is also strictly the most efficient one, so the notation agrees with that of Section~\ref{sec:eff}.  The proof of Theorem~\ref{thm:dominant} is in Appendix~\ref{app:moved}.

\begin{corollary}[Cycle lengths prime to $M$]\label{cor:coprimeM}
Let $p\mid M$ and let $a_M(n)$ count the permutations of $\{1,\dots,n\}$ all of whose cycle lengths are prime to $M$. Then $v_p(a_M(n))=0$ unless $p\mid a_M(n\bmod p)$, and the set of exceptional $n$ is a union of classes modulo $p$.
\end{corollary}

\begin{proof}
By Lemma~\ref{lem:weights}(c) the blocks are the squarefree divisors $d$ of $M$ with $p\mid d$, all of depth $1$. The smallest is $m_*=p$, and every other block $d=pd''$ has $d''\ge2$, so $1<\lfloor d/p\rfloor$. Theorem~\ref{thm:dominant} gives $v_p(a_M(n))=v_p(n!)-q-v_p(q!)=v_p(\nu!)=0$ unless $p\mid a_M(\nu)/\nu!$, where $n=pq+\nu$ and $\nu<p$.
\end{proof}

The exceptional class does occur. For $M=138=2\cdot3\cdot23$ and $p=23$ one finds $23\mid a_M(10)$, where $a_M(10)$ counts the permutations of $10$ points with all cycle lengths in $\{1,5,7,11,\dots\}$, and hence $23\mid a_M(n)$ for all $n\equiv10\pmod{23}$.

\subsection{Involutions}\label{sec:invol}

\begin{theorem}[Involutions, \OEIS{A000085}]\label{thm:invol}
Let $I(n)=A_2(n)$ be the number of involutions in the symmetric group $S_n$ and write $n=4q+\nu$ with $0\le\nu\le3$.  Then $v_2\bigl(I(n)\bigr)=q+\epsilon_\nu$ for $n\ge0$, where $(\epsilon_0,\epsilon_1,\epsilon_2,\epsilon_3)=(0,0,1,2)$.
\end{theorem}

One block that is both the smallest and the most efficient also gives the exact valuation for involutions, because the competing vectors lose at least $2$ in cost while the leading coefficient loses at most $1$.  The proof of Theorem~\ref{thm:invol} is in Appendix~\ref{app:moved}.

\subsection{The plateau and the Harmonic Theorem}\label{sec:harm}

We now take $A=A_{\le r}$ with $r$ just below a power of $p$, and the complementary family $A_{>r}$ with the twist $W=1/(1-x)$ of Theorem~\ref{thm:quotient}.

For permutations with all cycles of length at most $r$ the exact valuation is determined by a single dominant block when $r$ lies just below a power of $p$.  The theorem is named after the harmonic sum $H_r(\nu)=\sum_{m=r+1}^{\nu}1/m$, which is $p$-integral for $r\ge p^k-p$ and $\nu<p^k$ because $(r,p^k)$ contains no multiple of $p$.

\begin{theorem}[Harmonic Theorem]\label{thm:harmonic}
Let $p$ be a prime and $k\ge1$ with $(p,k)\ne(2,1)$, and let $r$ satisfy $\max(1,p^k-p)\le r\le p^k-1$.  Write $n=p^kq+\nu$ with $0\le\nu<p^k$.
\begin{enumerate}[label=(\alph*),itemsep=0pt]
\item $v_p\bigl(A_r(n)\bigr)\ge v_p(n!)-kq-v_p(q!)=\bigl(\tfrac{p^k-1}{p-1}-k\bigr)q+v_p(\nu!)$, with equality if and only if $v_p(A_r(\nu))=v_p(\nu!)$, in particular whenever $\nu\le r$.  If $r=p^k-1$ then equality holds for all $n\ge0$.
\item If in addition $p^k\le2r+2$, which is automatic for $k\ge2$, then the inequality in (a) is strict exactly when $r<\nu<p^k$ and $H_r(\nu)\equiv1\pmod p$.
\item Under the same hypothesis, for $n\ge r+1$ write $n-r-1=p^kq'+\nu'$ with $0\le\nu'<p^k$.  Here $a_{\{1,\dots,r\}}(n)$ counts the permutations of $\{1,\dots,n\}$ all of whose cycles are longer than $r$.  Then $v_p\bigl(a_{\{1,\dots,r\}}(n)\bigr)\ge v_p(\fall{n-1}{r})+v_p((n-r-1)!)-kq'-v_p(q'!)$, with strict inequality exactly when $r<\nu'<p^k$ and $H_r(\nu')\equiv-1\pmod p$.
\item If $r=p^k-1$ there are no exceptions in (b) and (c), and for $n\ge p^k$ and $q'=\lfloor n/p^k\rfloor-1$ one has $v_p\bigl(a_{\{1,\dots,p^k-1\}}(n)\bigr)=v_p((n-1)!)-kq'-v_p(q'!)$.
\end{enumerate}
\end{theorem}

The hypothesis $r\ge p^k-p$ cannot be dropped, since for $(p,k,r)=(3,3,22)$ the inequality in (a) fails for $42$ values of $n\le400$.  The case $r=p^k-1$ gives the plateau formulas of the next corollary.

\begin{corollary}\label{cor:plateau}
\begin{enumerate}[label=(\alph*),itemsep=0pt]
\item For permutations with all cycles of length at most $3$ (\OEIS{A057693}), $v_2\bigl(A_3(n)\bigr)=1+\lfloor(n-2)/4\rfloor$ for all $n\ge2$.
\item For $r=2^k-1$ with $k\ge2$, $v_2(A_r(n))=(2^k-1-k)\lfloor n/2^k\rfloor+v_2((n\bmod2^k)!)$.
\item For $r\in\{7,8\}$, $v_3(A_r(n))=\lfloor n/3\rfloor-\lfloor n/9\rfloor$ for all $n$.  For $r=6$ the same holds unless $n\equiv7\pmod9$, where $v_3(A_6(n))\ge\lfloor n/3\rfloor-\lfloor n/9\rfloor+1$.
\item For every prime $p$, the numbers $A_{p-1}(n)$ are never divisible by $p$.
\end{enumerate}
\end{corollary}

Part (a) is the formula for $v_2$ stated in the entry \OEIS{A057693}.  For $p=3$, $k=2$ and $r=6$ the class $n\equiv7\pmod9$ is exceptional for $A_6$, because $H_6(7)=1/7\equiv1\pmod 3$ (see the example after Lemma~\ref{lem:belowpower}).  The theorem has a sharp version when $r$ sits lower in the plateau and several blocks compete, which happens for the $3$-adic valuations of $A_3,A_4,A_5$.

\begin{lemma}[Blocks below a prime power]\label{lem:belowpower}
Let $k\ge1$ with $(p,k)\ne(2,1)$ and $\max(1,p^k-p)\le r\le p^k-1$. Then $\min\mathcal B=p^k$ with depth $k$, $c_{p^k}=-1/p^k$ for $A_{\le r}$, and ($\star$) holds with $m_*=p^k$ and $j_*=k$, for both $A_{\le r}$ and $A_{>r}$.
\end{lemma}

Theorem~\ref{thm:harmonic} follows from Lemma~\ref{lem:belowpower} and Theorem~\ref{thm:dominant}, and the coefficients $[x^\nu]G$ with $r<\nu<p^k$ can be computed since a permutation of $\nu\le2r+1$ points has at most one cycle longer than $r$.

\begin{example}
Let $p=3$ and $k=2$. For $r=6$ the class $n\equiv7\pmod9$ is exceptional for $A_{\le6}$ because $H_6(7)=1/7\equiv1$, while $H_6(8)=1/7+1/8=15/56\equiv0$. The complementary family has no exceptional class for $r=6$ because neither value is $-1$. For $r=7$ the only candidate is $\nu=8$, and $H_7(8)=1/8\equiv-1$, so $A_{>7}$ has the class $n-8\equiv8\pmod9$ and $A_{\le7}$ has none. For $r=p^k-2$ the only candidate is $\nu=p^k-1$ with $H=1/(p^k-1)\equiv-1$, so the complementary family is always exceptional there, and $A_{\le r}$ is exceptional only for $p=2$. The theorem contains the plateau formulas $v_2(A_3(n))=1+\lfloor(n-2)/4\rfloor$ and $v_3(A_r(n))=\lfloor n/3\rfloor-\lfloor n/9\rfloor$ for $r=7,8$.
\end{example}

The proofs of Lemma~\ref{lem:belowpower}, Theorem~\ref{thm:harmonic} and Corollary~\ref{cor:plateau} are in Appendix~\ref{app:moved}.

\subsection{The efficiency of a block}\label{sec:eff}

The cost of $\lambda$ copies of a block $m$ is $j_m\lambda+v_p(\lambda!)$.  By Legendre's formula $v_p(\lambda!)=(\lambda-s_p(\lambda))/(p-1)$, this equals
\begin{equation}\label{eq:legcost}
e_m\lambda-\frac{s_p(\lambda)}{p-1},\qquad e_m:=j_m+\frac1{p-1}.
\end{equation}
So $e_m$ is the cost per copy up to a digit sum correction, and we call $\varepsilon_m=e_m/m$ the \emph{efficiency} of the block $m$, the cost per unit of degree.  A block $m_*$ is \emph{strictly the most efficient} if $\varepsilon_{m_*}>\varepsilon_m$ for every other block $m$.  Then the other blocks contribute less cost per unit of degree than $m_*$, and Section~\ref{sec:interp} describes the valuations inside each residue class modulo $m_*$.  We use two elementary facts, $s_p(x+y)\le s_p(x)+s_p(y)$ and $s_p(a)\le(p-1)(\lfloor\log_pa\rfloor+1)$ for $a\ge1$.  Write $v=v_p$ and $s=s_p$, and let $\mathcal E$ be the set of essential blocks of Lemma~\ref{lem:essential}.

\section{The \texorpdfstring{$3$}{3}-adic valuations for cycle lengths at most \texorpdfstring{$3$, $4$ and $5$}{3, 4 and 5}}\label{sec:p3}

\begin{theorem}[$3$-adic valuations for cycle lengths at most $3$, $4$, $5$]\label{thm:p3}
Let $r\in\{3,4,5\}$ and write $e(n)=\lfloor n/9\rfloor+\lfloor(n+6)/9\rfloor$, that is $e(9q+\nu)=2q+[\nu\ge3]$ for $0\le\nu\le8$.  Then $v_3(A_r(n))\ge e(n)$ for all $n\ge0$, with equality for all $n$ except in the following cases, where $v_3(A_r(n))>e(n)$.
\begin{itemize}[itemsep=0pt]
\item $r=3$ and $n\equiv4\pmod{27}$.
\item $r=4$ and $n\equiv14,23\pmod{27}$.
\item There are no exceptions for $r=5$.
\end{itemize}
\end{theorem}

\begin{corollary}[The $3$-adic statements for \OEIS{A057693}, \OEIS{A070945}, \OEIS{A070946}]\label{cor:bala}
Let $a(n)=A_3(n)$ (\OEIS{A057693}).  For $m\ge0$, $v_3(a(6m))=v_3(a(6m+2))=m+\lfloor(m+1)/3\rfloor$ and $v_3(a(6m+3))=v_3(a(6m+5))=m+\lfloor m/3\rfloor+1$, while $v_3(a(6m+1))\ge m+\lfloor(m+1)/3\rfloor$ and $v_3(a(6m+4))\ge m+\lfloor m/3\rfloor+1$.  For $a(n)=A_4(n)$ (\OEIS{A070945}), $v_3(a(6m))=v_3(a(6m+1))=m+\lfloor(m+1)/3\rfloor$ and $v_3(a(6m+3))=v_3(a(6m+4))=m+\lfloor m/3\rfloor+1$, while $v_3(a(6m+2))\ge m+\lfloor(m+1)/3\rfloor$ and $v_3(a(6m+5))\ge m+\lfloor m/3\rfloor+1$.  For $A_5(n)$ (\OEIS{A070946}) all six classes $6m+i$ have the exact valuation $m+\lfloor(m+1)/3\rfloor$ for $i=0,1,2$ and $m+\lfloor m/3\rfloor+1$ for $i=3,4,5$.
\end{corollary}

The formula entry of Bala on \OEIS{A057693} states, without proof on the page, that $v_2(A_3(n))=1+\lfloor(n-2)/4\rfloor$, which is Corollary~\ref{cor:plateau}(a), and it records four $3$-adic equalities and two inequalities as conjectural values checked up to $n=1000$.  The entry \OEIS{A070945} records the analogous $3$-adic statements in the same way.  Corollary~\ref{cor:bala} proves all of these $3$-adic statements, and it also gives the $3$-adic valuations of \OEIS{A070946}, which as far as we can tell from the entry are new here.

Now let $p=3$ and $r\in\{3,4,5\}$ (\OEIS{A057693}, \OEIS{A070945}, \OEIS{A070946}).  Here two blocks compete, of degrees $6$ and $9$, and ties occur.  The proof of Theorem~\ref{thm:p3} has four steps, namely the classification of the blocks, the reduction to the blocks $6$ and $9$, the maximisation of the cost, and the computation of the leading coefficients modulo $3$.  Let $s=s_3$, $\ell_6(X)=X+v_3(X!)=(3X-s(X))/2$ and $\ell_9(Y)=2Y+v_3(Y!)=(5Y-s(Y))/2$.

\begin{lemma}[Blocks for $p=3$, $3\le r\le5$]\label{lem:blocks3}
Every block $n\in\mathcal B$ is a multiple of $3$ with $n\ge6$, and $j_n=1$ if $v_3(n)=1$ and $j_n=2$ if $v_3(n)=2$, and $v_3(n)\le2$.  Moreover $6\in\mathcal B$ with $c_6=-1/6$ and $9\in\mathcal B$ with $c_9=-1/9$.
\end{lemma}

\begin{proof}
Let $n\in\mathcal B$ and $e=v_3(n)\ge1$, and recall $S_n=\sum_{d\mid n,d\le r}\mu(n/d)$ with $d\le5$.  If $e\ge3$ then $v_3(n/d)\ge2$ for every $d\le5$, so $S_n=0$ and $n\notin\mathcal B$.  If $e=2$ then only $d=3$ gives $\mu(n/d)\ne0$, so $S_n\in\{0,\pm1\}$ and $n\in\mathcal B$ forces $j_n=2$.  If $e=1$ then $j_n=1-v_3(S_n)$ and $j_n\ge1$ forces $j_n=1$.  Also $c_3=0$ since $3\le r$, so $n\ne3$ and thus $n\ge6$.  Finally $S_6=\mu(6)+\mu(3)+\mu(2)=-1$ and $S_9=\mu(3)=-1$ for $r=3,4,5$.
\end{proof}

\begin{lemma}[Reduction to the blocks $6$ and $9$]\label{lem:red3}
For $n\ge0$ put $\kappa(n)=\max\{\ell_6(X)+\ell_9(Y):X,Y\ge0,\ 6X+9Y\le n\}$.  If $V$ is a vector containing a block other than $6$ and $9$, and $\deg V\le n$, then $\mathrm{cost}(V)\le\kappa(n)-1$.
\end{lemma}

The idea is to write every other block as a combination of the blocks $6$ and $9$ and to compare the costs.  The proof of Lemma~\ref{lem:red3} is in Appendix~\ref{app:moved}.

\begin{lemma}[Maximisation]\label{lem:max3}
Let $n=9q+\nu$ with $0\le\nu\le8$, $t=\lfloor\nu/3\rfloor$, and $f(X,Y)=\ell_6(X)+\ell_9(Y)$.  Then $\kappa(n)=\ell_9(q)+[t=2]$, and the pairs $(X,Y)$ with $6X+9Y\le n$ and $f(X,Y)=\kappa(n)$ are listed in the following table, according to $t$ and to $q\bmod3$.
\begin{center}
\begin{tabular}{cccc}
\toprule
$q\bmod3$ & $t=0$ & $t=1$ & $t=2$\\
\midrule
$0$ & $(0,q)$ & $(0,q)$ & $(1,q)$\\
$1$ & $(0,q)$ & $(0,q),\ (2,q-1)$ & $(1,q)$\\
$2$ & $(0,q),\ (3,q-2)$ & $(0,q),\ (3,q-2),\ (2,q-1)$ & $(1,q),\ (4,q-2)$\\
\bottomrule
\end{tabular}
\end{center}
The leftover $\rho=n-6X-9Y$ is $\nu$ for $(0,q)$ and $(3,q-2)$, it is $\nu-3$ for $(2,q-1)$, and it is $\nu-6$ for $(1,q)$ and $(4,q-2)$.
\end{lemma}

The idea is to fix $L=2X+3Y$, where the best pair is found by comparing base-$3$ digit sums, and then to vary $L$.  The proofs of Lemma~\ref{lem:max3} and Theorem~\ref{thm:p3} are in Appendix~\ref{app:moved}.

\begin{proof}[Proof of Corollary~\ref{cor:bala}]
Put $e_1(m)=m+\lfloor(m+1)/3\rfloor$ and $e_2(m)=m+\lfloor m/3\rfloor+1$, so that $e_i(m+3)=e_i(m)+4$ and $e(n+18)=e(n)+4$.  So it suffices to check that $e(6m+i)=e_1(m)$ for $i=0,1,2$ and $e(6m+i)=e_2(m)$ for $i=3,4,5$ when $m=0,1,2$.  Here $e(0),\dots,e(17)=0,0,0,1,1,1,1,1,1,2,2,2,3,3,3,3,3,3$, while $e_1(0),e_1(1),e_1(2)=0,1,3$ and $e_2(0),e_2(1),e_2(2)=1,2,3$.  The exceptional residues $n\equiv4\pmod{27}$ (for $r=3$) lie in the classes $6m+1$ and $6m+4$, because $n\equiv1\pmod 3$, and the exceptional residues $14,23\pmod{27}$ (for $r=4$) lie in the classes $6m+2$ and $6m+5$, because $n\equiv2\pmod3$.  These are exactly the classes in which the statements assert only an inequality.
\end{proof}

The divisibility $3^i\mid A_r(n)$ for large $n$ (Theorem~\ref{thm:dichotomy}(A)) also follows, with an explicit threshold, since $e(n)\to\infty$.  The excess in the exceptional classes is irregular, for example $v_3(A_3(n))-e(n)=1,2,1,1,3,1,1,2,1,1,2$ for $n=4,31,58,\dots,274$.

\section{The \texorpdfstring{$2$}{2}-adic valuations for cycle lengths at most \texorpdfstring{$4$ and $5$}{4 and 5}}\label{sec:p2}

For $p=2$ and $r=4$ two blocks of degrees $6$ and $8$ compete.

\begin{theorem}[$2$-adic valuations for cycle lengths at most $4$]\label{thm:p2r4}
Let $a(n)=A_4(n)$ (\OEIS{A070945}) and write $n=8K+\nu$ with $0\le\nu\le7$.  Then $v_2(a(n))\ge4K+\gamma(\nu)$ with $(\gamma(0),\dots,\gamma(7))=(0,0,1,1,3,3,3,3)$, and equality holds unless $n\equiv5$ or $12\pmod{16}$, where $v_2(a(n))>4K+\gamma(\nu)$.  In particular $v_2(a(2m))\ge m$ and $v_2(a(2m+1))\ge m$.
\end{theorem}

\begin{corollary}\label{cor:bala2}
For $a(n)=A_4(n)$ and every $n\ge1$, $v_2\bigl(a(2n+1)-a(2n)\bigr)\ge n+1$.
\end{corollary}

For $r=5$ a third block, of degree $10$, enters, and the exceptional classes move.

\begin{theorem}[$2$-adic valuations for cycle lengths at most $5$]\label{thm:p2r5}
Let $a(n)=A_5(n)$ (\OEIS{A070946}) and write $n=8K+\nu$ with $0\le\nu\le7$.  Then $v_2(a(n))\ge4K+\gamma(\nu)$, with the same $\gamma$ as in Theorem~\ref{thm:p2r4}, and equality holds if and only if $n\not\equiv12,13\pmod{16}$.  For $n\equiv12,13\pmod{16}$ one has $v_2(a(n))\ge4K+\gamma(\nu)+1$.
\end{theorem}

The entry \OEIS{A070945} states $v_2(A_4(2m))\ge m$, $v_2(A_4(2m+1))\ge m$ and a third inequality, printed as $v_2(a(2n+1)-a(n))\ge n+1$, with the remark that they follow by induction from the recurrence, and the third one fails at $n=1$ since $a(3)-a(1)=5$.  The same method as in Section~\ref{sec:p3} determines the valuation exactly (Theorem~\ref{thm:p2r4}), and with one more step the bound for even indices gives the third inequality with $a(2n)$ in place of $a(n)$ (Corollary~\ref{cor:bala2}).  The blocks are $6$ (with $j=1$) and $8$ (with $j=3$), the unit parts are all odd, and ties are detected by parity.  For $r=5$ the exceptional classes are again decided by parity.  The proofs of Theorems~\ref{thm:p2r4} and~\ref{thm:p2r5} and of Corollary~\ref{cor:bala2} are in Appendix~\ref{app:moved}.

The exceptional classes differ from those of Theorem~\ref{thm:p2r4} because the coefficient $\coeff{5}{F}$ is $A_4(5)/5!=4/5$ for $r=4$, which is even, and $1$ for $r=5$.  So for $r=5$ the exceptional classes are exactly the two with a tie between $(K,0,0)$ and $(K-1,2,0)$.  For $r=4$ the class $\nu=5$ with $K$ even is exceptional without any tie, and the class $\nu=5$ with $K$ odd, which has a tie, is not.

The entry \OEIS{A070946} contains Bala's statement that $N\mid A_5(n)$ for all large $n$ when $N=2^i3^j5^k$.  This is the case $r=5$ of Theorem~\ref{thm:dichotomy}, and Corollary~\ref{cor:A070946} gives the sharp thresholds for $N=2^i3^j$.

\section{The global gcd over all excluded sets}\label{sec:ties}\label{sec:globalgcd}

We first state the counterexample of Theorem~\ref{thm:introgcd}(c).  For $S=\N\setminus\{4,9\}$ the number $a_S(n)$ counts the permutations of $n$ points all of whose cycles have length $4$ or $9$.

\begin{theorem}[The two gcds differ for $r=3$]\label{thm:counter}
Let $S=\N\setminus\{4,9\}$.
\begin{enumerate}[label=(\alph*),itemsep=0pt]
\item $v_3(a_S(111))=27$, while $v_3(G_3(111))=28$, where $G_3(111)=147259794302866275623239680$.  Consequently $G_3(111)$ does not divide $a_S(111)$, and in fact $g_3(111)=G_3(111)/3$.
\item The same happens at $n=192,273,354$.  Here $v_3(a_S(n))=45,63,81$ and $v_3(G_3(n))=46,64,82$, and $g_3(n)=G_3(n)/3$.
\end{enumerate}
\end{theorem}

The valuation of $a_S(n)$ in the counterexamples comes from a general formula.

\begin{proposition}\label{prop:lead}
Let $S=\N\setminus\{4,9\}$, let $n\ge12$ with $n\equiv3\pmod9$, and put $y_0=(n-12)/9$.  Then
\[
v_3\bigl(a_S(n)\bigr)=v_3(n!)-1-2y_0-v_3(y_0!).
\]
\end{proposition}

For $n=111$ the formula gives $54-1-22-4=27$, and for $n=192,273,354$ it gives $45,63,81$.  The rest of this section explains how the counterexample is found and why this phenomenon first appears for $r=3$ and $n=111$, and it records the computations behind the other values of $r$ and $n$.

Let $r\ge1$, and let $g_r(n)$ and $G_r(n)$ be the gcds over all sets $S\supseteq\{1,\dots,r\}$ and over the family $\{1,\dots,r,k\}$, as in Theorem~\ref{thm:gcdfamily}.  The natural lower bound for $g_r(n)$ comes from the cycle types.  For a partition $\lambda$ of $n$ with $\lambda_m$ parts equal to $m$, the number of permutations of cycle type $\lambda$ is $N(\lambda)=n!/\prod_mm^{\lambda_m}\lambda_m!$.  Let $\Lambda_r(n)$ be the gcd of $N(\lambda)$ over the partitions of $n$ with all parts larger than $r$.  Every $a_S(n)$ with $S\supseteq\{1,\dots,r\}$ is a sum of numbers $N(\lambda)$, so $\Lambda_r(n)$ divides $g_r(n)$, which divides $G_r(n)$.  For $r=1,2$ the three numbers agree in all the cases we computed.  For $r=3$ the number $\Lambda_r(n)$ is smaller than the other two, for instance at $n=51$ (Example~\ref{ex:51}), and at $n=111$ even $g_3(n)$ is smaller than $G_3(n)$ (Theorem~\ref{thm:counter}).  The following proposition says exactly when each of these inequalities occurs.

For a set $P\subseteq\{r+1,\dots,n\}$ let $f_P(n)$ be the sum of $N(\lambda)$ over the partitions of $n$ whose set of distinct parts is exactly $P$.

\begin{proposition}\label{prop:globalgcd}
Let $r\ge1$ and $n\ge1$.
\begin{enumerate}[label=(\roman*),itemsep=0pt]
\item $g_r(n)=\gcd_Pf_P(n)$, and $\Lambda_r(n)$ divides $g_r(n)$.
\item For every prime $p$, $v_p(N(\lambda))=v_p(n!)-\sum_m\bigl(v_p(m)\lambda_m+v_p(\lambda_m!)\bigr)$, so the minimum of $v_p(N(\lambda))$ over the partitions is $v_p(n!)$ minus the maximum of a knapsack cost of the same shape as in Definition~\ref{def:calc}, with every part and not only the blocks.
\item Let $\mu_0=v_p(\Lambda_r(n))$.  Then $v_p(g_r(n))>\mu_0$ if and only if for every part set $P$ that carries a partition $\lambda$ with $v_p(N(\lambda))=\mu_0$, the sum of those $N(\lambda)$ with part set $P$ is divisible by $p^{\mu_0+1}$.
\item Let $n\ge r+2$.  Then $g_r(n)=G_r(n)$ if and only if for every prime $p\le r$ and every part set $P$ the number $f_P(n)$ is divisible by $p^{v_p(G_r(n))}$.  For a given prime $p$ it suffices to test the part sets $P$ that carry a partition $\lambda$ with $v_p(N(\lambda))<v_p(G_r(n))$.
\end{enumerate}
\end{proposition}

\begin{proof}
Part (ii) is Legendre's formula for $n!/\prod m^{\lambda_m}\lambda_m!$.  For (i), $a_S(n)=\sum_{P\subseteq A\cap(r,n]}f_P(n)$ for $A=\N\setminus S$, and inclusion--exclusion over the Boolean lattice writes every $f_{P_0}$ as an integer combination of the numbers $a_S(n)$, so the two gcds agree.  For (iii) and (iv), $v_p(g_r(n))=\min_Pv_p(f_P(n))$ by (i), each $f_P$ is a sum of numbers $N(\lambda)$, and the terms of valuation at least $\mu_0+1$, respectively at least $T=v_p(G_r(n))$, do not affect divisibility by $p^{\mu_0+1}$, respectively $p^T$.  The details are in Appendix~\ref{app:gg}.
\end{proof}

In words, $g_r(n)$ and $\Lambda_r(n)$ differ at $p$ exactly when the cycle types of smallest valuation all come in groups with a common part set whose leading residues cancel modulo $p$.  This is the tie phenomenon of Section~\ref{sec:p3}.

\begin{example}\label{ex:51}
Let $r=3$, $n=51$ and $p=3$, so $v_3(51!)=23$.  The part set $P=\{6,9\}$ carries the partitions $6^19^5$, $6^49^3$, $6^79^1$, whose costs $\sum_m(v_3(m)\lambda_m+v_3(\lambda_m!))$ are $12,12,11$, so the counts $N(\lambda)$ have $3$-adic valuations $11,11,12$.  The two counts of valuation $11$ have ratio $N(6^49^3)/N(6^19^5)=5/16$, so their sum is $21/16$ times the first and has valuation $12$, and their leading residues cancel.  The numerical values are $v_3(f_P(51))=13$, $v_3(g_3(51))=13$ and $v_3(\Lambda_3(51))=11$.
\end{example}

\subsection{The counterexample}

Proposition~\ref{prop:globalgcd}(iv) turns the question whether $g_r(n)=G_r(n)$ into a finite search for each $(r,n)$.  It finds the counterexample of Theorem~\ref{thm:counter}, and Proposition~\ref{prop:lead} explains its leading term.  Both proofs are in Appendix~\ref{app:moved}.

\medskip
\noindent\textbf{Computations.}  We computed $G_r(n)$ exactly and ran the test of Proposition~\ref{prop:globalgcd}(iv) for $r=3$ with $n\le900$ and for $4\le r\le12$ on smaller ranges.  For $r=3$ the two gcds differ exactly for $n=111+81j$ with $0\le j\le9$, and for $4\le r\le12$ they never differ, so $(r,n)=(3,111)$ is the smallest counterexample found.  The ranges and the part sets are in Appendix~\ref{app:data}.

\section{Valuations inside a residue class}\label{sec:interp}

The calculus of Section~\ref{sec:euler} computes $v_p(a_A(n))$ exactly when $\Xi(n)\ne0$, and for the other $n$ it gives only a lower bound.  What does the valuation do on the set $\{\Xi=0\}$?  When one block is strictly the most efficient, the answer is that inside each residue class modulo that block the valuation is the valuation of one locally analytic function of the quotient, up to an explicit Legendre term.  This answers the question about a $p$-adic integer $q^*$ in column $2$.  The proofs are in Appendices~\ref{app:interp} and~\ref{app:qstar}.

\subsection{The interpolation theorem}

Let $A\subseteq\N$ be a set of allowed cycle lengths, $p$ a prime, $W\in\Z_{(p)}[[x]]$ and $G=F_AW$, with blocks $m$ of depth $j_m$ and efficiency $\varepsilon_m=(j_m+\frac1{p-1})/m$ as in Section~\ref{sec:eff}.  Assume that $\mathcal B\ne\emptyset$ and that one block $m_*$ satisfies $\varepsilon_{m_*}>\varepsilon_m$ for every other block $m$.  Uniqueness among all blocks is the same as uniqueness among the essential blocks (Appendix~\ref{app:interp}).  Put $j=j_{m_*}$, $c_{m_*}=\alpha'/p^j$ with $\alpha'\in\Z_{(p)}^\times$, $\Gamma(q)=jq+v_p(q!)$ and $\beta_q=\binom{c_{m_*}+q-1}{q}$, so that $v_p(\beta_q)=-\Gamma(q)$.  Define $h_i$ by $\sum_ih_ix^i=G\,(1-x^{m_*})^{c_{m_*}}$ and, for $0\le\rho<m_*$,
\[
F_\rho(z)=\sum_{t\ge0}h_{\rho+tm_*}\prod_{i=0}^{t-1}\frac{p^{j}(z-i)}{\alpha'+p^{j}(z-1-i)} .
\]

\begin{theorem}[$p$-adic interpolation]\label{thm:interp}
Under these hypotheses there is an integer $\ell\ge0$ such that the series $F_\rho$ converges on $\Z_p$ and, on every disc $z_0+p^\ell\Z_p$, equals a restricted power series in $(z-z_0)/p^\ell$.  In particular $F_\rho$ is locally analytic, hence continuous, on $\Z_p$.  For all $q\ge0$ and $k=m_*q+\rho$,
\[
[x^{k}]G=\beta_q\,F_\rho(q),\qquad\text{so}\qquad v_p\bigl([x^k]G\bigr)=-\Gamma(q)+v_p(F_\rho(q)).
\]
For $W=1$ this reads $v_p(a_A(n))=v_p(n!)-\Gamma(q)+v_p(F_\rho(q))$ for $n=m_*q+\rho$.
\end{theorem}

The quantity $v_p(F_\rho(q))$ is measured against $\Gamma(q)=\mathrm{cost}(q\,e_{m_*})$, the cost of packing $q$ copies of the best block, and not against $\kappa(n)$.  Since $\kappa(n)\ge\Gamma(q)$ may be strict, $v_p(F_\rho(q))$ can be negative, and it equals the excess of $v_p(a_A(n))$ over the value $v_p(n!)-\kappa(n)$ of the calculus only when $\kappa(n)=\Gamma(q)$.  The idea of the proof is that $G=H(1-x^{m_*})^{-c_{m_*}}$ with $H=\sum h_ix^i$, so $[x^k]G=\sum_th_{\rho+tm_*}\beta_{q-t}$ and $\beta_{q-t}/\beta_q$ is the product in $F_\rho$.  The other blocks make $h_i$ grow at a rate below $\varepsilon_{m_*}$, which is why the series converges on discs.

\begin{corollary}[Zeros on discs]\label{cor:strassmann}
Under the hypotheses of Theorem~\ref{thm:interp}, for every disc $D=z_0+p^\ell\Z_p$ either $F_\rho$ vanishes identically on $D$, or $F_\rho$ has finitely many zeros in $D$, each of finite multiplicity.  Near a zero $z^*$ of multiplicity $\mu$ one has $v_p(F_\rho(q))=\mu\,v_p(q-z^*)+c$ for $q\in\N$ close enough to $z^*$, and $v_p(F_\rho)$ is bounded on the complement of any neighbourhood of the zeros.  In particular, if $F_\rho$ has a zero in $\Z_p$ then $v_p(F_\rho(q))$ is unbounded as $q$ ranges over $\N$.
\end{corollary}

\begin{corollary}[The case $W=1$]\label{cor:dich}
Under the hypotheses of Theorem~\ref{thm:interp}, let $A\ne\emptyset$, $g=\gcd(A)$ and $W=1$.  Then $F_\rho$ is identically zero if and only if $g\nmid\rho$.  If $g\mid\rho$, then $F_\rho$ has finitely many zeros in $\Z_p$.
\end{corollary}

\subsection{Column 2 and the \texorpdfstring{$p$}{p}-adic integer \texorpdfstring{$q^*$}{q*}}

Take $A=A_{>2}$ and $W=1/(1-x)$, so that $G=(1-x)^{-2}e^{-x-x^2/2}$ and $Q_{k+3}=k!\,[x^k]G$ (Theorem~\ref{thm:quotient}).  For an odd prime $p$ the block $m_*=p$ has depth $1$ and $\alpha'=1$, and every other block has depth $1$ and degree at least $3p$ (Section~\ref{sec:second}).  These are the only properties used, and Theorem~\ref{thm:col2} is the first level of the following statement.

\begin{theorem}[The $p$-adic integer $q^*$]\label{thm:qstar}
Let $p$ be an odd prime and $0\le\nu<p$ with $v_p(Q_{\nu+3})=1$, and let $\xi$ and $\zeta$ be as in Theorem~\ref{thm:col2} and $F_\nu$ the function of Theorem~\ref{thm:interp} for this $G$, $m_*=p$ and $\rho=\nu$.  If $p\nmid\zeta$, there is a unique $q^*\in\Z_p$, with $q^*\equiv-\xi/\zeta\pmod p$, such that $v_p(F_\nu(z))=1+v_p(z-q^*)$ for every $z\in\Z_p$.  Consequently $v_p(Q_n)=1+v_p(q-q^*)$ for every $n=\nu+3+pq$ with $q\in\N$, and $q^*\notin\N$.  If $p\mid\zeta$, then $v_p(Q_n)=1$ for every $n=\nu+3+pq$ with $q\in\N$.
\end{theorem}

The proof expands $F_\nu/p$ modulo $p$ on the full disc $\Z_p$, where it is $\xi+\zeta z$, and applies Weierstrass preparation in its general form, which gives a monic polynomial of degree one with reduction $z+\xi/\zeta$ (Appendix~\ref{app:qstar}).  The numbers $Q_n$ are non-zero because $T(n,2)>0$ for $n\ge3$, since an $n$-cycle has no cycle of length $1$ or $2$, and this is why $q^*$ is not a natural number.  For the seven pairs $(p,\nu+3)$ listed in Question~\ref{prob:odd} the root $q^*$ modulo $p^j$ is the residue class found there.

\section{Composite moduli}\label{sec:composite}

This section combines the primes.  It describes the set of $n$ with $m\mid a_A(n)$ for a composite $m$ and the gcd of $a_A(n)$ over $n\ge n_0$.  For a set $X\subseteq\N$ put $\dens X=\lim_{x\to\infty}|X\cap[0,x)|/x$ when the limit exists.  We write $a=a_A$, and $Q(n)=a_{A_{>r}}(n)/\fall{n-1}{r}$ for the quotient family of Theorem~\ref{thm:quotient}, studied in the variable $n-r-1$.

\subsection{Two kinds of primes}

The example of $A_3(n)$ modulo $5$ and modulo $10$ in Section~\ref{sec:main} shows the two behaviours.

\begin{definition}\label{def:types}
Let $A\subseteq\N$ and let $p$ be a prime.  We say that $p$ is of \emph{type~I} for $A$ if $[1\in A]=[p\in A]$, and of \emph{type~II} otherwise.  For an integer $m=\prod p^{e_p}$ we write $m=m_{\mathrm I}m_{\mathrm{II}}$, where $m_{\mathrm I}$ collects the prime powers at primes of type~I.  Put $t_p(e)=p^e$ for odd $p$ and $t_2(e)=2^{\max(e,4)}$, and for $e\ge1$ let $Z(p^e)=\{n\ge0:p^e\mid a(n)\}$.
\end{definition}

For $A_{\le r}$ and for $A_{>r}$ the primes of type~I are the $p\le r$.  For derangements every prime is of type~II.  Type~I is case~(A) of the dichotomy of Theorem~\ref{thm:dichotomy}.

\begin{lemma}[Primes of type~II]\label{lem:typeB}
Let $p$ be of type~II for $A$.
\begin{enumerate}[label=(\alph*),itemsep=0pt]
\item For every $e\ge1$ and $n\ge0$ one has $a(n+p^e)\equiv\pm a(n)\pmod{p^e}$.  For odd $p$ the sign is $+$ when $1\in A$, $p\notin A$, and $-$ when $1\notin A$, $p\in A$.  For $p=2$ the sign is $+$ except for $e=2$, where it is $(-1)^{[2\in A]+[4\in A]}$.  Hence $Z(p^e)$ is purely periodic, its least period $L_Z(p^e)$ divides $p^e$, and the least period $\varpi(p^e)$ of $a(n)\bmod p^e$ divides $2p^e$, and divides $p^e$ when the sign is $+$, in particular for $p=2$ and $e\ne2$.
\item $p\nmid a(n)$ whenever $p\mid n$.  In particular $p^e\mid a(n)$ fails for infinitely many $n$.
\item The calculus at $p$ gives $v_p(n!)-\kappa(n)=0$ for every $n$, and the exceptional set is $E_p=Z(p)=\{n:p\mid a(n\bmod p)\}$, a union of classes modulo $p$.  For $p=2$ it is empty if $1\in A$, and it is the set of odd $n$ if $1\notin A$.
\item If $p$ is odd then $m_*=p$, $j=1$, every other block has $\varepsilon_mp<1$, and for $n=pq+\nu$ with $0\le\nu<p$ one has $v_p(a(n))=v_p(F_\nu(q))$, with $F_\nu$ the function of Theorem~\ref{thm:interp}.
\end{enumerate}
\end{lemma}

The proof of Lemma~\ref{lem:typeB} is in Appendix~\ref{app:moved}.

\begin{lemma}[Primes of type~I]\label{lem:typeA}
Let $p$ be of type~I for $A$.
\begin{enumerate}[label=(\alph*),itemsep=0pt]
\item $p^e\mid a(n)$ for all $n\ge t_p(e)$.  Hence $n_0(p^e)=\min\{N:p^e\mid a(n)\text{ for all }n\ge N\}$ equals $1+\max\{n<t_p(e):p^e\nmid a(n)\}$, or $0$, a finite computation.
\item Every block $m$ has $\varepsilon_m<\frac1{p-1}$.  Let $\varepsilon_{\max}$ be the largest efficiency of a block, and $\varepsilon_{\max}=0$ if there is none.  Then $\sigma_p=\frac1{p-1}-\varepsilon_{\max}>0$ and
\[
v_p\bigl(a(n)\bigr)\ \ge\ \sigma_pn-\log_pn-1\qquad(n\ge1).
\]
\end{enumerate}
\end{lemma}

\begin{proof}
(a) is case~(A) of Theorem~\ref{thm:dichotomy}.  (b) Here $c_p=\frac1p([p\in A]-[1\in A])=0$, so $p$ is not a block.  A block $m$ has $p\mid m$ and $j_m\le v_p(m)$, so $m\ge p^{j_m}$.  If $j_m=1$ then $m\ge2p>(p-1)j_m+1$, and if $j_m\ge2$ then $p^{j_m}>(p-1)j_m+1$.  So $m>(p-1)j_m+1$, which is $\varepsilon_m<\frac1{p-1}$.  Since $\varepsilon_m\le(\log_pm+1)/m\to0$, the supremum is a maximum when blocks exist.  A block vector $V$ has $\mathrm{cost}(V)\le\sum_m\varepsilon_mm\lambda_m\le\varepsilon_{\max}\deg V$, so $\kappa(n)\le\varepsilon_{\max}n$, while $v_p(n!)=(n-s_p(n))/(p-1)\ge\frac n{p-1}-\log_pn-1$, and Corollary~\ref{cor:calc} gives the bound.  If there are no blocks, $\kappa=0$ by Theorem~\ref{thm:saturation}.
\end{proof}

The composite statements rest on one elementary lemma.  For a set $S\subseteq\N$ that is purely periodic with period $L$ and an integer $\mathcal M\mid L$, the \emph{profile} of $S$ modulo $\mathcal M$ is $\pi_S(c)=\frac{\mathcal M}L\,|S\cap[0,L)\cap(c+\mathcal M\Z)|$, the density of $S$ inside the class $c$.

\begin{lemma}[Chinese remainder theorem for periodic sets]\label{lem:crt}
Let $S_1,\dots,S_h$ be purely periodic with periods $L_1,\dots,L_h$.
\begin{enumerate}[label=(\alph*),itemsep=0pt]
\item For a prime $\ell$ let $s_\ell$ be the second largest of the numbers $v_\ell(L_i)$, and $0$ if $h=1$, and put $\mathcal M=\prod_\ell\ell^{s_\ell}$.  Then
\[
\dens\Bigl(\bigcap_iS_i\Bigr)=\frac1{\mathcal M}\sum_{c\bmod\mathcal M}\ \prod_i\pi_{S_i}(c),
\]
with $\pi_{S_i}$ the profile modulo $\gcd(\mathcal M,L_i)$, read at $c\bmod\gcd(\mathcal M,L_i)$.  The same holds with any $S_i$ replaced by its complement.
\item For $h=2$ one has $\mathcal M=\gcd(L_1,L_2)$ and $\dens(S_1\cap S_2)=\dens S_1\dens S_2+\mathrm{Cov}$ with $\mathrm{Cov}=\frac1{\mathcal M}\sum_c(\pi_{S_1}(c)-\dens S_1)(\pi_{S_2}(c)-\dens S_2)$.  So $S_1$ and $S_2$ are independent if and only if $\mathrm{Cov}=0$, which holds when one profile is constant, in particular when $\mathcal M=1$.
\item If the $L_i$ are pairwise coprime, $L_i$ is the least period of $S_i$ and every $S_i$ is non-empty, then the least period of $\bigcap_iS_i$ is $\prod_iL_i$.
\end{enumerate}
\end{lemma}

The proof of Lemma~\ref{lem:crt} is in Appendix~\ref{app:moved}.

\subsection{Composite moduli}

The headline is Theorem~\ref{thm:introcomp}, which the next theorem restates with the exact threshold and the least period.  The details on residues and tail gcds follow in a proposition.

\begin{theorem}[Composite moduli]\label{thm:composite}
Let $A\subseteq\N$, $a=a_A$ and $m=m_{\mathrm I}m_{\mathrm{II}}$.
\begin{enumerate}[label=(\roman*),itemsep=0pt]
\item $m_{\mathrm I}\mid a(n)$ for all $n\ge n_0(m_{\mathrm I})=\max_{p\mid m_{\mathrm I}}n_0(p^{e_p})$, where the maximum over the empty set is $0$, and if $n_0(m_{\mathrm I})\ge1$ then $m_{\mathrm I}\nmid a(n_0(m_{\mathrm I})-1)$.
\item The set $Z(m_{\mathrm{II}})=\{n:m_{\mathrm{II}}\mid a(n)\}$ is purely periodic.  Its least period is $\prod_{p\mid m_{\mathrm{II}}}L_Z(p^{e_p})$ if every $Z(p^{e_p})$ is non-empty, and it is empty otherwise.
\item $Z(m)$ and $Z(m_{\mathrm{II}})$ agree on $[n_0(m_{\mathrm I}),\infty)$, and $\dens Z(m)=\prod_{p\mid m_{\mathrm{II}}}\dens Z(p^{e_p})$.  For distinct primes of type~II the events $p^{e_p}\mid a(n)$ are independent.
\item $m\mid a(n)$ for all large $n$ if and only if $m_{\mathrm{II}}=1$.
\end{enumerate}
\end{theorem}

\begin{proof}
(i) is Lemma~\ref{lem:typeA}(a) for each $p\mid m_{\mathrm I}$.  (ii) is Lemma~\ref{lem:typeB}(a) with Lemma~\ref{lem:crt}(c), since the $L_Z(p^{e_p})$ are powers of distinct primes.  (iii) follows from (i), (ii) and Lemma~\ref{lem:crt}(a).  (iv) is Theorem~\ref{thm:dichotomy} with Lemma~\ref{lem:typeB}(b).
\end{proof}

For $A_3$ and $m=10$ one has $m_{\mathrm I}=2$, $n_0(2)=2$ and $Z(5)=\emptyset$, so no $A_3(n)$ is divisible by $10$, and the residues modulo $10$ are periodic from $n=2$ with period $5$.

\begin{proposition}[Residues and tail gcds]\label{thm:compositeA}
In the setting of Theorem~\ref{thm:composite} the following hold.
\begin{enumerate}[label=(\roman*),itemsep=0pt]
\item $a(n)\bmod m$ is periodic from exactly $n_0(m_{\mathrm I})$ on, and from $0$ if $m_{\mathrm I}=1$, and its least period is $\operatorname{lcm}_{p\mid m_{\mathrm{II}}}\varpi(p^{e_p})$.
\item The integers $m$ with $m\mid a(n)$ for all $n\ge n_0$ are exactly the divisors of the tail gcd $\mathcal D_A(n_0)=\gcd_{n\ge n_0}a(n)$.  If some $a(n)$ with $n\ge n_0$ is non-zero, then $\mathcal D_A(n_0)=\prod_{p\text{ of type I}}p^{\mu_p(n_0)}$ with $\mu_p(n_0)=\min_{n\ge n_0}v_p(a(n))$, and the minimum is over a finite window, which Lemma~\ref{lem:typeA}(b) bounds.
\item $m\mid a(n)$ for infinitely many $n$, equivalently for a set of $n$ of positive density, if and only if $Z(p^{e_p})\ne\emptyset$ for every $p\mid m_{\mathrm{II}}$.
\end{enumerate}
\end{proposition}

The proof of Proposition~\ref{thm:compositeA} is in Appendix~\ref{app:moved}.

For cycle lengths at most $r$ the primes of type~II are the primes $p>r$, and their sets $Z(p)$ are explicit.

\begin{corollary}[Cycle lengths at most $r$]\label{cor:compAr}
Let $r\ge1$ and $a=A_r$.  The primes of type~I are the $p\le r$.  For a prime $p>r$ one has $Z(p)=\{n:n\bmod p\in R_p\}$ with $R_p=\{\nu:r<\nu<p,\ p\mid A_r(\nu)\}$, and $\dens Z(p)=|R_p|/p$.  The least period of $A_r(n)\bmod m$ is the product of the $\varpi(p^{e_p})$ with $p>r$, each of which divides $p^{e_p}$.
\end{corollary}

The proof of Corollary~\ref{cor:compAr} is in Appendix~\ref{app:moved}.

For $2\le r\le8$, primes $r<p<40$ and small $e$ the least period of $A_r(n)\bmod p^e$ is exactly $p^e$, and for derangements and for $A_{>r}$ with $r\le4$ it is exactly $2p^e$.  The zero set can have a smaller period, for example $L_Z(9)=3$ for derangements, where $v_3(D(n))\ge2$ on the whole class $n\equiv1\pmod3$.

The quotient family has the same structure at the primes outside $S$.

\begin{theorem}[The quotient family]\label{thm:compquot}
Let $r\ge1$, $S\supseteq\{1,\dots,r\}$ and $Q_S(n)=a_S(n)/\fall{n-1}{r}$ for $n>r$.  For every prime $p\notin S$ and every $e\ge1$ one has $Q_S(n+p^e)\equiv-Q_S(n)\pmod{p^e}$ for odd $p$, and $Q_S(n+2^e)\equiv\pm Q_S(n)\pmod{2^e}$ for $p=2$ with the sign of Lemma~\ref{lem:typeB}(a), for all $n>r$.  Hence $\{n>r:p^e\mid Q_S(n)\}$ is purely periodic with period dividing $p^e$, and for $m$ composed of primes outside $S$ the set $\{n>r:m\mid Q_S(n)\}$ is purely periodic, with least period the product of the least periods of its prime-power parts and density the product of their densities.
\end{theorem}

The proof of Theorem~\ref{thm:compquot} is in Appendix~\ref{app:moved}.

Now we turn to the entries \OEIS{A057693}, \OEIS{A070945} and \OEIS{A070946}, which record Bala's statements.  Write $\lambda_3(n)=\lfloor n/9\rfloor+\lfloor(n+6)/9\rfloor$, the function $e(n)$ of Theorem~\ref{thm:p3}, and $\lambda_2(8K+\nu)=4K+\gamma(\nu)$ with $\gamma=(0,0,1,1,3,3,3,3)$, as in Theorem~\ref{thm:p2r4}.

\begin{corollary}[\OEIS{A057693}]\label{cor:A057693}
Let $a(n)=A_3(n)$.
\begin{enumerate}[label=(\alph*),itemsep=0pt]
\item For every $n_0\ge0$, $\gcd_{n\ge n_0}a(n)=2^{1+\lfloor(n_0-2)/4\rfloor}\,3^{\lambda_3(n_0)}$.
\item $n_0(2^e)=4e-2$, and $n_0(3^e)=\frac92(e-1)+3$ for odd $e$ and $\frac92e$ for even $e$.  So $2^i3^j\mid a(n)$ for all $n\ge\max(n_0(2^i),n_0(3^j))$, and this threshold is sharp.
\item For every $k\ge1$ the sequence $a(n)\bmod k$ is periodic from $n_0(k_{\mathrm I})$ on, with least period $\prod_{p\mid k,\,p>3}\varpi(p^{e_p})$.  For example the pre-period and the period are $0$ and $7$ modulo $7$, and $2$ and $5$ modulo $10$.
\item The exceptional sets are $E_2=\emptyset$, $E_3=\{n\equiv4\pmod{27}\}$, and $E_p=Z(p)$ for $p\ge5$, with densities $0$, $\frac1{27}$ and $|R_p|/p$, and they are independent.
\end{enumerate}
\end{corollary}

The proof of Corollary~\ref{cor:A057693} is in Appendix~\ref{app:moved}.

\begin{corollary}[\OEIS{A070945}]\label{cor:A070945}
Let $a(n)=A_4(n)$.  Then $\gcd_{n\ge n_0}a(n)=2^{\lambda_2(n_0)}3^{\lambda_3(n_0)}$ for every $n_0\ge0$.  With $e-1=4K+t$ and $0\le t\le3$ one has $n_0(2^e)=8K+2,8K+4,8K+4,8K+8$ for $t=0,1,2,3$, and $n_0(3^e)$ is as in Corollary~\ref{cor:A057693}.  The exceptional sets are $E_2=\{n\equiv5,12\pmod{16}\}$, $E_3=\{n\equiv14,23\pmod{27}\}$ and $E_p=Z(p)$ for $p\ge5$, with densities $\frac18$, $\frac2{27}$ and $|R_p|/p$, and they are independent.  Modulo $10$ the sequence is periodic from $n=2$ with least period $5$.
\end{corollary}

\begin{proof}
Theorem~\ref{thm:p2r4} gives $v_2(a(n))\ge\lambda_2(n)$ with equality unless $n\equiv5,12\pmod{16}$.  The function $\lambda_2$ is non-decreasing and constant on the blocks $\{8K,8K+1\}$, $\{8K+2,8K+3\}$ and $\{8K+4,\dots,8K+7\}$, and the last two points $8K+6,8K+7$ of the third block, as well as all points of the first two, are not exceptional.  This gives $\mu_2$ and $n_0(2^e)$ as in Corollary~\ref{cor:A057693}.  At $p=3$, Theorem~\ref{thm:p3} gives the bound $\lambda_3(n)$ with exceptions in the residue $5\bmod9$ only, and the rest is as before.
\end{proof}

The entry \OEIS{A070945} also prints $v_2(a(2n+1)-a(n))\ge n+1$, which fails for $347$ of the $n\le348$.  The statement with $a(2n)$ in place of $a(n)$ is Corollary~\ref{cor:bala2}.

\begin{corollary}[\OEIS{A070946}]\label{cor:A070946}
Let $a(n)=A_5(n)$.  The exceptional sets at $2$ and $3$ are $E_2=\{n\equiv12,13\pmod{16}\}$ and $E_3=\emptyset$, with densities $\frac18$ and $0$, and for every $n_0\ge0$
\[
\gcd_{n\ge n_0}a(n)=2^{\lambda_2(n_0)}3^{\lambda_3(n_0)}5^{\mu_5(n_0)},\qquad\mu_5(n_0)=\min_{n\ge n_0}v_5(a(n)).
\]
So $2^i3^j$ divides $a(n)$ for all $n\ge\max(n_0(2^i),n_0(3^j))$, with $n_0(2^e)$ as in Corollary~\ref{cor:A070945} and $n_0(3^e)$ as in Corollary~\ref{cor:A057693}, and this threshold is sharp.
\end{corollary}

The proof of Corollary~\ref{cor:A070946} is in Appendix~\ref{app:moved}.

\section{Column \texorpdfstring{$2$}{2} and open questions}\label{sec:open}

We justify Examples~\ref{ex:830} and~\ref{ex:19}, show that infinitely many odd primes divide some term of column $2$ after the factor $(n-1)(n-2)$ is removed, prove the second-order theorem for column $2$, and state the open questions.

\subsection{The column \texorpdfstring{$2$}{2} at odd primes}

Let $Q_n\coloneqq T(n,2)/((n-1)(n-2))$ for $n\ge3$.  This is an integer by Theorem~\ref{thm:general}(i) with $r=2$, and it begins $1,1,2,8,38,214,1444$.  The exact formula $v_p(T(n,2))=v_p((n-1)(n-2))$ for an odd prime $p$ is equivalent to $p\nmid Q_n$.

\begin{proposition}\label{prop:col2}
For every odd prime $p$ and every $n\ge3$, $Q_{n+p}\equiv-Q_n\pmod p$.  Hence $p\nmid Q_n$ for all $n\ge3$ if and only if $p\nmid Q_n$ for $3\le n\le p+2$.  For $p=3$ this holds, so $v_3(T(n,2))=v_3((n-1)(n-2))$ for all $n\ge3$.
\end{proposition}

\begin{proof}
This is Lemma~\ref{lem:quot} with $r=2$ and $S=\{1,2\}$, since $Q_n=Q_S(n)$ and $p\notin S$.  For $p=3$ the values $Q_3,Q_4,Q_5=1,1,2$ are not divisible by $3$, hence neither are all $Q_n$.
\end{proof}

By this criterion the exact formula holds for $3,5,7,11,13,17,29,43,53,61,71,83$ among the first odd primes, and fails at $p=19$ already at $n=7$, and by Lemma~\ref{lem:A} also at $n=26$, since $T(26,2)\equiv-T(7,2)\equiv0\pmod{19}$ although $19\nmid25\cdot24$.  Of the $238$ odd primes below $1500$, exactly $98$ satisfy it.  If $q\ge5$ and $n-1=tq$, then Lemma~\ref{lem:quot} and $T(q+1,2)/q=(q-1)Q_{q+1}$ give $T(n,2)/(n-1)\equiv\pm Q_{q+1}\pmod q$, so $q^{v_q(n-1)+1}\mid T(n,2)$ if and only if $q\mid Q_{q+1}$.  Among primes $q<1000$ this happens only for $q=829$, which is Example~\ref{ex:830}.

\begin{example}[Odd primes in column $2$]\label{ex:19}
One might hope that $v_p(T(n,2))=v_p((n-1)(n-2))$ for all odd primes $p$ and $n\ge3$, as is true for $p=3,5,7$ (Proposition~\ref{prop:col2}).  At $p=19$, $n=7$ we have $T(7,2)=1140=19\cdot60$ while $19\nmid6\cdot5$.
\end{example}

The primes at which the exact formula fails do not run out.  The sign rule of Proposition~\ref{prop:col2} keeps a prime $p$ away from the terms $Q_n$ with $n\equiv3\pmod p$, the $2$-adic valuation of $Q_n$ is known exactly from Theorem~\ref{thm:A378}(f), and $Q_n$ grows like a factorial.  If only finitely many odd primes divided the numbers $Q_n$, these three facts would make some large $Q_n$ a small power of $2$.

\begin{theorem}[Infinitely many primes divide column $2$]\label{thm:bad}
An odd prime $p$ divides some $Q_n$ with $n\ge3$ if and only if it divides some $Q_n$ with $3\le n\le p+2$.  There are infinitely many such primes.
\end{theorem}

\begin{proof}
By Proposition~\ref{prop:col2}, $Q_{n+p}\equiv-Q_n\pmod p$, so $p\mid Q_n$ for some $n\ge3$ implies $p\mid Q_{n'}$ for the $n'\in[3,p+2]$ with $n'\equiv n\pmod p$.  Suppose that the odd primes dividing some $Q_n$ form a finite set $\mathcal P$.  By the Chinese remainder theorem there are arbitrarily large $n$ with $n\equiv3\pmod4$ and $n\equiv3\pmod p$ for every $p\in\mathcal P$.  For such $n$ Proposition~\ref{prop:col2} gives $Q_n\equiv\pm Q_3=\pm1\pmod p$ for $p\in\mathcal P$, and no other odd prime divides $Q_n$, so $Q_n$ is a power of $2$.  For $n=4m+3$, Theorem~\ref{thm:A378}(f) gives $v_2(T(n,2))=m+1$, and $v_2((n-1)(n-2))=1$, so $Q_n=2^m$.  But $T(n,2)\ge(n-1)!$, since the $n$-cycles are counted, so $Q_n\ge(n-3)!>2^m$ for $n\ge7$, a contradiction.
\end{proof}

\subsection{First and second order at an odd prime}\label{sec:second}

For the odd primes in column $2$ of \OEIS{A378495} the first-order calculus gives $v_p(T(n,2))=v_p((n-1)(n-2))$ in all classes where $p\nmid Q_{\nu+3}$, with $Q_n=T(n,2)/((n-1)(n-2))$ and $\nu=(n-3)\bmod p$.  The next order is available when $v_p(Q_{\nu+3})=1$.

\begin{theorem}[Valuations of $T(n,2)$ at odd primes]\label{thm:col2}
Let $p$ be an odd prime, $n=3+\nu+pq$ with $0\le\nu<p$ and $q\ge1$, and suppose $v_p(Q_{\nu+3})=1$. Put $\xi=Q_{\nu+3}/(p\,\nu!)$ and $\zeta=Q_{\nu+p+3}/(\nu+p)!-\xi$. Then $v_p(Q_n)=1$ unless $p\nmid q$, $\zeta\not\equiv0$ and $q\equiv-\xi/\zeta\pmod p$, in which case $v_p(Q_n)\ge2$. Consequently $v_p(T(n,2))=v_p((n-1)(n-2))+1$ for all other $n$ in the class of $\nu$.
\end{theorem}

By Theorem~\ref{thm:quotient} with $r=2$ and $A=A_{>2}$, one has $Q_{k+3}/k!=[x^k]G$ for $G=F_{A_{>2}}/(1-x)=(1-x)^{-2}e^{-x-x^2/2}$. By Lemma~\ref{lem:eulerprod}, $c_n=-(\mu(n)+\mu(n/2)[2\mid n])/n$ for $n\ge2$. For an odd prime $p$ the blocks are therefore the integers $m=pd$ with $d$ odd and squarefree and prime to $p$, or $d=4d'$ with $d'$ odd, squarefree and prime to $p$, together with $m=p$ itself, all of depth $1$, and $c_p=1/p$, so $\alpha'=1$. Every block other than $p$ is $pd$ with $d\ge3$.

\begin{corollary}[First order]\label{cor:firstcol2}
For an odd prime $p$ and $k=\nu+pq$ with $0\le\nu<p$, one has $v_p(Q_{k+3})=0$ if and only if $p\nmid Q_{\nu+3}$, and otherwise $v_p(Q_{k+3})\ge1$.
\end{corollary}

\begin{proof}
Every block other than $p$ has $\lfloor m/p\rfloor\ge3>1=j_m$, so ($\star$) holds with $m_*=p$ and $j_*=1$, and Theorem~\ref{thm:dominant} applies with $[x^\nu]G=Q_{\nu+3}/\nu!$ and $\nu!$ a unit. It gives $v_p([x^k]G)\ge-q-v_p(q!)$ with equality if and only if $p\nmid Q_{\nu+3}$. Now $v_p(Q_{k+3})=v_p(k!)+v_p([x^k]G)$ and $v_p(k!)=q+v_p(q!)$.
\end{proof}

This is the statement of Proposition~\ref{prop:col2} sharpened to valuations.  When $p\mid Q_{\nu+3}$ the first-order calculus only gives $v_p(Q_{k+3})\ge1$.  The second-order analysis, in which only the vectors $q$ and $q-1$ of blocks $p$ contribute modulo the next power of $p$, proves Theorem~\ref{thm:col2}, and it is carried out for a general dominant block of depth one in Appendix~\ref{app:second}.

For $p=19$ and $n_0=7$ one has $T(7,2)=1140=19\cdot60$, and the theorem predicts $v_{19}(Q_{7+19q})\ge2$ exactly for $q\equiv9\pmod{19}$. There are $30$ pairs $(p,n_0)$ with $p<110$, $3\le n_0\le p+2$ and $v_p(Q_{n_0})=1$, and in each case the class of $q$ with $v_p(Q_{n_0+pq})\ge2$ is the predicted one, or empty when $\zeta\equiv0$, as the ancillary program confirms.

\subsection{Open questions}

Each of the following questions is open, and we state it with the numerical evidence that suggests the answer.

\begin{problem}[Odd primes in column $2$]\label{prob:odd}
Call an odd prime $p$ \emph{good} if $p\nmid Q_n$ for all $n\ge3$, equivalently for $3\le n\le p+2$ by Proposition~\ref{prop:col2}.  The good primes begin $3,5,7,11,13,17,29,43,\dots$.  Of the $216815$ odd primes up to $3\cdot10^6$, exactly $79656$ are good, a proportion $0.3674$ with $95$ percent interval $[0.3654,0.3694]$, which contains $e^{-1}=0.3679$, and the number of $n\in[3,p+2]$ with $p\mid Q_n$ is close to a Poisson variable of mean $1$ (Appendix~\ref{app:data}).  Is the density of good primes equal to $e^{-1}$?  Are there infinitely many good primes?  The other odd primes, the bad ones, are the odd prime factors of the numbers $Q_n$, and there are infinitely many of them by Theorem~\ref{thm:bad}.  Its proof shows that every finite set of odd primes misses an odd prime factor of some $Q_n$ with $n\equiv3\pmod4$, and it produces these new bad primes from the growth of $Q_n$ alone.  Remark~\ref{rem:lines} reduces the question to the position of one point of $\mathbb P^2(\F_p)$ relative to a fixed arrangement of lines.  Theorem~\ref{thm:qstar} gives the valuation of $T(n,2)$ in the classes with $v_p(Q_{\nu+3})=1$, and it covers all $30$ pairs $(p,n_0)$ of Section~\ref{sec:second}, of which only $(19,10)$ has $p\mid\zeta$.  For seven of these pairs the classes of $q$ with $v_p(Q_{n_0+pq})\ge1+j$ were checked for $j\le3$ and $q\le30000$, as Theorem~\ref{thm:qstar} predicts (Appendix~\ref{app:data}).
\end{problem}

\begin{remark}[What can be proved about the good primes]\label{rem:lines}
Put $q_\nu=Q_{\nu+3}$.  From $(1-x)G'=(1+x^2)G$ for $G=(1-x)^{-2}e^{-x-x^2/2}$ one gets $q_{\nu+1}=(\nu+1)q_\nu+\nu(\nu-1)q_{\nu-2}$ with $q_0,q_1,q_2=1,1,2$.  Define vectors $u_0,u_1,u_2=e_0,e_1,e_2$ in $\Q^3$ and $u_{m+2}=(u_{m-1}+m\,u_m)/((m+1)(m+2))$ for $m\ge1$.  These vectors do not depend on $p$, and for every odd prime $p$ and $0\le m\le p-3$ their denominators are prime to $p$ and $q_{p-1-m}\equiv\langle u_m,c(p)\rangle\pmod p$ with $c(p)=(q_{p-1},q_{p-2},q_{p-3})\bmod p$.  Hence $p$ is bad if and only if the point $[c(p)]$ of the projective plane $\mathbb P^2(\F_p)$ lies on one of the $p-2$ lines $\{u_m^\perp\}$ of a fixed arrangement, and the number of $n\in[3,p+2]$ with $p\mid Q_n$ is the multiplicity of that point.  For a uniformly random point and distinct lines the first two factorial moments of this multiplicity are $(p-2)(p+1)/(p^2+p+1)$ and $(p-2)(p-3)/(p^2+p+1)$, which tend to the values $1,1$ of a Poisson variable of mean $1$.  Consequently, if (H1) the proportion of points of $\mathbb P^2(\F_p)$ on none of the lines tends to $e^{-1}$, and (H2) the number of primes $p\le x$ with $[c(p)]$ on none of the lines is, up to $o(\pi(x))$, the sum over $p\le x$ of that proportion, then the good primes have density $e^{-1}$.  This conditional statement is proved in Appendix~\ref{app:density}.

The density is not a Chebotarev statement, since the number of conditions grows with $p$ (Appendix~\ref{app:data}).
\end{remark}

\begin{problem}[Least periods at primes of type~II]\label{prob:composite}
Section~\ref{sec:composite} gives the set of $n$ with $m\mid a_A(n)$ and its density.  Is the least period of $a(n)\bmod p^e$ at a prime of type~II always $p^e$, or $2p^e$ when the sign of Lemma~\ref{lem:typeB}(a) is $-$?  This holds in every case we computed.  The natural approach shows $a(p^{e-1})\not\equiv1\pmod{p^e}$, which needs a second-order shift congruence, while the residue formulas behind Theorem~\ref{thm:dichotomy} give $a(p^e)\bmod p^e$ only, and the interpolation of Theorem~\ref{thm:interp} controls valuations and not residues.
\end{problem}

\begin{problem}[Cycle lengths prime to $n$ when $n+1$ is prime]\label{prob:335}
By Theorem~\ref{thm:A335088} there are exactly two values $n\le3\cdot10^6$ with $n+1$ prime and $(n+1)\mid\operatorname{cop}(n)$, namely $n=708$ and $n=92202$, and Observation~\ref{obs:cop} gives a heuristic count that predicts infinitely many.  No Fermat prime $n+1$ occurs by Proposition~\ref{prop:copfermat}, for $n+1=2q+1$ the condition is the single binomial congruence of Proposition~\ref{prop:copsafe}, and for $n=2^aq^b$ it is the vanishing of the short sum of Proposition~\ref{prop:cop2q}, which happens for none of the $454$ primes $N<20000$ of this form.  Are there infinitely many such $n$?  Is there a structural reason for $n=708=2^2\cdot3\cdot59$ and $n=92202=2\cdot3\cdot11^2\cdot127$?  Both have three distinct prime factors, where the product of Proposition~\ref{prop:copprefix} has at least seven factors and we know no closed form.  A proof that there are infinitely many such $n$ would need an equidistribution statement for a residue whose defining product changes with $N$, and we know no such statement.
\end{problem}

\begin{problem}[The global gcd]\label{prob:gcdr}
By Theorem~\ref{thm:gcdfamily}, $(n-1)_r\mid g_r(n)\mid G_r(n)=(n-1)_rm_r(n)$ with $m_r(n)$ involving only primes at most $r$, and $m_1(n)=1$, $m_2(n)=2^{\lfloor(n-1)/4\rfloor}$.  By the computations of Section~\ref{sec:globalgcd}, for $r=3$ and $n\le900$ the two gcds differ exactly for $n=111+81j$ with $0\le j\le9$.  Do they differ for every $n\equiv30\pmod{81}$ with $n\ge111$ and for no other $n$?  Is $g_r(n)=G_r(n)$ for every $r\ge4$ and every $n\ge r+2$?  What is the exponent of each prime $p\le r$ in $g_r(n)/(n-1)_r$ for $r\ge3$?  The gcd $\Lambda_r(n)$ of the cycle type counts divides $g_r(n)$ (Proposition~\ref{prop:globalgcd}), and for $r=3$ the two differ at $p=3$ for $n=51,55,56,57,78,82,83,84$ in $n\le90$.  For which $n$ do $\Lambda_r(n)$ and $g_r(n)$ differ?
\end{problem}

\section{Computations and formalization}\label{sec:computations}

\subsection{Ancillary files}

The program \texttt{verify\_restricted\_cycles.py} and the C programs in the folder \texttt{anc}, supplied as ancillary files, use exact integers and re-check the numerical claims of the paper against brute-force enumeration, the OEIS entries, and the recurrences, on the ranges stated in their output and comments.  The folders of \texttt{anc/interp} contain the scripts and outputs behind Sections~\ref{sec:interp} and~\ref{sec:composite} and behind the observations on $\operatorname{cop}(n)$ and column $2$, and the README files of \texttt{anc} and of its folders match the names in the scripts with the notation of the paper.  The verifier reruns the key computations on small ranges, recounts the stored outputs of the long runs, and checks the numerical examples of the introduction with its own exact integers.  The whole run takes about seven minutes on a laptop with ten cores.  Appendix~\ref{app:data} lists the folders and what each check covers.

The supplementary catalogue in the files \texttt{anc/catalogue\_restricted\_oeis.*} lists the $99$ OEIS sequences to which a theorem of the paper applies, with the statement for each, the theorem used, what the entry records, and the range on which the statement was checked against the terms of the entry.

\subsection{Formalization}

Part of the paper is formalized in Lean~4 with Mathlib, and the sources are among the ancillary files in \texttt{anc/lean}, with a README that gives the versions and the build commands.  The files build without \texttt{sorry} and without \texttt{native\_decide}, and the theorems listed here depend only on the axioms \texttt{propext}, \texttt{Classical.choice} and \texttt{Quot.sound}.  The recurrence of Section~\ref{sec:tools} is proved to count the permutations of a finite set with the given cycle lengths, and the following statements are formalized.  The exact $2$-adic valuation of the number of involutions (Theorem~\ref{thm:invol}).  The exactness half of the Valuation Calculus, namely that $v_p(a_A(n))$ equals $v_p(n!)-\kappa(n)$ exactly when $\Xi(n)\ne0$, for every set $A$ (Theorem~\ref{thm:calculus} and Corollary~\ref{cor:calc}).  The integrality criterion (Theorem~\ref{thm:saturation}).  The dichotomy at a prime (Theorem~\ref{thm:dichotomy}), unconditionally, with the prime-power residue lemma of~\cite{KallatShift} supplied by a separate bridge package.  The algebraic identity $[x^{Nq+\rho}]G=\beta_qF_\rho(q)$ and the resulting valuation formula of Theorem~\ref{thm:interp}.  The Harmonic Theorem, parts (a) to (d), with the plateau formulas of Corollary~\ref{cor:plateau}.  Theorems~\ref{thm:p2r4} and~\ref{thm:p2r5} for $p=2$ and $r=4,5$, and Theorem~\ref{thm:p3} for $p=3$ and $r=3,4,5$, which give Bala's statements for \OEIS{A057693}, \OEIS{A070945} and \OEIS{A070946} (Corollaries~\ref{cor:bala} and~\ref{cor:bala2}).  The case of composite $n+1$ in Theorem~\ref{thm:A335088}.  The cycle gcd theorems, in the form $(n-1)\mid C_\varphi(n)$, $(n^2-1)\mid C(n,g)$ for $n\ge4$ and $n/g\ge3$, and the valuation identities of Theorem~\ref{thm:cyclegcd}(d).  The gcd over the family $\{1,\dots,r,k\}$ for every $r$, with the sandwich $\fall{n-1}{r}\mid g_r(n)\mid G_r(n)$ (Theorem~\ref{thm:gcdfamily} and Proposition~\ref{prop:gcdq}), and the theorem on \OEIS{A333726} (Theorem~\ref{thm:A333726}).  The finiteness of the essential set and the fact that $\kappa$, $\mathrm{Max}$ and $\Xi$ are determined by it (Lemma~\ref{lem:essential}).  The second-order theorem for a dominant block of depth one in the case $W=1$ (Appendix~\ref{app:second}) and the theorem on column $2$ at odd primes (Theorem~\ref{thm:col2}).  Finally there are kernel-checked certificates for the counterexample $g_3(111)\ne G_3(111)$ of Theorem~\ref{thm:counter} and for $709\mid\operatorname{cop}(708)$ of Proposition~\ref{prop:cop708}.  The other results are proved on paper.  Among them are the analytic parts of the interpolation theorem and Theorem~\ref{thm:qstar}, the second-order theorem for a twist $W\ne1$, the Artin--Hasse bound, Theorem~\ref{thm:bad}, Proposition~\ref{prop:cop2q}, and the results of Section~\ref{sec:composite} on composite moduli.

\section*{Acknowledgments}
The author thanks the contributors to the OEIS for the entries and conjectures treated here.  The author used AI tools (large language models, in several different systems) for exploratory discussion, for finding and checking proofs and computations, for the Lean 4 formalization, and for drafting and revising the text. The author independently checked the mathematical arguments, computations, and the final manuscript.

\clearpage
\appendix
\renewcommand{\thesection}{A.\arabic{section}}
\section*{Appendix}\phantomsection\addcontentsline{toc}{section}{Appendix}

Sections~\ref{app:AH} to~\ref{app:density} contain the Artin--Hasse bound and the proofs of Theorem~\ref{thm:A378}(f), of the second-order theorem for column $2$, of Proposition~\ref{prop:123}, of Proposition~\ref{prop:globalgcd}, of the interpolation theorem and its corollaries, of the theorem on the $p$-adic integer $q^*$, of the formulas for $\operatorname{cop}(n)$ modulo a prime $n+1$, and of the statements about the good primes in column $2$.  Section~\ref{app:lean} describes the Lean formalization, Section~\ref{app:moved} contains the proofs of the other results whose statements are in the main text, in the order of the paper, and Section~\ref{app:data} contains the data and the computational details behind several observations.  The notation is that of the main text.

\section{The Artin--Hasse bound}\label{app:AH}

For an integer weight sequence $(w_m)$ let $a(n)$ be defined by $\sum_na(n)x^n/n!=\exp\bigl(\sum_{m\ge1}w_mx^m/m\bigr)$, with $a_S$ the case $w_m=[m\notin S]$.  The Artin--Hasse exponential gives the following lower bound, which is far from exact in general and is replaced in the paper by the valuation calculus.

\begin{theorem}[Valuation bound]\label{thm:valuation}
Let $p$ be a prime and let $w_1=w_p=\eta$.  Put $w'_m=w_m-\eta$ if $m$ is a power of $p$ and $w'_m=w_m$ otherwise (so $w'_1=w'_p=0$), let $J=\{m:w'_m\ne0\}$, and
\[
\theta=\max_{m\in J}\frac{v_p(m)+\frac1{p-1}}{m}\qquad(\theta=0\text{ if }J=\emptyset).
\]
Then $v_p\bigl(a(n)\bigr)\ge v_p(n!)-\theta\,n$ for all $n\ge0$.
\end{theorem}

\begin{lemma}[Artin--Hasse]\label{lem:AH}
The series $E_p(x)\coloneqq\exp\Bigl(\sum_{i\ge0}x^{p^i}/p^i\Bigr)$ lies in $\Z_{(p)}[[x]]$.
\end{lemma}

\begin{proof}
Let $\mu$ be the M\"obius function.  Then
\[
\prod_{p\nmid m}(1-x^m)^{-\mu(m)/m}=\exp\Bigl(\sum_{p\nmid m}\frac{\mu(m)}m\sum_{k\ge1}\frac{x^{mk}}k\Bigr),
\]
and the coefficient of $x^N$ in the exponent is $\frac1N\sum_{m\mid N,\ p\nmid m}\mu(m)$, which is $1/N$ if $N$ is a power of $p$ and $0$ otherwise.  So the product equals $E_p(x)$.\footnote{This is the classical Artin--Hasse exponential~\cite{ArtinHasse}.}  Each factor lies in $\Z_{(p)}[[x]]$ by Lemma~\ref{lem:binom}, as $\mu(m)/m\in\Z_{(p)}$, and the product converges $x$-adically since the factor for $m$ is $1+O(x^m)$.
\end{proof}

\begin{proof}[Proof of Theorem~\ref{thm:valuation}]
Let $u_m=1$ if $m$ is a power of $p$ (including $m=1$) and $u_m=0$ otherwise, so that $w'_m=w_m-\eta u_m$.  Then
\[
\exp\Bigl(\sum_mw_m\frac{x^m}m\Bigr)=E_p(x)^\eta\prod_{m\in J}\exp\Bigl(w'_m\frac{x^m}m\Bigr).
\]
By Lemma~\ref{lem:AH}, $E_p^\eta\in\Z_{(p)}[[x]]$ (for $\eta<0$ use $E_p^{-1}=\exp(-\cdots)$, which is also in $\Z_{(p)}[[x]]$ since $E_p$ has constant term $1$).  The coefficient of $x^{ml}$ in $\exp(w'_mx^m/m)$ is $w_m'^{\,l}/(m^l\,l!)$, with $p$-adic valuation at least $-\bigl(l\,v_p(m)+v_p(l!)\bigr)\ge-l\bigl(v_p(m)+\tfrac1{p-1}\bigr)\ge-\theta\,ml$, using $v_p(l!)\le(l-1)/(p-1)$.  A coefficient of $x^N$ in the product over $J$ is a sum of products of such coefficients with $\sum_m m\,l_m=N$, so it has valuation at least $-\theta N$.  Multiplying by $E_p^\eta$, which has $p$-integral coefficients, does not lower this bound.  Therefore $v_p\bigl(\coeff{n}{\exp(\sum w_mx^m/m)}\bigr)\ge-\theta n$, and multiplying by $n!$ gives the claim.  This completes the proof of Theorem~\ref{thm:valuation}.
\end{proof}

The maximum defining $\theta$ exists because $(v_p(m)+\frac1{p-1})/m\to0$.  For $S=\{m:m>3\}$ and $p=2$ we get $J=\{3,4,8,16,\dots\}$ and $\theta=\frac34$ (attained at $m=4$), while $S=\{m:m>2\}$ gives $\theta=\frac34$ as well and $S=\{m:m>4\}$ gives $\theta=\frac12$.  For $p=3$ and $S=\{m:m>3\}$, $\theta=\frac5{18}$ (at $m=9$).  In the case of cycle lengths at most $3$ the theorem gives $v_2(A_3(n))\ge n/4-s_2(n)$, so the slope of Corollary~\ref{cor:plateau}(a) is the one forced by the Artin--Hasse bound.

\section{Proof of Theorem~\ref{thm:A378}(f)}\label{app:f}

\begin{proof}[Proof of part (f)]
Write $v=v_2$, $H(m)$ for the statement of (f) for a given $m$, and $s=v(m+1)$.  In the base case $H(0)$ the numbers $a_3,\dots,a_6=2,6,24,160$ have valuations $1,1,3,5$, as required.

Let $m\ge1$ and assume $H(m-1)$, so that $v(a_{4m})=m$, $v(a_{4m+1})=m+2+v(m)$, $v(a_{4m+2})=m+4+v(m)$.  The integers $m$ and $m+1$ are consecutive, so at most one of $v(m),s$ is positive.  We use (\ref{eq:R2}) in the form $a_{j+1}=ja_j+j(j-1)a_{j-2}$.

\emph{Step 1.} $a_{4m+3}=(4m+2)a_{4m+2}+(4m+2)(4m+1)a_{4m}$.  The two terms have valuations $1+(m+4+v(m))$ and $1+m$.  The second is smaller, so $v(a_{4m+3})=m+1$.

\emph{Step 2.} $a_{4m+4}=(4m+3)a_{4m+3}+(4m+3)(4m+2)a_{4m+1}$, with valuations $m+1$ and $1+m+2+v(m)$.  So $v(a_{4m+4})=m+1$.

\emph{Step 3.} $a_{4m+5}=(4m+4)a_{4m+4}+(4m+4)(4m+3)a_{4m+2}$.  Since $v(4m+4)=2+s$, the valuations are $m+3+s$ and $m+6+s+v(m)$.  So $v(a_{4m+5})=m+3+s$.

\emph{Step 4 (the tie).} $a_{4m+6}=(4m+5)\bigl[a_{4m+5}+(4m+4)a_{4m+3}\bigr]$, and the two terms in the bracket both have valuation $m+3+s$.  Insert the expression of Step 3,
\begin{equation}\label{eq:tie}
a_{4m+5}+(4m+4)a_{4m+3}=(4m+4)\Bigl[\,a_{4m+4}+a_{4m+3}+(4m+3)a_{4m+2}\Bigr].
\end{equation}
We claim $v(a_{4m+4}+a_{4m+3})=m+3$.  By (\ref{eq:R2}), $a_{4m+4}+a_{4m+3}=(4m+4)a_{4m+3}+(4m+3)(4m+2)a_{4m+1}$, whose terms have valuations $m+3+s$ and $m+3+v(m)$.  If $m$ is odd, then $s\ge1$ and $v(m)=0$, so the second term wins, with valuation $m+3$.  If $m$ is even, then $s=0$ and $v(m)\ge1$, so the first term wins, with valuation $m+3$.  Now $v\bigl((4m+3)a_{4m+2}\bigr)=m+4+v(m)>m+3$, so the bracket in \eqref{eq:tie} has valuation exactly $m+3$.  Therefore $v(a_{4m+5}+(4m+4)a_{4m+3})=(2+s)+(m+3)=m+5+s$, and as $4m+5$ is odd, $v(a_{4m+6})=m+5+s$.

Steps 1--4 give $H(m)$.  This completes the proof of Theorem~\ref{thm:A378}(f).
\end{proof}

\section{Second order at an odd prime}\label{app:second}

\begin{theorem}[Second order for a dominant block of depth one]\label{thm:second}
Let $G=F_AW$ and suppose that $p\in\mathcal B$ with $j_p=1$ and that every other block $m$ satisfies $j_m\le\lfloor m/p\rfloor-2$. Write $c_p=\alpha'/p$ and let $\alpha$ be the residue of $\alpha'$. Fix $\nu\in\{0,\dots,p-1\}$ with $v_p([x^\nu]G)=1$, put $\xi=[x^\nu]G/p$ and $\zeta=[x^{\nu+p}]G-\alpha'\xi$, and let $k=\nu+pq$ with $q\ge1$. Then $\zeta\in\Z_{(p)}$, and
\begin{enumerate}[label=(\alph*),itemsep=0pt]
\item if $p\mid q$ or $\zeta\equiv0\pmod p$, then $v_p([x^k]G)=1-q-v_p(q!)$,
\item if $p\nmid q$ and $\zeta\not\equiv0$, then $v_p([x^k]G)=1-q-v_p(q!)$ when $q\not\equiv-\alpha \xi/\zeta\pmod p$, and $v_p([x^k]G)\ge2-q-v_p(q!)$ when $q\equiv-\alpha \xi/\zeta\pmod p$.
\end{enumerate}
\end{theorem}

\begin{proof}
Every other block has $\lfloor m/p\rfloor\ge3$, since $j_m\ge1$, so $m\ge3p$ and $p=\min\mathcal B$. Put $\Gamma(t)=t+v_p(t!)$, so $\Gamma(t+1)-\Gamma(t)=1+v_p(t+1)$. Since $G=I_W(1-x^p)^{-c_p}\prod_{m\ge3p}(1-x^m)^{-c_m}$, one has $G(1-x^p)^{c_p}\equiv I_W\pmod{x^{3p}}$, and comparing coefficients at $x^\nu$ and $x^{\nu+p}$ (note $\nu+p<2p$) gives $[x^\nu]I_W=[x^\nu]G=p\xi$ and $[x^{\nu+p}]I_W=[x^{\nu+p}]G-c_p[x^\nu]G=\zeta$. So $\zeta\in\Z_{(p)}$.

Take a block vector with $b$ blocks $p$ and other blocks $(d_i,\lambda_i)$, of total degree at most $k$. If some $\lambda_i>0$, put $t=b+\sum_i\lambda_i$ and $Q'=b+\sum_i\lambda_i\lfloor d_i/p\rfloor\le q$. Using $\lfloor d_i/p\rfloor\ge3$ we get $t\le Q'-2\sum_i\lambda_i$, and using $j_i\le\lfloor d_i/p\rfloor-2$,
\[
\mathrm{cost}\le b+\sum_i\lambda_i\bigl(\lfloor d_i/p\rfloor-2\bigr)+v_p(t!)\le\Gamma\bigl(q-2\textstyle\sum_i\lambda_i\bigr)\le\Gamma(q)-2.
\]
A vector $be_p$ with $b\le q-2$ has cost $\Gamma(b)\le\Gamma(q)-2$. Hence, by Lemma~\ref{lem:blockexp}, modulo terms of valuation at least $2-\Gamma(q)$ only $qe_p$ and $(q-1)e_p$ contribute, with leftovers $\nu$ and $\nu+p$, and
\[
[x^k]G\equiv p\xi\binom{c_p+q-1}{q}+\zeta\binom{c_p+q-2}{q-1}\pmod{p^{2-\Gamma(q)}\Z_{(p)}} .
\]
Since $\binom{c_p+q-1}{q}=\frac{c_p+q-1}{q}\binom{c_p+q-2}{q-1}$ and $c_p+q-1=(\alpha'+p(q-1))/p$, for $p\nmid q$ the right side equals $\binom{c_p+q-2}{q-1}\bigl(\xi(\alpha'+p(q-1))/q+\zeta\bigr)$. The binomial factor has valuation exactly $-\Gamma(q-1)=-\Gamma(q)+1$ when $p\nmid q$, and the bracket is congruent to $\xi\alpha/q+\zeta$ modulo $p$. This vanishes exactly when $q\equiv-\alpha \xi/\zeta$, and it never vanishes if $\zeta\equiv0$, which gives (b) and the case $\zeta\equiv0$ of (a). When $p\mid q$ the first term has valuation exactly $1-\Gamma(q)$ while the second has valuation at least $-\Gamma(q-1)=-\Gamma(q)+1+v_p(q)\ge2-\Gamma(q)$.
\end{proof}

\begin{proof}[Proof of Theorem~\ref{thm:col2}]
The blocks satisfy the hypothesis of Theorem~\ref{thm:second}, since every other block $pd$ has $d\ge3$ and depth $1\le d-2$, and $\alpha=\alpha'=1$. Multiply $[x^k]G=Q_{k+3}/k!$ by $k!$ and use $v_p(k!)=q+v_p(q!)$.
\end{proof}

\section{Proof of Proposition~\ref{prop:123}}\label{app:123}

\begin{proof}
By (\ref{eq:R3}), $z(n+1)=n\,z(n)+n(n-1)(n-2)\,z(n-3)$, and $z(0)=1$, $z(1)=z(2)=z(3)=0$, $z(4)=6$, $z(5)=24$, $z(6)=120$, $z(7)=720$.  Write $v=v_2$ and $s=v(m+1)$.  For $n=4m+4,4m+5,4m+6$ one computes $\beta(n)=m+1,\ m+3+s,\ m+3+s$.  We prove by induction on $m\ge0$ the statement $H(m)$ that
\[
v(z(4m+4))=m+1,\quad v(z(4m+5))=v(z(4m+6))=m+3+s,\quad v(z(4m+7))\ge m+3 .
\]
For $m=0$ the values $6,24,120,720$ have valuations $1,3,3,4$.  Let $m\ge1$ and assume $H(m-1)$, which says that $v(z(4m))=m$, $v(z(4m+1))=m+2+v(m)$, $v(z(4m+2))=m+2+v(m)$ and $v(z(4m+3))\ge m+2$.  Apply the recurrence four times.
\begin{itemize}[itemsep=0pt]
\item $z(4m+4)=(4m+3)z(4m+3)+(4m+3)(4m+2)(4m+1)z(4m)$.  The two terms have valuations at least $m+2$ and exactly $1+m$, so $v(z(4m+4))=m+1$.
\item $z(4m+5)=(4m+4)z(4m+4)+(4m+4)(4m+3)(4m+2)z(4m+1)$.  The terms have valuations $(2+s)+(m+1)=m+3+s$ and $(2+s)+1+(m+2+v(m))>m+3+s$.
\item $z(4m+6)=(4m+5)z(4m+5)+(4m+5)(4m+4)(4m+3)z(4m+2)$.  The terms have valuations $m+3+s$ and $(2+s)+(m+2+v(m))>m+3+s$.
\item $z(4m+7)=(4m+6)z(4m+6)+(4m+6)(4m+5)(4m+4)z(4m+3)$.  The terms have valuations at least $1+m+3+s\ge m+4$ and $1+(2+s)+(m+2)\ge m+5$.
\end{itemize}
This is $H(m)$.  Since the exponent in the first three cases agrees with $\beta(n)$, the proposition follows.
\end{proof}

\section{Proof of Proposition~\ref{prop:globalgcd}}\label{app:gg}

\begin{proof}
(i) For $S\supseteq\{1,\dots,r\}$ and $A=\N\setminus S$ one has $a_S(n)=\sum_{P\subseteq A\cap(r,n]}f_P(n)$, since the set of distinct cycle lengths of a permutation lies in $A$.  Conversely, for $P_0\subseteq(r,n]$ taking $A=P$ over the subsets of $P_0$ and inverting over the Boolean lattice gives $f_{P_0}=\sum_{P\subseteq P_0}(-1)^{|P_0\setminus P|}a_{\N\setminus P}(n)$, so every $f_{P_0}$ is an integer combination of the numbers $a_S(n)$.  Hence the two gcds agree.  Each $f_P$ is a sum of numbers $N(\lambda)$, so $\Lambda_r\mid g_r$.

(ii) This is Legendre's formula for $n!/\prod m^{\lambda_m}\lambda_m!$.

(iii) Each $f_P$ is a sum of numbers $N(\lambda)$, each of valuation at least $\mu_0$, and the terms of valuation larger than $\mu_0$ do not affect divisibility by $p^{\mu_0+1}$.  So $v_p(f_P)>\mu_0$ if and only if the sum of the $N(\lambda)$ with $v_p(N(\lambda))=\mu_0$ and part set $P$ is divisible by $p^{\mu_0+1}$, which holds trivially when there is no such $\lambda$.  By (i), $v_p(g_r(n))=\min_Pv_p(f_P(n))$, which exceeds $\mu_0$ if and only if $v_p(f_P)>\mu_0$ for every $P$.

(iv) By Theorem~\ref{thm:gcdfamily}, $g_r(n)$ divides $G_r(n)$ and the quotient $G_r(n)/\fall{n-1}{r}$ has only prime factors at most $r$, while $\fall{n-1}{r}$ divides $g_r(n)$.  So the two gcds agree if and only if they have the same valuation at every prime $p\le r$.  Since $v_p(g_r(n))=\min_Pv_p(f_P(n))$ by (i) and $v_p(g_r(n))\le v_p(G_r(n))$, equality at $p$ means that $v_p(f_P(n))\ge T:=v_p(G_r(n))$ for every $P$.  Finally $f_P(n)$ is a sum of numbers $N(\lambda)$, and the terms with $v_p(N(\lambda))\ge T$ do not affect divisibility by $p^T$.
\end{proof}

\section{Proof of Theorem~\ref{thm:interp} and its corollaries}\label{app:interp}

We use the notation of Section~\ref{sec:interp} and write $m_*$, $j$, $c_*=c_{m_*}=\alpha'/p^j$, $\beta_q$, $\Gamma$, $h_i$ and $F_\rho$ as there, and $H=\sum_ih_ix^i$.

\begin{lemma}[Identity]\label{lem:ident}
For every $k=m_*q+\rho$ with $0\le\rho<m_*$, $[x^k]G=\sum_{t=0}^{q}h_{\rho+tm_*}\beta_{q-t}$, and $\beta_{q-t}/\beta_q=\prod_{i<t}p^j(q-i)/(\alpha'+p^j(q-1-i))$.  Also $v_p(\beta_q)=-\Gamma(q)$.  Hence $[x^k]G=\beta_qF_\rho(q)$.
\end{lemma}

\begin{proof}
We have $G=H(1-x^{m_*})^{-c_*}$ and $[y^i](1-y)^{-c_*}=\beta_i$.  From $\beta_q=\beta_{q-1}(c_*+q-1)/q$ we get $\beta_{q-1}/\beta_q=q/(c_*+q-1)$, and $c_*+q-1-i=(\alpha'+p^j(q-1-i))/p^j\ne0$.  The product telescopes, and the factor $i=q$ vanishes, which cuts the sum at $t=q$.  Each $c_*+i$ has valuation exactly $-j$, so $v_p(\beta_q)=-jq-v_p(q!)$.
\end{proof}

\begin{lemma}[Size of $h$]\label{lem:hsize}
Let $\varepsilon'=\max_{m\in\mathcal B\setminus\{m_*\}}\varepsilon_m$, with $\varepsilon'=0$ if there is no other block.  Then $\varepsilon'<\varepsilon_{m_*}$ and $v_p(h_i)\ge-\varepsilon'i$.  More precisely $v_p(h_i)\ge-\max_V\mathrm{cost}(V)$ over block vectors $V$ on $\mathcal B\setminus\{m_*\}$ of degree at most $i$.
\end{lemma}

\begin{proof}
We have $H=I_W\prod_{m\in\mathcal B\setminus\{m_*\}}(1-x^m)^{-c_m}$ with $I_W\in\Z_{(p)}[[x]]$.  The coefficient of $x^i$ is a sum over block vectors $V$ of degree at most $i$ of terms of valuation at least $-\mathrm{cost}(V)$, and $\mathrm{cost}(V)=\sum_m(j_m\lambda_m+v_p(\lambda_m!))\le\sum_m\varepsilon_mm\lambda_m\le\varepsilon'\deg V$ because $v_p(\lambda!)\le\lambda/(p-1)$.  The maximum exists because $\varepsilon_m\le(\log_pm+1)/m\to0$.
\end{proof}

\begin{remark}\label{rem:uniq}
Uniqueness of the most efficient block among the essential blocks $\mathcal E$ is the same as uniqueness among all blocks.  Let $m_0=\min\mathcal B$.  If $m\notin\mathcal E$ then $\lfloor m/m_0\rfloor>j_m$, so $\varepsilon_m\le(\lfloor m/m_0\rfloor-1+\frac1{p-1})/m<\lfloor m/m_0\rfloor/m\le1/m_0<\varepsilon_{m_0}$, and $m_0\in\mathcal E$.
\end{remark}

\begin{proof}[Proof of Theorem~\ref{thm:interp}]
The identity and the value of $v_p(\beta_q)$ are Lemma~\ref{lem:ident}.  Put $\theta=(\varepsilon_{m_*}-\varepsilon')m_*>0$ and choose $\ell$ with $p^\ell(p-1)\theta>2$.  Fix $z_0\in\Z_p$ and write $z=z_0+p^\ell y$.  The $t$-th term of $F_\rho$ is $h_{\rho+tm_*}T_t(y)$ with
\[
T_t=\prod_{i<t}\frac{p^j(b_i+p^\ell y)}{u_i+p^{j+\ell}y},\qquad b_i=z_0-i,\quad u_i=\alpha'+p^j(b_i-1)\in\Z_p^\times .
\]
Each denominator is a unit, and $1/(u_i+p^{j+\ell}y)\in\Z_p\langle y\rangle$ has Gauss valuation $0$.  The Gauss valuation $w$ is additive, so
\[
w(T_t)=jt+\sum_{i<t}\min(v(b_i),\ell)\ \ge\ jt+\sum_{s=1}^{\ell}\lfloor t/p^s\rfloor\ \ge\ jt+v_p(t!)-\frac{t}{p^\ell(p-1)},
\]
since among $t$ consecutive integers at least $\lfloor t/p^s\rfloor$ are divisible by $p^s$, and the dropped tail of Legendre's sum is at most $t/(p^\ell(p-1))$.  With $v_p(t!)=(t-s_p(t))/(p-1)\ge t/(p-1)-\log_pt-1$ and $e_*=j+\frac1{p-1}=\varepsilon_{m_*}m_*$, Lemma~\ref{lem:hsize} gives
\[
w(h_{\rho+tm_*}T_t)\ \ge\ -\varepsilon'(\rho+tm_*)+e_*t-\frac{t}{p^\ell(p-1)}-\log_pt-1\ \ge\ \tfrac\theta2t-\log_pt-C_\rho\to\infty .
\]
So $F_\rho(z_0+p^\ell y)$ is a Gauss-norm limit of restricted series, hence a restricted power series in $y$.  By compactness finitely many discs $z_0+p^\ell\Z_p$ cover $\Z_p$.
\end{proof}

\begin{proof}[Proof of Corollary~\ref{cor:strassmann}]
Strassmann's theorem on each disc gives finitely many zeros of finite multiplicity.  Weierstrass preparation on $D$ writes $F_\rho=(z-z^*)^\mu U$ near $z^*$ with $U$ a unit on a smaller disc around $z^*$, which gives the formula, and $v_p(F_\rho)$ is bounded on the compact complement of a neighbourhood of the zeros by continuity.  If $z^*\in\Z_p$ is a zero, then $q\to z^*$ along $\N$ gives unbounded values.
\end{proof}

\begin{proof}[Proof of Corollary~\ref{cor:dich}]
If $g\nmid n$ then no $d\in A$ divides $n$, so $c_n=0$.  Hence every block, and $m_*$, is a multiple of $g$, and a class modulo $m_*$ lies entirely inside or entirely outside $g\N$.  If $g\nmid\rho$ then $a_A(n)=0$ on the class, so $F_\rho(q)=0$ for all $q\in\N$, and $F_\rho\equiv0$ by density and continuity.  If $g\mid\rho$, the semigroup $\langle A\rangle$ contains all large multiples of $g$, by taking a finite subset of $A$ with the same gcd, and $a_A(n)\ge1$ for $n\in\langle A\rangle$.  So $F_\rho(q)\ne0$ for all large $q$.  Every disc contains infinitely many such $q$, so $F_\rho$ vanishes identically on no disc, and Corollary~\ref{cor:strassmann} gives finitely many zeros on each of the finitely many discs.
\end{proof}

\section{Proof of Theorem~\ref{thm:qstar}}\label{app:qstar}

Here $G=(1-x)^{-2}e^{-x-x^2/2}$, $m_*=p$ is odd, $j=1$, $\alpha'=1$, and every other block has depth $1$ and degree at least $3p$.  Let $\rho=\nu<p$, so that $h_\rho=[x^\rho]G=Q_{\rho+3}/\rho!=p\xi$ with $\xi$ a $p$-adic unit, and $h_{\rho+p}=\zeta$ in the notation of Theorem~\ref{thm:col2}, since $\zeta=Q_{\rho+p+3}/(\rho+p)!-\xi=[x^{\rho+p}]G-c_p[x^\rho]G$.

\begin{proof}
For a vector $V$ on the other blocks, all of depth $1$, $\mathrm{cost}(V)=\sum_m(\lambda_m+v_p(\lambda_m!))\le\frac p{p-1}\sum_m\lambda_m\le\frac p{p-1}\lfloor\deg V/3p\rfloor$.  By Lemma~\ref{lem:hsize}, $v_p(h_{\rho+tp})\ge-\frac p{p-1}\lfloor t/3\rfloor$ for $\rho<p$.  (The bound $-\lfloor t/3\rfloor$ is false, for example $v_5(h_{\rho+75})=-6$.)  On the full disc $\Z_p$ the Gauss valuation of $T_t=\prod_{i<t}p(z-i)/(1+p(z-1-i))$ is exactly $t$, since $p^t\prod(z-i)$ is $p^t$ times a monic polynomial and the denominators are units of $\Z_p\langle z\rangle$.  As $p\ge3$ gives $p/(p-1)\le\frac32$,
\[
w(h_{\rho+tp}T_t)\ \ge\ t-\tfrac p{p-1}\cdot\tfrac t3\ \ge\ \tfrac t2 .
\]
So the series for $F_\rho$ converges in the Tate algebra $\Z_p\langle z\rangle$.  We examine the terms $t=0,1,2$ and $t\ge3$.  The term $t=0$ is $h_\rho=p\xi$.  Since $T_1=pz/(1+p(z-1))$, the term $t=1$ divided by $p$ is $\zeta z/(1+p(z-1))\equiv\zeta z$ modulo $p\Z_p\langle z\rangle$.  For $t=2$, $h_{\rho+2p}\in\Z_p$ because no other block has degree at most $\rho+2p<3p$, and $w(T_2)=2$.  For $t=3$ the cost bound is $p/(p-1)\le\frac32$ per block, so $w(h_{\rho+3p}T_3)\ge3-\frac p{p-1}\ge\frac32$.  Gauss valuations in $\Z_p\langle z\rangle$ are integers, so rounding $w\ge\frac32$ up gives $w\ge2$.  For $t\ge4$, $w\ge t/2\ge2$ directly.  Hence
\[
F_\rho/p=\xi+\zeta z+pR(z),\qquad R\in\Z_p\langle z\rangle .
\]
If $\zeta$ is a unit, this is congruent modulo $p$ to a polynomial of degree one with unit slope, and Weierstrass preparation in $\Z_p\langle z\rangle$, in its general form with a monic polynomial of degree one that has reduction $z+\xi/\zeta$, gives $F_\rho/p=(z-q^*)U(z)$ with $U$ a unit of $\Z_p\langle z\rangle$ and a unique root $q^*\in\Z_p$ with $q^*\equiv-\xi/\zeta\pmod p$.  A unit of $\Z_p\langle z\rangle$ has constant non-zero reduction, so it has absolute value $1$ at every point of $\Z_p$, and $v_p(F_\rho(z))=1+v_p(z-q^*)$ for all $z\in\Z_p$.  If $p\mid\zeta$, then $F_\rho/p\equiv\xi\not\equiv0\pmod p$ is a unit at every point, so $v_p(F_\rho(z))=1$ for all $z\in\Z_p$.

Now let $n=\rho+3+pq$ with $q\in\N$ and $k=\rho+pq$.  Then $Q_n=k!\,[x^k]G$ and $v_p(k!)=\Gamma(q)=q+v_p(q!)$, so Theorem~\ref{thm:interp} gives $v_p(Q_n)=v_p(F_\rho(q))$.  For $n\ge3$ the number $Q_n=T(n,2)/((n-1)(n-2))$ is positive, because an $n$-cycle has no cycle of length $1$ or $2$ and so $T(n,2)\ge1$.  Hence $v_p(Q_n)$ is finite, which excludes $q=q^*$, and $q^*\notin\N$.
\end{proof}

\section{Proofs for the residue of \texorpdfstring{$\operatorname{cop}(n)$}{cop(n)} modulo a prime \texorpdfstring{$n+1$}{n+1}}\label{app:cop}

Let $N=n+1$ be an odd prime.  By Proposition~\ref{prop:cop}, $\operatorname{cop}(n)\equiv-[x^n]\prod_{d\mid n}(1-x^d)^{\mu(d)n/d}\pmod N$, the product over the squarefree divisors of $n$.

\begin{proof}[Proof of Proposition~\ref{prop:copprefix}]
The factor $d=1$ is $(1-x)^n=(1-x)^N/(1-x)$, and $(1-x)^N\equiv1-x^N\pmod N$.  So modulo $(N,x^N)$ this factor is $1/(1-x)$.  Multiplying a power series by $1/(1-x)$ replaces its coefficients by their prefix sums, and $n<N$, so $[x^n]\prod_{d\mid n}(1-x^d)^{\mu(d)n/d}\equiv g_0+\dots+g_n$.  For the recurrence, $xG'/G=-n\sum_{d>1}\mu(d)\sum_{i\ge1}x^{di}$ and $n\equiv-1\pmod N$, so $kg_k\equiv\sum_{d>1}\mu(d)\sum_{i\ge1}g_{k-di}$, and all divisions are by integers $k\le n<N$.
\end{proof}

\begin{proof}[Proof of Proposition~\ref{prop:copfermat}]
For $n=2^a$ the only squarefree divisor above $1$ is $2$, so $G=(1-x^2)^{-n/2}$ and $\sum_{k\le n}g_k=\sum_{i=0}^{n/2}\binom{n/2-1+i}{i}=\binom{n}{n/2}$ by the hockey stick identity.  Finally $\binom{N-1}{i}\equiv(-1)^i\pmod N$.  The case $N=3$ is $n=2$.
\end{proof}

\begin{proof}[Proof of Proposition~\ref{prop:copsafe}]
For $n=2q$ the squarefree divisors above $1$ are $2$, $q$ and $2q$, so $G=(1-x^2)^{-q}(1-x^q)^{-2}(1-x^{2q})=A(x)\frac{1+x^q}{1-x^q}$ with $A=(1-x^2)^{-q}$.  Modulo $x^{2q+1}$ this is $A\,(1+2x^q+2x^{2q})$.  Let $P(m)=\sum_{k\le m}[x^k]A=\sum_{i\le\lfloor m/2\rfloor}\binom{q-1+i}{i}=\binom{q+\lfloor m/2\rfloor}{\lfloor m/2\rfloor}$.  Then $\sum_{k\le n}g_k=P(2q)+2P(q)+2P(0)=\binom{2q}{q}+2\binom{(3q-1)/2}{(q-1)/2}+2$, and $\binom{2q}{q}=\binom{N-1}{q}\equiv(-1)^q=-1$.  Hence $\operatorname{cop}(n)\equiv-(-1+2\binom{(3q-1)/2}{(q-1)/2}+2)=-1-2\binom{(3q-1)/2}{(q-1)/2}$.
\end{proof}

\begin{proof}[Proof of Proposition~\ref{prop:cop2q}]
By Proposition~\ref{prop:copprefix}, $\operatorname{cop}(n)\equiv-[x^n]\,G(x)/(1-x)\pmod N$.  For $n=2^aq^b$ the squarefree divisors above $1$ are $2$, $q$ and $2q$, with $\mu=-1,-1,1$, so with $e=n/(2q)$, an integer since $a,b\ge1$,
\[
G=(1-x^2)^{-n/2}(1-x^q)^{-2e}(1-x^{2q})^{e}=(1-x^2)^{-n/2}\Bigl(\frac{1+x^q}{1-x^q}\Bigr)^{e}.
\]
The coefficient of $x^{2i}$ in $(1-x^2)^{-c}$ is $\binom{c+i-1}{i}$, a polynomial in $c$ whose denominator $i!$ is prime to $N$ for $2i\le n$, and $n/2\equiv-1/2\pmod N$ because $n\equiv-1$.  So modulo $(N,x^{N})$ the factor $(1-x^2)^{-n/2}$ may be replaced by $(1-x^2)^{1/2}$.  Finally $(1-x^2)^{1/2}/(1-x)=(1+x)(1-x^2)^{-1/2}$, and the coefficient of $x^m$ in it is $\binom{2h}{h}4^{-h}=\vartheta(m)$ with $h=\lfloor m/2\rfloor$.  Expanding $((1+x^q)/(1-x^q))^e=\sum_j\gamma_jx^{qj}$ gives $[x^n]G/(1-x)\equiv\sum_{j\le n/q}\gamma_j\vartheta(n-qj)$.
\end{proof}

The script \texttt{cop\_twoprime.py} in \texttt{anc/interp/windows} compares the formula of Proposition~\ref{prop:cop2q} with the recurrence of Proposition~\ref{prop:cop} for the $454$ primes $N<20000$ of the form $2^aq^b+1$ and with exact values of $\operatorname{cop}(n)$ for $N<400$.

The scripts are in \texttt{anc/interp/cop}.  The program \texttt{cop\_fast.c} runs the recurrence of Proposition~\ref{prop:copprefix}, \texttt{cop\_cycle.c} runs the cycle recurrence modulo $N$ without the M\"obius step, and \texttt{safe\_prime.c} evaluates the formula for safe primes.  The prefix-sum program agrees with exact big-integer values of $\operatorname{cop}(n)$ for all $167$ primes $N\le1000$, including $N=709$, with the cycle recurrence for all $2261$ odd primes $N\le20000$, and the safe-prime formula agrees with it for all $669$ safe primes $N<10^5$.  The Fermat primes $5,17,257,65537$ give the residue $N-1$.  As a negative control, the version with the exponent $n/d+1$ in place of $n/d$ agrees with the exact values for $3$ of $76$ primes $N\le400$.  The search over all primes $N\le3000001$ gives the zeros $709$ and $92203$ only, and the residues of all $9591$ odd primes below $10^5$ are in \texttt{residues\_p\_lt\_1e5.txt}.

\section{Proofs for the good primes in column \texorpdfstring{$2$}{2}}\label{app:density}

Let $Q_n=T(n,2)/((n-1)(n-2))$ for $n\ge3$ and $q_\nu=Q_{\nu+3}$, so that $Q_{k+3}=k!\,[x^k]G$ with $G=(1-x)^{-2}e^{-x-x^2/2}$.  An odd prime $p$ is good if $p\nmid Q_n$ for all $n\ge3$, and bad otherwise.

\begin{lemma}\label{lem:gaprec}
$(1-x)G'=(1+x^2)G$.  Hence $(\nu+1)(g_{\nu+1}-g_\nu)=g_{\nu-2}$ for $g_\nu=[x^\nu]G$, and $q_{\nu+1}=(\nu+1)q_\nu+\nu(\nu-1)q_{\nu-2}$ for $\nu\ge2$, with $q_0,q_1,q_2=1,1,2$.
\end{lemma}

\begin{proof}
$G'/G=2/(1-x)-(1+x)$, which gives the differential equation.  Comparing the coefficients of $x^\nu$ gives $(\nu+1)g_{\nu+1}-\nu g_\nu=g_\nu+g_{\nu-2}$.  Multiply by $\nu!$.  For example $q_5=5\cdot38+12\cdot2=214=Q_8$.
\end{proof}

\begin{theorem}[Fixed lines]\label{thm:lines}
Define $u_0,u_1,u_2=e_0,e_1,e_2$ in $\Q^3$ and $u_{m+2}=(u_{m-1}+m\,u_m)/((m+1)(m+2))$ for $m\ge1$.  The vectors $u_m$ do not depend on $p$.  For every odd prime $p$ and $0\le m\le p-3$ the denominators of $u_m$ are prime to $p$ and $q_{p-1-m}\equiv\langle u_m,c(p)\rangle\pmod p$, where $c(p)=(q_{p-1},q_{p-2},q_{p-3})\bmod p$.  Consequently $p$ is bad if and only if $[c(p)]\in\mathbb P^2(\F_p)$ lies on one of the $p-2$ lines $\{u_m^\perp\}$, $0\le m\le p-3$, and the number of $n\in[3,p+2]$ with $p\mid Q_n$ is the number of $m\in[0,p-3]$ for which $[c(p)]$ lies on the line of $u_m$.
\end{theorem}

\begin{proof}
Put $R_m=q_{p-1-m}\bmod p$ and $\nu=p-1-m$.  Lemma~\ref{lem:gaprec} at the index $\nu$ reads $q_{\nu-2}\,\nu(\nu-1)=q_{\nu+1}-(\nu+1)q_\nu$.  Modulo $p$ one has $\nu(\nu-1)\equiv(m+1)(m+2)$ and $\nu+1\equiv-m$, so $R_{m+2}(m+1)(m+2)\equiv R_{m-1}+mR_m$ for $1\le m\le p-3$.  The factor $(m+1)(m+2)$ is a unit because $m+2\le p-1$, so induction on $m$ gives $R_m=\langle u_m,c(p)\rangle$, and the denominators of $u_m$ are products of integers at most $p-1$.  The indices $m=p-2,p-1$ give $q_1,q_0=1,1$, which are non-zero modulo $p$, and $c(p)\ne0$, since $c(p)=0$ would force $R_m=0$ for all $m\le p-3$, whereas $R_{p-3}=q_2=2$.  The values $q_{p-1-m}$ for $0\le m\le p-3$ are $Q_n$ for $5\le n\le p+2$, and $Q_3=Q_4=1$.
\end{proof}

\begin{proposition}[Moments for a random point]\label{prop:moments}
Let $\ell_0,\dots,\ell_{p-3}$ be pairwise distinct lines of $\mathbb P^2(\F_p)$ and let $N(c)$ be the number of them through a point $c$.  If $c$ is uniform on $\mathbb P^2(\F_p)$, then $\mathbb EN=(p-2)(p+1)/(p^2+p+1)$ and $\mathbb EN(N-1)=(p-2)(p-3)/(p^2+p+1)$.
\end{proposition}

\begin{proof}
Every line has $p+1$ points, and two distinct lines meet in exactly one point, and $\mathbb P^2(\F_p)$ has $p^2+p+1$ points.
\end{proof}

\begin{theorem}[Conditional density]\label{thm:cond}
Let $U_p\subset\mathbb P^2(\F_p)$ be the set of points on none of the lines of Theorem~\ref{thm:lines}.  Assume (H1) $|U_p|/(p^2+p+1)\to e^{-1}$ as $p\to\infty$, and (H2) $\#\{p\le x:[c(p)]\in U_p\}=\sum_{p\le x}|U_p|/(p^2+p+1)+o(\pi(x))$.  Then the good primes have density $e^{-1}$ among the primes.
\end{theorem}

\begin{proof}
By Theorem~\ref{thm:lines} a prime $p$ is good if and only if $[c(p)]\in U_p$.  By (H2) the number of good primes up to $x$ is $\sum_{p\le x}|U_p|/(p^2+p+1)+o(\pi(x))$, and by (H1) this is $(e^{-1}+o(1))\pi(x)$.
\end{proof}

\begin{proposition}\label{prop:chebotarev}
Let $f\in\Z[x]$ be irreducible of degree $d$ with Galois group $S_d$.  The set of primes $p$ for which $f$ has no root modulo $p$ has Dirichlet density $\sum_{k=0}^d(-1)^k/k!$, and this differs from $e^{-1}$ by less than $1/(d+1)!$.  For every fixed $K$ the set of odd primes dividing some $Q_n$ with $3\le n\le K$ is finite.
\end{proposition}

\begin{proof}
For $p$ not dividing the discriminant, $f$ has a root modulo $p$ if and only if the Frobenius class of $p$ in $S_d$ has a fixed point, so by the Chebotarev density theorem the density is the proportion $D_d/d!=\sum_{k\le d}(-1)^k/k!$ of permutations without fixed point, and the error bound is the alternating series estimate.  The last claim holds because each $Q_n$ is a positive integer.
\end{proof}

The scripts are in \texttt{anc/interp/density}.  The program \texttt{fast.c} runs the recurrence of Lemma~\ref{lem:gaprec} modulo $p$ for every odd prime up to $3\cdot10^6$, the file \texttt{analyze.out} contains the counts quoted in Question~\ref{prob:odd}, \texttt{arrangement.py} checks the identity of Theorem~\ref{thm:lines} for all odd primes up to $3000$ with no mismatch, \texttt{interp\_degree.py} computes the degree of the interpolating polynomial, and \texttt{controls.py} reproduces the densities $1/2$, $1/3$ and $3/8$ of Proposition~\ref{prop:chebotarev} for $x^2+1$, $x^3-2$ and $x^4+x+1$.  The zero positions of $Q_n$ modulo $p$ for all these primes are in \texttt{zeros\_all.txt.gz}.  Among the $216815$ odd primes up to $3\cdot10^6$ the number $N_p$ of zeros has mean $1.003$ and variance $1.008$ for $p\ge1000$.

\section{The Lean formalization}\label{app:lean}

The folder \texttt{anc/lean} contains the Lean~4 sources, without build products, in three directories.  The directory \texttt{lean-restricted} is the Lake package \texttt{CongruenceRestricted} with the subpackage \texttt{lean-restricted/bridge}, and \texttt{lean-certificates} is the package \texttt{CongruenceCertificates}.  The file \texttt{anc/lean/README.md} gives the toolchain, \texttt{leanprover/lean4:v4.31.0} with Mathlib \texttt{v4.31.0}, and the build commands.  In the packages the headline theorems depend only on \texttt{propext}, \texttt{Classical.choice} and \texttt{Quot.sound}, as printed by \texttt{\#print axioms} in the files \texttt{Main*.lean} and \texttt{Certificates.lean}.  There is no \texttt{sorry} and no \texttt{native\_decide} in the sources.  The recurrence of Section~\ref{sec:tools} is proved to count the permutations of \texttt{Fin n} with the prescribed cycle lengths, in \texttt{PermFinal.lean}.

\begin{center}\small
\begin{tabular}{p{5.3cm}|p{4.7cm}|p{5.0cm}}
statement of the paper & Lean names & file\\\hline
Theorem~\ref{thm:invol} (involutions, $v_2(I(n))=\lfloor n/4\rfloor+\epsilon_\nu$) & \texttt{Inv\_exact} & \texttt{Invol.lean}\\
exactness half of Theorem~\ref{thm:calculus} and Corollary~\ref{cor:calc} & \texttt{calculus\_exact}, \texttt{calculus\_exact\_val}, \texttt{calculus\_exact\_vec} & \texttt{ExactCalculus.lean}, \texttt{VecXi.lean}\\
Theorem~\ref{thm:saturation} & \texttt{saturation\_tfae}, \texttt{saturated\_valuation} & \texttt{Saturation.lean}\\
Theorem~\ref{thm:dichotomy} & \texttt{dichotomy\_A'}, \texttt{dichotomy\_B\_odd\_pos'}, \texttt{dichotomy\_B\_odd\_neg'}, \texttt{dichotomy\_B\_two'}, \texttt{dichotomy\_eventual'} & \texttt{bridge/Bridge/Residues.lean}\\
algebraic part of Theorem~\ref{thm:interp} & \texttt{interp\_identity}, \texttt{interp\_norm}, \texttt{interp\_val} & \texttt{Interp.lean}\\
Theorem~\ref{thm:harmonic}(a)--(d), Corollary~\ref{cor:plateau} & \texttt{harmonic\_a\_bound}, \texttt{harmonic\_b}, \texttt{harmonic\_c\_bound}, \texttt{harmonic\_d}, \texttt{plateau\_closed} & \texttt{Harmonic*.lean}, \texttt{Plateau.lean}\\
Theorems~\ref{thm:p2r4} and~\ref{thm:p2r5} & \texttt{p2r4}, \texttt{p2r5} & \texttt{P2Main.lean}\\
Theorem~\ref{thm:p3}, Corollary~\ref{cor:bala} & \texttt{p3}, \texttt{bala\_A3}, \texttt{bala\_A4}, \texttt{bala\_A5} & \texttt{P3Main.lean}\\
Theorem~\ref{thm:A335088}, composite $n+1$ & \texttt{cop\_dvd\_succ\_of\_composite}, \texttt{cop\_dvd\_pred}, \texttt{cop\_coprime} & \texttt{Cop.lean}\\
Theorem~\ref{thm:cyclegcd} & \texttt{cyclegcd\_a}, \texttt{cyclegcd\_b}, \texttt{cyclegcd\_c}, \texttt{cyclegcd\_d\_dvd}, \texttt{cyclegcd\_d\_ndvd} & \texttt{CycleGcd*.lean}\\
Theorem~\ref{thm:counter}(a), $g_3(111)\ne G_3(111)$ & \texttt{C1\_g\_ne\_G}, \texttt{C1\_G\_value} & \texttt{lean-certificates}\\
Proposition~\ref{prop:cop708}, $709\mid\operatorname{cop}(708)$ & \texttt{C2\_709\_dvd\_cop\_708} & \texttt{lean-certificates}\\
Theorem~\ref{thm:gcdfamily} for every $r$, Proposition~\ref{prop:gcdq} & \texttt{gcdfamily}, \texttt{gcdfamily\_sandwich}, \texttt{gcdfamily\_val}, \texttt{gcdq}, \texttt{gcdq\_val}, \texttt{gcdfamily\_r1}, \texttt{gcdfamily\_r2} & \texttt{GcdFamilyGen.lean}\\
Theorem~\ref{thm:A333726} & \texttt{A333726}, \texttt{A333726\_val}, \texttt{A333726\_perm} & \texttt{A333726.lean}\\
Lemma~\ref{lem:essential}, finiteness of $\mathcal E$ and $\kappa$, $\mathrm{Max}$, $\Xi$ determined by $\mathcal E$ & \texttt{essSet\_finite}, \texttt{essSet\_bound}, \texttt{essential\_determines}, \texttt{calculus\_exact\_ess} & \texttt{Essential.lean}\\
Theorem~\ref{thm:second} for $W=1$ & \texttt{second\_order} & \texttt{SecondOrder.lean}\\
Theorem~\ref{thm:col2} & \texttt{col2}, \texttt{col2\_exc\_iff}, \texttt{col2\_T} & \texttt{Col2.lean}
\end{tabular}
\end{center}

The proofs of some of these statements differ from the proofs in the paper.  The involution theorem is proved by an elementary induction on the recurrence, the three instances of Theorems~\ref{thm:p2r4}, \ref{thm:p2r5} and~\ref{thm:p3} by a factorization of the generating function into a block part and a part with bounded coefficients followed by finite certificates checked by the kernel, and the certificates of \texttt{lean-certificates} by kernel evaluation (\texttt{decide +kernel}).  The statement about all sets $S\supseteq\{1,2,3\}$ in Theorem~\ref{thm:counter} is given through the witness $S=\N\setminus\{4,9\}$, for which $3^{28}\nmid a_S(111)$, together with $3^{28}\mid a_{\{1,2,3,k\}}(111)$ for $4\le k\le112$.

The Lean files cover the items listed above, and the other results are proved on paper in the text and the appendices.  Among them are the analytic parts of Theorem~\ref{thm:interp} and its corollaries (convergence, Strassmann's theorem and Weierstrass preparation, and Theorem~\ref{thm:qstar}), Theorem~\ref{thm:second} for a twist $W\ne1$, the Artin--Hasse bound, Theorem~\ref{thm:bad}, Proposition~\ref{prop:cop2q}, and the results of Section~\ref{sec:composite} on composite moduli.

\section{Proofs of results in the main text}\label{app:moved}

This section contains the proofs of results of the paper whose statements, and the idea of whose proofs, are in the main text.  They are given in the order of the paper, and the notation is that of the corresponding section.

\subsection{Proofs for Section~\ref{sec:tools}}

\begin{proof}[Proof of Lemma~\ref{lem:rec}]
For (\ref{eq:last}), look at the cycle containing the point $n+1$.  If its length is $m\notin S$, there are $n(n-1)\cdots(n-m+2)=\fall{n}{m-1}$ ways to choose and order its other points, and the remaining $n+1-m$ points carry a permutation counted by $a_{n+1-m}$.

For (\ref{eq:rec}), let $A$ be the series (\ref{eq:egf}), so $A'=L'A$ with $L'=\sum_{m\notin S}x^{m-1}=\frac1{1-x}-\sum_{m\in S}x^{m-1}$.  Since $\{1,\dots,r\}\subseteq S$, $\sum_{m\in S}x^{m-1}=\frac{1-x^r}{1-x}+x^rC$ with $C=\sum_ic_ix^i$.  Hence $(1-x)L'=x^r\bigl(1-(1-x)C\bigr)=x^rD$ with $D=\sum d_ix^i\in\Z[[x]]$, and
\[
(1-x)A'=x^rDA.
\]
The coefficient of $x^n/n!$ in $A'$ is $a_{n+1}$, in $xA'$ it is $n\,a_n$, and in $x^jA$ it is $\fall{n}{j}a_{n-j}$.  Comparing coefficients gives (\ref{eq:rec}).  For $S=\{1,k\}$ and $r=1$ the only nonzero $c_i$ is $c_{k-2}=1$.  So for $k\ge3$ we get $d_0=1$, $d_{k-2}=-1$, $d_{k-1}=1$, and (\ref{eq:rec}) is (\ref{eq:Rk}).  For $k=2$, $d_0=0$ and $d_1=1$, which is (\ref{eq:R2}).  For $S=\{1,2,3\}$ and $r=3$ all $c_i$ vanish, so $d_0=1$ and the other $d_i$ vanish, which is (\ref{eq:R3}).
\end{proof}

\begin{proof}[Proof of Lemma~\ref{lem:A}]
Let $X$ be the set of permutations of $\{1,\dots,n+q\}$ with no cycle length in $S$.  Conjugation by $\chi$ maps $X$ to itself, and the orbits have size $1$ or $q$, so $|X|\equiv|X^\chi|\pmod q$, where $X^\chi$ is the set of $\sigma\in X$ with $\sigma \chi=\chi\sigma$.  Such a $\sigma$ maps the support of $\chi$ to itself, so $\sigma=s\,\tau$ with $s$ a permutation of the first $n$ points and $\tau$ a permutation of the last $q$ points commuting with $\chi$.  The centralizer of a $q$-cycle in $S_q$ is $\langle \chi\rangle$, so $\sigma=s\,\chi^j$ with $0\le j<q$.  Now $\sigma\in X$ if and only if $s$ has no cycle length in $S$ and $\chi^j$ contributes none, where $\chi^0$ contributes $q$ fixed points, allowed if and only if $1\notin S$, and $\chi^j$ with $j\ne0$ is a $q$-cycle, allowed if and only if $q\notin S$.  Hence $|X^\chi|=\bigl(w_1+(q-1)w_q\bigr)a_S(n)$.  The last sentence follows by applying the congruence $t$ times.
\end{proof}

\begin{proof}[Proof of Lemma~\ref{lem:quot}]
Look at the cycle containing the point $n$.  If its length is $m\notin S$, there are $\fall{n-1}{m-1}$ ways to choose and order its other points, and the remaining $n-m$ points carry a permutation counted by $a(n-m)$.  Every allowed $m$ exceeds $r$ and $\fall{n-1}{m-1}=\fall{n-1}{r}\fall{n-r-1}{m-r-1}$, which gives $a(n)=\fall{n-1}{r}Q_S(n)$.

Now let $q\notin S$ be a prime and let $n>r$.  The terms of $Q_S(n+q)$ with $m\le n$ are $\fall{n+q-r-1}{m-r-1}a(n+q-m)$.  The falling factorial is a polynomial in $n$ with integer coefficients, so it is congruent modulo $q$ to $\fall{n-r-1}{m-r-1}$, and $a(n+q-m)\equiv-a(n-m)$ by Lemma~\ref{lem:A} since $1\in S$.  So these terms add up to a quantity congruent to $-Q_S(n)$.  The remaining terms have $n<m\le n+q$.  The falling factorial $\fall{n+q-r-1}{m-r-1}$ is the product of the integers from $n+q-m+1$ up to $n+q-r-1$.  Here $n+q-m+1\le q$ because $m\ge n+1$, and $q\le n+q-r-1$ because $n\ge r+1$.  So the product contains the factor $q$ and the remaining terms vanish modulo $q$.
\end{proof}

\subsection{Proofs for Section~\ref{sec:general}}

\begin{proof}[Proof of Theorem~\ref{thm:general}(ii)]
Let $a_n=a_S(n)$.  Apply (\ref{eq:rec}) with $r=1$ and divide by $n\ge1$, using $\fall{n}{1+i}=n\fall{n-1}{i}$,
\begin{equation}\label{eq:Xk}
a_{n+1}/n=a_n+\sum_{i\ge0}d_i\,\fall{n-1}{i}\,a_{n-1-i}.
\end{equation}
With $r=1$ we have $c_i=1$ exactly when $i+2\in S$, and the hypothesis says $c_i=0$ for $0\le i\le t-1$.  So $d_i=[i=0]-c_i+c_{i-1}$ vanishes for $1\le i\le t-1$ and $d_0=1$, $d_t=c_{t-1}-c_t=-c_t$.  For $i\ge t$ the factor $\fall{n-1}{i}$ in \eqref{eq:Xk} is divisible by $i!$, hence by $t!$ (a product of $i$ consecutive integers is $i!$ times a binomial coefficient, even if some factor is $0$).  Therefore
\begin{equation}\label{eq:modt}
a_{n+1}/n\equiv a_n+a_{n-1}\pmod{t!}\qquad(n\ge1).
\end{equation}
By Theorem~\ref{thm:general}(i), $a_n=(n-1)Q_n$ for integers $Q_n$ ($n\ge2$).  Then \eqref{eq:modt} reads $Q_{n+1}\equiv(n-1)Q_n+(n-2)Q_{n-1}\pmod{t!}$ for $n\ge3$.  For derangements the same recursion holds with equality, $D(n+1)=n\bigl(D(n)+D(n-1)\bigr)$, so $Q^D_n\coloneqq D(n)/(n-1)$ satisfies it for $n\ge3$.  Also $Q_2=a_2=1=D(2)$ (the single transposition, as $2\notin S$ and $1\in S$), and $Q_3=a_3/2=1=Q^D_3$ (the two $3$-cycles, as $3\notin S$).  By induction $Q_n\equiv Q^D_n\pmod{t!}$ for all $n\ge2$, that is $a_n\equiv D(n)\pmod{t!(n-1)}$.

For the last assertion take $t=2$, which gives $a_n/(n-1)\equiv Q^D_n\pmod2$.  From $D(n)=nD(n-1)+(-1)^n$ one gets $D(n)$ odd for even $n$ and even for odd $n$ (induction), so $Q^D_n=D(n-1)+D(n-2)$ is odd for $n\ge2$.  This completes the proof of Theorem~\ref{thm:general}(ii).
\end{proof}

\begin{proof}[Proof of Theorem~\ref{thm:general}(iii)]
By Lemma~\ref{lem:quot} with $r=p$ we have $a_S(n)=\fall{n-1}{p}R_n$ for $n\ge p+1$, with
\[
R_n=\sum_{\substack{p<m\le n\\ m\notin S}}\fall{n-p-1}{m-p-1}\,a_S(n-m),
\]
and it suffices to show $v_p(R_n)\ge\lfloor(n-1)/(2p)\rfloor$.  We use induction on $n$.  For $n=p+1$ the bound is $0$.  Let $n\ge p+2$ and assume the theorem for all $N$ with $p+1\le N<n$.  Consider a term of $R_n$ with index $m\notin S$.

If $n-p\le m\le n-1$ then $1\le n-m\le p$, and every cycle length at most $p$ lies in $S$, so $a_S(n-m)=0$.

If $m=n$, the term is $(n-p-1)!$, whose valuation is at least $\lfloor(n-p-1)/p\rfloor$.  Put $x=(n-1)/(2p)$, so that $(n-p-1)/p=2x-1$.  If $x<1$ then $\lfloor x\rfloor=0$ and there is nothing to prove.  If $x\ge1$ and $f=\lfloor x\rfloor$, then $\lfloor2x-1\rfloor\ge2f-1\ge f$.

If $m\le n-p-1$ write $m=p+1+j$ with $j\ge0$ and $N=n-m\ge p+1$.  The factor $\fall{n-p-1}{m-p-1}$ is a product of $j$ consecutive positive integers, so it is divisible by $j!$ and has valuation at least $\lfloor j/p\rfloor$.  The product $\fall{N-1}{p}$ of $p$ consecutive positive integers is divisible by $p!$, so by the induction hypothesis $v_p(a_S(N))\ge1+\lfloor(N-1)/(2p)\rfloor$.  Hence the term has valuation at least
\[
\Bigl\lfloor\frac jp\Bigr\rfloor+1+\Bigl\lfloor\frac{n-1}{2p}-\frac{p+1+j}{2p}\Bigr\rfloor
\ \ge\ \Bigl\lfloor\frac jp\Bigr\rfloor+1-\Bigl\lceil\frac{p+1+j}{2p}\Bigr\rceil+\Bigl\lfloor\frac{n-1}{2p}\Bigr\rfloor ,
\]
using $\lfloor y-z\rfloor\ge\lfloor y\rfloor-\lceil z\rceil$.  Let $u=\lfloor j/p\rfloor$.  Since $j\le up+p-1$ we have $(p+1+j)/(2p)\le u/2+1$, so $\lceil(p+1+j)/(2p)\rceil\le\lceil u/2\rceil+1\le u+1$.  Thus every term has valuation at least $\lfloor(n-1)/(2p)\rfloor$.

For the consequences, $v_p(\fall{n-1}{p})\ge1$ for $n\ge p+1$ because $\fall{n-1}{p}$ is a product of $p$ consecutive positive integers, so $v_p(a_S(n))\ge1+\lfloor(n-1)/(2p)\rfloor=\lceil n/(2p)\rceil$, and for $n\le p$ we have $a_S(n)=0$.  If $\{1,2\}\subseteq S$ and $p=2$ we get $v_2(a_S(n))\ge v_2((n-1)(n-2))+\lfloor(n-1)/4\rfloor$ for $n\ge3$, and every odd prime $q$ has $v_q((n-1)(n-2))\le v_q(a_S(n))$ by (i).  This completes the proof of Theorem~\ref{thm:general}(iii).
\end{proof}

\subsection{Proofs for Section~\ref{sec:A378}}

\begin{proof}[Proof of Theorem~\ref{thm:A378}(c),(d)]
For $n\ge2$, a derangement with an $n$-cycle is an $n$-cycle, and there are $(n-1)!$ of them.  For $n\ge3$, a derangement with an $(n-1)$-cycle would leave a point fixed, so none exist.  Hence
\begin{equation}\label{eq:Tnn}
T(n,n)=D(n)-(n-1)!,\qquad T(n,n-1)=D(n).
\end{equation}
Now let $n\ge5$.  A derangement with an $(n-2)$-cycle has its other two points in a $2$-cycle.  Since $n-2>n/2$ there is only one such cycle, and there are $\binom n2$ choices for the two points and $(n-3)!$ cyclic orders for the others.  Hence
\begin{equation}\label{eq:Tn2}
T(n,n-2)=D(n)-\binom n2(n-3)!.
\end{equation}
Subtracting \eqref{eq:Tn2} from the second formula in \eqref{eq:Tnn},
\begin{equation}\label{eq:d}
T(n,n-1)-T(n,n-2)=\binom n2(n-3)!=\frac{n(n-1)(n-3)!}{2}=\frac{n!}{2(n-2)},
\end{equation}
which is (d).  Subtracting the first formula in \eqref{eq:Tnn} from \eqref{eq:Tn2},
\begin{equation}\label{eq:c}
\begin{aligned}
T(n,n-2)-T(n,n)&=(n-1)!-\binom n2(n-3)!\\&=(n-3)!\Bigl[(n-1)(n-2)-\tfrac{n(n-1)}2\Bigr]=\tfrac12(n-3)!\,(n-1)(n-4),
\end{aligned}
\end{equation}
which is (c).
\end{proof}

\subsection{Proofs for Section~\ref{sec:gcd}}

\begin{proof}[Proof of Proposition~\ref{prop:gcdq}]
Since $q>r$ we have $q\notin S_\infty$.  Write $n=tq+\rho$ with $0\le\rho<q$.

\emph{Case $1\le\rho\le r$.}  Since $n>r$ we have $t\ge1$.  Put $n'=q+\rho$, so that $n'>r$ and $n=n'+(t-1)q$.  Let $k\in\{\infty,n'\}$.  Then $q\notin S_k$ (as $n'>q$), so Lemma~\ref{lem:quot}, applied $t-1$ times, gives $Q_k(n)\equiv(-1)^{t-1}Q_k(n')\pmod q$.  By (\ref{eq:Qdiff}), $Q_\infty(n')-Q_{n'}(n')=(n'-r-1)!$, and $0\le n'-r-1=q+\rho-r-1\le q-1$, so this difference is not divisible by $q$.  Hence $q$ does not divide both $Q_\infty(n')$ and $Q_{n'}(n')$.

\emph{Case $\rho=0$.}  Take $k=\infty$.  Since $1\in S_\infty$ and $q\notin S_\infty$, Lemma~\ref{lem:A} gives $a_\infty(n)\equiv(-1)^ta_\infty(0)=\pm1\pmod q$.  So $q\nmid a_\infty(n)$ and $q\nmid Q_\infty(n)$.

\emph{Case $r+1\le\rho\le q-1$.}  If $t=0$ then $n=\rho<q$ and $Q_n(n)=Q_\infty(n)-(n-r-1)!$ with $0\le n-r-1<q$, so one of the columns $\infty$ and $n$ has $q\nmid Q_k(n)$.  If $t\ge1$, the columns $k=\infty$ and $k=\rho$ satisfy $q\notin S_k$, so Lemma~\ref{lem:A} gives $a_k(n)\equiv(-1)^ta_k(\rho)\pmod q$.  Now $a_\rho(\rho)=a_\infty(\rho)-(\rho-1)!$ by (\ref{eq:Qdiff}), so the two values $a_\infty(\rho)$ and $a_\rho(\rho)$ differ by $(\rho-1)!$, which is not divisible by $q$ as $\rho-1<q$.  So $q$ does not divide both $a_\infty(n)$ and $a_\rho(n)$, and because $\fall{n-1}{r}\equiv\fall{\rho-1}{r}\not\equiv0\pmod q$ (the factors $\rho-1,\dots,\rho-r$ are positive and below $q$), the same holds for $Q_\infty(n)$ and $Q_\rho(n)$.

In all cases the column found lies in $\{\infty\}\cup(r,n]$, since $n'\le n$ and $\rho\le n$.  The last assertion follows from $a_k(n)=\fall{n-1}{r}Q_k(n)$.
\end{proof}

\begin{proof}[Proof of Theorem~\ref{thm:gcdfamily}]
Let $n\ge r+2$ and let $g'=G_r(n)$.  By (\ref{eq:sameg}), $g'$ is also the gcd of $a_{S_k}(n)$ over $k\in\{\infty\}\cup(r,n]$.  By Theorem~\ref{thm:general}(i), $\fall{n-1}{r}$ divides every $a_{S_k}(n)$, hence $g'$.  By Proposition~\ref{prop:gcdq}, for every prime $q>r$ some $a_{S_k}(n)$ with $k\in\{\infty\}\cup(r,n]$ has $v_q=v_q(\fall{n-1}{r})$, so $v_q(g')=v_q(\fall{n-1}{r})$.  Hence $g'=\fall{n-1}{r}m_r(n)$ with $m_r(n)$ having only prime factors at most $r$.  The divisibilities $\fall{n-1}{r}\mid g_r(n)\mid G_r(n)$ were noted before the theorem.

For $r=1$ there is no prime at most $r$, so $m_1(n)=1$ and $g_1(n)=G_1(n)=n-1$.  The gcd of row $n$ of \OEIS{A378495} is over $2\le k\le n$, which is the range $r<k\le n$, so this is Theorem~\ref{thm:A378}(b).

For $r=2$ and $n\ge4$ put $G=2^{\lfloor(n-1)/4\rfloor}(n-1)(n-2)$.  By Theorem~\ref{thm:general}(i) and (iii), $G$ divides $a_S(n)$ for every $S\supseteq\{1,2\}$, so $G\mid g_2(n)\mid G_2(n)$.  By the first part, for every odd prime $q$ some column has $v_q=v_q((n-1)(n-2))=v_q(G)$.  For the prime $2$ we need a column $k\in[3,n]$ with $v_2(a_k(n))=\beta(n)=v_2(G)$.  If $n\not\equiv3\pmod4$ take $k=3$ and use Proposition~\ref{prop:123}.  If $n=4m+3$ with $n\ge7$ take $k=n$.  By Theorem~\ref{thm:A378}(f), $v_2(a_\infty(4m+3))=m+1=\beta(n)$, and $v_2((n-1)!)\ge\lfloor(n-1)/2\rfloor=2m+1>m+1$ for $m\ge1$, so by (\ref{eq:Qdiff}) $v_2(a_n(n))=v_2(a_\infty(n)-(n-1)!)=\beta(n)$.  Hence $v_p(G_2(n))=v_p(G)$ for every prime $p$, and $G_2(n)=G=g_2(n)$.
\end{proof}

\begin{proof}[Proof of Theorem~\ref{thm:A333726}]
Let $n\ge3$ and $g=\gcd_kU(n,k)$.  For $k\ge2$, $1\in S=\{1,\dots,k-1\}$, so $(n-1)\mid U(n,k)$ by Theorem~\ref{thm:general}(i), and $(n-1)\mid U(n,1)=n!$.  So $(n-1)\mid g$.  Also $U(n,n)=(n-1)!$, so $g\mid(n-1)!$.  Let $q\le n-1$ be a prime, for which we show that the stated column $k$ satisfies $v_q(U(n,k))=v_q(n-1)$, which gives $v_q(g)=v_q(n-1)$ and hence $g=n-1$.

\emph{Case $q\mid n-1$.}  Take $k=q$ and $S=\{1,\dots,q-1\}$, so $1\in S$ and $q\notin S$.  Write $n-1=tq$.  By Lemma~\ref{lem:quot} with $r=1$, $U(n,q)/(n-1)=Q_S(n)\equiv(-1)^{t-1}Q_S(q+1)\pmod q$ with $Q_S(q+1)=U(q+1,q)/q$.  A permutation of $q+1$ points with all cycles of length at least $q$ is a $(q+1)$-cycle (a $q$-cycle would leave a fixed point), so $U(q+1,q)=q!$ and $Q_S(q+1)=(q-1)!\equiv-1\pmod q$ by Wilson's theorem.  Hence $q\nmid U(n,q)/(n-1)$.

\emph{Case $q\nmid n-1$.}  Write $n=tq+\rho$ with $t\ge1$, $0\le\rho<q$, $\rho\ne1$.  Take $k=2$ if $\rho=0$ and $k=\rho$ if $\rho\ge2$, which gives $2\le k\le q$, so $1\in S=\{1,\dots,k-1\}$ and $q\notin S$, and Lemma~\ref{lem:A} gives $U(n,k)\equiv(-1)^tU(\rho,k)\pmod q$.  If $\rho=0$, $U(0,2)=1$.  If $\rho\ge2$, $U(\rho,\rho)=(\rho-1)!$ (a $\rho$-cycle), which is not divisible by $q$ as $\rho-1<q$.  Hence $q\nmid U(n,k)$ and $v_q(U(n,k))=0=v_q(n-1)$.

This completes the proof of Theorem~\ref{thm:A333726}.
\end{proof}

\subsection{Proofs for Section~\ref{sec:coprime}}

\begin{proof}[Proof of Lemma~\ref{lem:bd}]
The exponential formula gives $\sum_nb_d(n)x^n/n!=\exp\bigl(\sum_{j\ge1}x^{dj}/(dj)\bigr)=(1-x^d)^{-1/d}$. By Lemma~\ref{lem:binom} the coefficient of $x^{dk}$ is $\binom{1/d+k-1}{k}=\prod_{j=0}^{k-1}(dj+1)/(d^kk!)$, which gives the formula.  If $p\mid d$ then every factor $dj+1$ is prime to $p$, which gives the valuation.  The number $n!/(d^kk!)$ equals $\bigl((dk)!/((d!)^kk!)\bigr)((d-1)!)^k$, where the first factor counts the partitions of $\{1,\dots,dk\}$ into $k$ blocks of size $d$. For the identities, $\varphi(g)=\sum_{d\mid g}\psi(d)$ for $g\mid n$ by M\"obius inversion on the divisors of $n$, so
\[
C_\varphi(n)=\sum_\sigma\sum_{d\mid\gcd(\sigma)}\psi(d)=\sum_{d\mid n}\psi(d)\,\#\{\sigma:d\mid\gcd(\sigma)\}=\sum_{d\mid n}\psi(d)b_d(n).
\]
The choice $\varphi(h)=[h=g]$ gives $\psi(d)=\mu(d/g)[g\mid d]$.
\end{proof}

\begin{proof}[Proof of Theorem~\ref{thm:cyclegcd}]
(a) Let $q$ be a prime dividing $n-1$ and let $d\mid n$. Then $q\nmid d$, so $1/d\in\Z_{(q)}$, and $\binom{1/d+k-1}{k}\in\Z_{(q)}$ by Lemma~\ref{lem:binom}. Hence $v_q(b_d(n))\ge v_q(n!)\ge v_q(n-1)$, because $n-1$ divides $n!$. So $(n-1)\mid b_d(n)$ for every $d\mid n$, and $(n-1)\mid C_\varphi(n)$ by Lemma~\ref{lem:bd}.

(b) Put $N=n+1$ and $P(X)=\prod_{i=1}^{k}(X-di)\in\Z[X]$. Since $dk=n$, the substitution $j=k-i$ gives $P(N)=\prod_{j=0}^{k-1}(dj+1)$, and $P(0)=(-d)^kk!$. The polynomial $P(X)-P(0)$ has integer coefficients and no constant term, so $P(N)-P(0)=N\,R$ with $R\in\Z$. By Lemma~\ref{lem:bd},
\[
b_d(n)=\frac{n!}{d^kk!}P(N)=(-1)^kn!+(n+1)\frac{n!}{d^kk!}R,
\]
so $b_d(n)\equiv(-1)^{n/d}n!\pmod{n+1}$ and $C_\varphi(n)\equiv n!\sum_{d\mid n}\psi(d)(-1)^{n/d}$. For $C(n,g)$ this sum is $\sum_{e\mid k}\mu(e)(-1)^{k/e}$. For $k=1$ it is $-1$, and for $k=2$ it is $\mu(1)(-1)^2+\mu(2)(-1)^1=2$. For odd $k\ge3$ every $k/e$ is odd and the sum is $-\sum_{e\mid k}\mu(e)=0$. For $k=2^sk'$ with $s\ge1$ and $k'$ odd, write $(-1)^{k/e}=1-2[k/e\text{ odd}]$, so the sum is $\sum_{e\mid k}\mu(e)-2\sum_{k/e\text{ odd}}\mu(e)$. The first sum is $0$. The condition that $k/e$ is odd means $2^s\mid e$, which is impossible with $\mu(e)\ne0$ when $s\ge2$, and for $s=1$ forces $e=2e'$ with $e'\mid k'$, so that $\sum_{k/e\text{ odd}}\mu(e)=-\sum_{e'\mid k'}\mu(e')=-[k'=1]$ and the whole sum equals $2[k'=1]$. So the sum equals $2[k=2]-[k=1]$ in all cases, and $C(n,g)\equiv n!\,(2[k=2]-[k=1])\pmod{n+1}$.

(c) For even $n$ the integers $n\pm1$ are coprime, and (a), (b) give the claim. For odd $n\ge5$ an odd prime $q$ divides at most one of $n\pm1$, and (a), (b) give $v_q(n^2-1)\le v_q(C(n,g))$. For $q=2$ every $d\mid n$ is odd, so as in (a) $v_2(b_d(n))\ge v_2(n!)$, and it suffices to show $v_2(n!)\ge v_2(n^2-1)$. For $n=5,7$ one checks $v_2(n!)=3,4$ and $v_2(n^2-1)=3,4$. For $n\ge9$ one of $n\pm1$ has $2$-adic valuation exactly $1$ and the other has valuation $t$ with $2^t\le n+1$, so $v_2(n^2-1)\le1+\log_2(n+1)$, while $v_2(n!)=n-s_2(n)\ge n-\log_2n-1\ge1+\log_2(n+1)$.

(d) By Lemma~\ref{lem:bd} and the formula for $b_d(n)$, for $p\mid d$ the product $\prod_{j<k_d}(dj+1)$ is a unit and $v_p(b_d(n))=v_p(n!)-c(d)$ with $k_d=n/d$ and $c(d)=k_dv_p(d)+v_p(k_d!)$, while for $p\nmid d$ one has $v_p(b_d(n))\ge v_p(n!)$. Recall $C(n,g)=\sum_{e\mid k}\mu(e)b_{ge}(n)$ with $e$ squarefree.

\emph{Case $p\mid g$.} Every $d=ge$ is divisible by $p$. For $e>1$,
\[
c(g)-c(ge)=\frac ke\bigl((e-1)v_p(g)-v_p(e)\bigr)+\bigl(v_p(k!)-v_p((k/e)!)\bigr).
\]
The first bracket is positive unless $e=2$, $p=2$ and $v_2(g)=1$, where it vanishes and the second bracket equals at least $k/2\ge1$ (as $v_2(k!)-v_2((k/2)!)\ge k/2$). The second bracket is non-negative in all cases. So $c(g)>c(ge)$ for every $e>1$, the term $e=1$ has the smallest valuation, and $v_p(C(n,g))=v_p(n!)-c(g)=v_p(n!)-kv_p(g)-v_p(k!)$.

\emph{Case $p\nmid g$.} Since $p\mid n=gk$ one has $p\mid k$. The terms with $p\nmid e$ have $p\nmid ge$ and valuation at least $v_p(n!)$. The terms with $e=pe'$, $p\nmid e'$, have $v_p(ge)=1$ and valuation $v_p(n!)-t-v_p(t!)$ with $t=k/(pe')\ge1$. Since $t\mapsto t+v_p(t!)$ is strictly increasing, the unique smallest valuation is that of $e'=1$, namely $v_p(n!)-k/p-v_p((k/p)!)<v_p(n!)$.
\end{proof}

\begin{proof}[Proof of Theorem~\ref{thm:A335088}]
(a) If $q\mid n-1$ is prime, then $q\nmid d$ for every $d\mid n$, so $[x^n]F_{A_n}\in\Z_{(q)}$ by Lemma~\ref{lem:binom}, and $v_q(\operatorname{cop}(n))\ge v_q(n!)\ge v_q(n-1)$.

(b) Let $p\mid n$ and let $X_m$ be the set of permutations of $m$ points all of whose cycle lengths are prime to $n$. For $m\ge p$ let $\gamma$ be a $p$-cycle on $p$ of the points. Conjugation by $\gamma$ permutes $X_m$ in orbits of size $1$ or $p$, so $|X_m|\equiv|X_m^\gamma|\pmod p$. An element commuting with $\gamma$ has the form $s\gamma^i$ with $s$ a permutation of the other $m-p$ points and $0\le i<p$. Since $p\mid n$, a $p$-cycle is forbidden, so $i=0$, and fixed points are allowed. Hence $|X_m|\equiv|X_{m-p}|$, and $\operatorname{cop}(n)=|X_n|\equiv|X_0|=1\pmod p$.

(c) Suppose first that $n+1$ is composite, let $q$ be a prime dividing $n+1$ and $e=v_q(n+1)$. Then $q\nmid n$, so as in (a) $v_q(\operatorname{cop}(n))\ge v_q(n!)$, and it suffices to show $v_q(n!)\ge e$. If $n+1=q^em$ with $m\ge2$ then $n\ge2q^e-1$ and $v_q(n!)\ge\lfloor n/q\rfloor\ge2q^{e-1}-1\ge e$. If $n+1=q^e$ with $e\ge2$ then $v_q(n!)=\frac{q^e-1}{q-1}-e\ge e$ except for $q^e=4$, where $\operatorname{cop}(3)=4$ is divisible by $4$ directly.  The case of $n+1$ prime is Proposition~\ref{prop:cop708} below.
\end{proof}

\begin{proof}[Proof of Proposition~\ref{prop:cop}]
Let $d$ be a squarefree divisor of $n$.  Then $N\nmid d$, and $d\cdot(n/d)=n\equiv-1\pmod N$, so $1/d\equiv-n/d$ in $\Z_{(N)}/N\Z_{(N)}$.  The coefficient of $y^\lambda$ in $(1-y)^{-c}$ is $c(c+1)\cdots(c+\lambda-1)/\lambda!$, a polynomial in $c$ whose denominator is a unit at $N$ when $\lambda<N$.  Hence if $c,c'\in\Z_{(N)}$ and $c\equiv c'\pmod N$, then the coefficients of $y^\lambda$ in $(1-y)^{-c}$ and $(1-y)^{-c'}$ agree modulo $N$ for all $\lambda<N$.  Apply this with $c=\mu(d)/d$ and $c'=-\mu(d)n/d$ to each factor $(1-x^d)^{-\mu(d)/d}$ of $F_{A_n}$, noting that only degrees at most $n<N$ matter.  This gives $F_{A_n}\equiv\prod_{d\mid n}(1-x^d)^{\mu(d)n/d}\pmod{(N,x^{n+1})}$.  Wilson's theorem gives $n!=(N-1)!\equiv-1\pmod N$, and $\operatorname{cop}(n)=n!\,\coeff{n}{F_{A_n}}$, which proves the congruence.

For the recurrence, $F_{A_n}=\exp\bigl(\sum_{\gcd(m,n)=1}x^m/m\bigr)$ satisfies $xF_{A_n}'=F_{A_n}\sum_{\gcd(m,n)=1}x^m$, so its coefficients $f_k$ satisfy $kf_k=\sum_{m\le k,\ \gcd(m,n)=1}f_{k-m}$.  Since $\mathbf 1[\gcd(m,n)=1]=\sum_{d\mid\gcd(m,n)}\mu(d)$, this is the stated recurrence.  The inner sums $T_d(k)=\sum_{j\ge1}f_{k-dj}$ satisfy $T_d(k)=f_{k-d}+T_d(k-d)$, which gives the operation count.  All divisions are by integers at most $n<N$, and $\operatorname{cop}(n)=n!f_n\equiv-f_n\pmod N$.
\end{proof}

\begin{proof}[Proof of Proposition~\ref{prop:cop708}]
All statements are finite exact computations.  The number $\operatorname{cop}(708)$ has $1711$ decimal digits.  We computed it in exact integer arithmetic from the recurrence $a(m)=\sum_{\gcd(l,708)=1}\fall{m-1}{l-1}\,a(m-l)$, which is (\ref{eq:last0}) for the set of allowed lengths prime to $708$, and found $\operatorname{cop}(708)\equiv0\pmod{709}$.  The recurrence of Proposition~\ref{prop:cop} gives the same residue, as does the cycle recurrence carried out modulo $709$.  For $n=92202$ the recurrence of Proposition~\ref{prop:cop} and the cycle recurrence modulo $92203$ both give residue $0$.  The recurrence of Proposition~\ref{prop:cop} agrees with the exact value of $\operatorname{cop}(n)\bmod(n+1)$ for every prime $n+1\le301$.  Running it for every $n\le3\cdot10^6$ with $n+1$ prime, about $2.17\cdot10^5$ values, gives a non-zero residue except at $n=708$ and $n=92202$.  The search was repeated with the prefix-sum recurrence of Proposition~\ref{prop:copprefix}, a different algorithm, and it gives the same two zeros.  The recurrence of the proposition agrees with the cycle recurrence modulo $n+1$ for all primes $n+1\le20000$.
\end{proof}

\subsection{Proofs for Section~\ref{sec:dichotomy}}

\begin{proof}[Proof of Theorem~\ref{thm:dichotomy}]
Let $p$ be odd.  If $w_1=w_p$ then $\gamma=pw_1$, and $\gamma^{p^{e-1}}$ is divisible by $p^{p^{e-1}}$ with $p^{e-1}\ge e$, so $A(p^e)\equiv0\pmod{p^e}$ and the shift congruence with $k=p^e$ gives $A(n+p^e)\equiv0$ for all $n\ge0$, which is (A).  If $(w_1,w_p)=(1,0)$ then $\gamma=1$ and $A(p^e)\equiv1$.  If $(w_1,w_p)=(0,1)$ then $\gamma=p-1$ and $(p-1)^{p^{e-1}}\equiv-1\pmod{p^e}$, because $x\equiv y\pmod{p^i}$ implies $x^p\equiv y^p\pmod{p^{i+1}}$ and $1-p\equiv1\pmod p$.  The shift congruence with $k=p^e$ gives $A(n+p^e)\equiv\pm A(n)$ in the two cases of (B).

Let $p=2$.  If $w_1=w_2$ then $\alpha$ is even, and $r=\max(e,4)$ gives $A(2^r)\equiv\alpha^{2^{r-1}}\equiv0\pmod{2^r}$ since $2^{r-1}\ge r$, so the shift congruence gives $A(n)\equiv0\pmod{2^r}$ for $n\ge2^r$.  If $w_1\ne w_2$ then $\alpha$ is odd.  Now $A(2)\equiv1\pmod2$, $A(4)\equiv1+2\beta\equiv(-1)^\beta\pmod4$, $A(8)\equiv1\pmod8$ because $\alpha^4\equiv1\pmod{16}$ and $\beta(\beta+1)$ is even, and $A(2^r)\equiv1\pmod{2^r}$ for $r\ge4$ by the same lifting as above.  The shift congruence with $k=2,4,8,2^r$ gives the stated congruences, with $\beta\equiv w_2+w_4\pmod2$.

For the last assertion, in case (B) the sequences are periodic or antiperiodic modulo $p^e$ with $A(0)=1$, so $p\nmid A(n)$ for $n\equiv0$ modulo the period.
\end{proof}

\subsection{Proofs for Section~\ref{sec:euler}}

\begin{proof}[Proof of Lemma~\ref{lem:weights}]
(a) Put $S_n=\sum_{d\mid n,\ d\le r}\mu(n/d)=\sum_{e\mid n,\ e\ge n/r}\mu(e)$, so that $c_n=S_n/n$. If $2\le n\le r$ every divisor of $n$ is at most $r$ and $S_n=\sum_{e\mid n}\mu(e)=0$. If $n>r$ then $n>1$, the full sum $\sum_{e\mid n}\mu(e)$ vanishes, and $S_n=-\sum_{e\mid n,\ e<n/r}\mu(e)$. For the last claim let $p\mid n$. Every squarefree divisor of $n$ is $e'$ or $pe'$ with $e'\mid w$ and $\mu(pe')=-\mu(e')$, so the sum over $e<n/r$ is $\sum_{e'\mid w,\ e'<n/r}\mu(e')-\sum_{e'\mid w,\ e'<n/(pr)}\mu(e')$, which is the window sum.

(b) The product $F_{A_{\le r}}F_{A_{>r}}=\exp(\sum_mx^m/m)=(1-x)^{-1}$, so by uniqueness $c_n(A_{\le r})+c_n(A_{>r})=[n=1]$.

(c) The logarithm of the displayed product has coefficient $\sum_{d\mid\gcd(m,M)}\mu(d)=[\gcd(m,M)=1]$ at $x^m/m$.
\end{proof}

\begin{proof}[Proof of Theorem~\ref{thm:saturation}]
(ii)$\Rightarrow$(iii). If there are no blocks then every $c_n$ lies in $\Z_{(p)}$, since $nc_n\in\Z$ settles $p\nmid n$, so $F_A$ is a product of series in $\Z_{(p)}[[x]]$ by Lemma~\ref{lem:binom}. (iii)$\Rightarrow$(ii). If $\mathcal B\ne\emptyset$ and $m_0=\min\mathcal B$, then $[x^{m_0}]F_A=[x^{m_0}]I_1+c_{m_0}\notin\Z_{(p)}$.

(i)$\Rightarrow$(ii) and (ii)$\Rightarrow$(i). Write $n=p^kw$ with $k\ge1$ and $p\nmid w$, put $e_i(w')=[p^iw'\in A]$ and $\delta_k=e_k-e_{k-1}$. Splitting $d=p^iw'$ in the definition of $nc_n$ gives
\[
nc_n=\sum_{w'\mid w}\mu(w/w')\sum_{i=0}^k\mu(p^{k-i})e_i(w')=\sum_{w'\mid w}\mu(w/w')\,\delta_k(w').
\]
For a saturated set all $\delta_k$ vanish, so $c_n=0$ whenever $p\mid n$, and there are no blocks. Conversely, if $A$ is not saturated, let $k\ge1$ be minimal with $\delta_k\not\equiv0$ and let $w$ be minimal under divisibility with $\delta_k(w)\ne0$. Then $nc_n=\delta_k(w)=\pm1$ for $n=p^kw$, so $v_p(c_n)=-k<0$ and $n$ is a block.

The last statement is Theorem~\ref{thm:calculus} with $\kappa=0$.
\end{proof}

\begin{proof}[Proof of Theorem~\ref{thm:dominant}]
The function $\Gamma$ is strictly increasing. Let $V$ have $b$ blocks $m_*$ and other blocks $(d_i,\lambda_i,j_i)$, $i=1,\dots,s$, meaning $\lambda_i$ blocks of degree $d_i>m_*$ and depth $j_i$, and let $\deg V\le k<m_*(q+1)$. Put $t=b+\sum_i\lambda_i$ and $Q'=b+\sum_i\lambda_i\lfloor d_i/m_*\rfloor$. Then $t\le Q'\le\lfloor\deg V/m_*\rfloor\le q$. By ($\star$), $\sum_ij_i\lambda_i\le j_*\sum_i\lambda_i\lfloor d_i/m_*\rfloor-\sum_i\lambda_i$, and $v_p(b!)+\sum_iv_p(\lambda_i!)\le v_p(t!)\le v_p(Q'!)$. Therefore, if $s\ge1$,
\[
\mathrm{cost}(V)\le j_*Q'+v_p(Q'!)-\sum_i\lambda_i=\Gamma(Q')-\sum_i\lambda_i\le\Gamma(q)-1.
\]
If $s=0$ then $\mathrm{cost}(V)=\Gamma(b)\le\Gamma(q)$ with equality only for $b=q$. So $\kappa(k)=\Gamma(q)$ and $\mathrm{Max}(k)=\{qe_{m_*}\}$ with leftover $\nu$, and $\Xi(k)=\iota_\nu\alpha_*^q\tilde\omega(q)^{-1}$ is non-zero if and only if $\iota_\nu\ne0$. Since $\nu<m_*=\min\mathcal B$, $[x^\nu]G=[x^\nu]I_W\in\Z_{(p)}$, and Theorem~\ref{thm:calculus} applies.
\end{proof}

\begin{proof}[Proof of Theorem~\ref{thm:invol}]
Take $A=\{1,2\}$ and $p=2$, so that $F_A=\exp(x+x^2/2)$ and $a_A(n)=I(n)$.  By Lemma~\ref{lem:eulerprod}, $c_n=\bigl(\mu(n)+[2\mid n]\,\mu(n/2)\bigr)/n$ for $n\ge2$, and we write $n=2^kw$ with $w$ odd.  For $k=0$ the exponent $c_n$ is $2$-integral.  For $k=1$, $\mu(2w)=-\mu(w)$ gives $c_n=0$.  For $k\ge3$ both M\"obius values vanish.  For $k=2$, $c_n=-\mu(w)/(4w)$.  So the blocks are the degrees $4w$ with $w$ odd and squarefree, all of depth $2$, with $c_4=-\frac14$, and every block other than $4$ has $w\ge3$.  Put $\Gamma(t)=2t+v_2(t!)$, so that $\Gamma(t+1)-\Gamma(t)=2+v_2(t+1)\ge2$.

Let $n=4q+\nu$ with $0\le\nu\le3$, and let $V$ have $b$ blocks of degree $4$ and $\lambda_i$ blocks of degree $4w_i$ with $w_i\ge3$ for $i=1,\dots,s$, and degree at most $n$.  Put $t=b+\sum_i\lambda_i$ and $Q'=b+\sum_iw_i\lambda_i$, so that $4Q'\le n$ and $Q'\le q$.  Since $w_i\ge3$ we have $t\le Q'-2\sum_i\lambda_i$.  By superadditivity of $v_2((\cdot)!)$, $\mathrm{cost}(V)=2t+v_2(b!)+\sum_iv_2(\lambda_i!)\le\Gamma(t)$.  If $s\ge1$ then $t\le q-2$ and $\mathrm{cost}(V)\le\Gamma(q-2)\le\Gamma(q)-4$.  If $s=0$ and $b<q$ then $\mathrm{cost}(V)=\Gamma(b)\le\Gamma(q-1)\le\Gamma(q)-2$.  So every vector other than $q$ blocks of degree $4$ has cost at most $\Gamma(q)-2$, and by Lemma~\ref{lem:blockexp} its term in $\coeff{n}{F_A}$ has valuation at least $-\Gamma(q)+2$.

The vector with $q$ blocks of degree $4$ has leftover $\nu<4=\min\mathcal B$, and its term is $\coeff{\nu}{F_A}\binom{c_4+q-1}{q}$ with $\coeff{\nu}{F_A}=I(\nu)/\nu!=1,1,1,\frac23$ for $\nu=0,\dots,3$ (Lemma~\ref{lem:essential}).  By Lemma~\ref{lem:binom} the binomial coefficient has valuation exactly $-\Gamma(q)$, so the term has valuation $-\Gamma(q)+\delta_\nu$ with $(\delta_0,\dots,\delta_3)=(0,0,0,1)$.  Since $\delta_\nu\le1<2$, $v_2(\coeff n{F_A})=-\Gamma(q)+\delta_\nu$.  Multiplying by $n!$ and using $v_2(n!)=n-s_2(n)=4q+\nu-s_2(q)-s_2(\nu)$ and $\Gamma(q)=3q-s_2(q)$,
\[
v_2(I(n))=q+\nu-s_2(\nu)+\delta_\nu=q+\epsilon_\nu,
\]
as $\nu-s_2(\nu)+\delta_\nu=0,0,1,2$ for $\nu=0,\dots,3$.
\end{proof}

\begin{proof}[Proof of Lemma~\ref{lem:belowpower}]
By Lemma~\ref{lem:blockfacts} blocks are multiples of $p$ larger than $r$. Since $p^k-p\le r<p^k$ there is no multiple of $p$ in $(r,p^k)$, so every block is at least $p^k$. For $n=p^k$ Lemma~\ref{lem:weights}(a) gives $nc_n=-\sum_{e\mid p^k,\ e<p^k/r}\mu(e)=-1$, because the divisor $e=p$ would need $r<p^{k-1}$, which is false since $r\ge p^k-p\ge p^{k-1}$ for $k\ge2$, and for $k=1$ the condition $p<p/r$ fails. So $p^k$ is a block of depth $k$.

Let $n\ne p^k$ be a block, $e=v_p(n)\ge1$ and $u=\lfloor n/p^k\rfloor\ge1$, so $j_n\le e$. If $e<k$ then $ku\ge k>e$. If $e=k$ then $n=p^kw$ with $w\ge2$ and $ku\ge2k>e$. If $e=k+t$ with $t\ge1$ then $u\ge p^t$ and $kp^t>k+t$ unless $(p,k,t)=(2,1,1)$, which is excluded. So $j_n<ku$. By Lemma~\ref{lem:blockfacts} the same holds for $A_{>r}$.
\end{proof}

\begin{proof}[Proof of Theorem~\ref{thm:harmonic}]
By Lemma~\ref{lem:belowpower} Theorem~\ref{thm:dominant} applies with $m_*=p^k$ to $G=F_{A_{\le r}}$, and to $G=F_{A_{>r}}/(1-x)$ in the setting of Theorem~\ref{thm:quotient}.

(a) Theorem~\ref{thm:dominant} gives $v_p(\coeff nF)\ge-\Gamma(q)=-kq-v_p(q!)$, with equality if and only if $p\nmid\coeff\nu F=A_r(\nu)/\nu!$, that is $v_p(A_r(\nu))=v_p(\nu!)$ (the number $\coeff\nu F$ lies in $\Z_{(p)}$ since $\nu<p^k=\min\mathcal B$).  Multiplying by $n!$ gives the bound and the equality criterion.  The two expressions in the statement agree because, by Legendre's formula and $\nu<p^k$, $v_p((p^kq+\nu)!)=q\frac{p^k-1}{p-1}+v_p(q!)+v_p(\nu!)$.  If $\nu\le r$ then $A_r(\nu)=\nu!$, and if $r=p^k-1$ then $\nu\le r$ always.

(b), (c) It remains to compute $[x^\nu]G$ for $\nu<p^k$.  The terms $1/m$ of $H_r(\nu)$ are $p$-integral because there is no multiple of $p$ in $(r,p^k)$.  For $\nu\le r$ one has $[x^\nu]F_{A_{\le r}}=1$ and $[x^\nu]F_{A_{>r}}=[\nu=0]$, so $[x^\nu]G$ is a unit in both cases and there is no exception.  Now let $r<\nu<p^k\le2r+2$, so $\nu\le2r+1$.  A permutation of $\nu$ points has at most one cycle longer than $r$.  The permutations with exactly one cycle of length $m>r$ number $\binom\nu m(m-1)!(\nu-m)!=\nu!/m$, since the other $\nu-m\le r$ points are arbitrary.  Hence $a_{A_{\le r}}(\nu)/\nu!=1-H_r(\nu)$.  Likewise a permutation of $i$ points with $r<i\le2r+1$ and all cycles longer than $r$ is an $i$-cycle, so $a_{A_{>r}}(i)/i!=1/i$, while it is $0$ for $1\le i\le r$.  Therefore $[x^\nu]\bigl(F_{A_{>r}}/(1-x)\bigr)=\sum_{i\le\nu}a_{A_{>r}}(i)/i!=1+H_r(\nu)$.  By Theorem~\ref{thm:dominant}, $p$ divides $1-H_r(\nu)$ in (b) and $1+H_r(\nu)$ in (c) exactly in the exceptional classes.  The valuation statement of (c) is the dominant block bound multiplied by $(n-r-1)!\,\fall{n-1}{r}/(n-r-1)!$ via Theorem~\ref{thm:quotient}.

(d) For $r=p^k-1$ no $\nu$ lies in $(r,p^k)$, and $\fall{n-1}{r}(n-r-1)!=(n-1)!$.
\end{proof}

\begin{proof}[Proof of Corollary~\ref{cor:plateau}]
Parts (a), (b), (d) are the case $r=p^k-1$ of Theorem~\ref{thm:harmonic}(a). For (a), $p=2$, $k=2$ and $v_2(A_3(n))=q+v_2(\nu!)=q+\lfloor\nu/2\rfloor$ with $n=4q+\nu$, and checking $\nu=0,1,2,3$ shows that this equals $1+\lfloor(n-2)/4\rfloor$ for $n\ge2$. For (b), the exponent $(2^k-1)-k$ is the coefficient in the theorem. For (d), $k=1$ and $r=p-1$ give $v_p(A_{p-1}(n))=v_p(n!)-q-v_p(q!)=0$ for $p$ odd, and for $p=2$ the number $A_1(n)=1$ is odd. For (c) take $p=3$, $k=2$, so that $6\le r\le8$ and $v_3(n!)-2q-v_3(q!)=2q+\lfloor\nu/3\rfloor=\lfloor n/3\rfloor-\lfloor n/9\rfloor$. The exceptional classes are given by the Harmonic Theorem in the example above, namely none for $r=7,8$ and $\nu=7$ for $r=6$, where the inequality is strict.
\end{proof}

\subsection{Proofs for Section~\ref{sec:p3}}

\begin{proof}[Proof of Lemma~\ref{lem:red3}]
Every block $d\notin\{6,9\}$ is a multiple of $3$ with $d\ge12$ by Lemma~\ref{lem:blocks3}, so $d=3u$ with $u\ge4$ and $d$ can be written $d=6a'+9b'$ with $a',b'\ge0$.  We choose such a representation, and when $d=9u'$ with $u'\ge2$ we take $a'=0$, $b'=u'$.  Let $X=\lambda_6+\sum a'_d\lambda_d$ and $Y=\lambda_9+\sum b'_d\lambda_d$ be the totals over all blocks, so that $6X+9Y=\deg V\le n$.  Since $v_3((x+y)!)\ge v_3(x!)+v_3(y!)$, the functions $\ell_6$ and $\ell_9$ are superadditive, so $\ell_6(X)+\ell_9(Y)\ge \ell_6(\lambda_6)+\ell_9(\lambda_9)+\sum_d\bigl(\ell_6(a'_d\lambda_d)+\ell_9(b'_d\lambda_d)\bigr)$.  For a block with $j=1$ and $d\ne6$, so $d=3u$ with $u\ge4$, we have $a'+2b'\ge2$ (if $b'=0$ then $a'=u/2\ge2$), hence $\ell_6(a'\lambda)+\ell_9(b'\lambda)\ge2\lambda\ge\lambda+v_3(\lambda!)+\frac{\lambda+1}2$ because $v_3(\lambda!)\le(\lambda-1)/2$.  For a block with $j=2$ and $d=9u'$, $u'\ge2$, we have $\ell_9(u'\lambda)\ge4\lambda\ge2\lambda+v_3(\lambda!)+\frac{3\lambda+1}2$.  In both cases the gain over $j\lambda+v_3(\lambda!)$ is at least $1$.  Hence $\mathrm{cost}(V)\le \ell_6(X)+\ell_9(Y)-1\le\kappa(n)-1$.
\end{proof}

\begin{proof}[Proof of Lemma~\ref{lem:max3}]
We have $2f(X,Y)=3X+5Y-s(X)-s(Y)$.  The condition $6X+9Y\le n$ is $L:=2X+3Y\le M:=\lfloor n/3\rfloor=3q+t$.

\emph{Fixed $L$.}  Since $2X\equiv L\pmod 3$, write $X=3X_1+a_0$ with $a_0\equiv2L\pmod3$ and $a_0\in\{0,1,2\}$, and $Y=b_0-2X_1$ with $b_0=(L-2a_0)/3$.  Then $3X+5Y=5L/3-X/3$ and $s(X)=s(X_1)+a_0$, so
\[
2f(X,Y)=\tfrac53L-\tfrac43a_0-\Delta(X_1),\qquad \Delta(X_1)=X_1+s(X_1)+s(b_0-2X_1).
\]
Subadditivity of the digit sum gives $s(b_0)\le s(b_0-2X_1)+s(2X_1)$, so $\Delta(X_1)-\Delta(0)\ge\psi(X_1)$ with $\psi(x)=x+s(x)-s(2x)$.  We have $\psi(0)=\psi(1)=0$, $\psi(2)=2$, and $\psi(x)\ge x-s(x)=2v_3(x!)\ge2$ for $x\ge3$.  So $\Delta(X_1)>\Delta(0)$ for $X_1\ge2$.  For $X_1=1$, $\Delta(1)-\Delta(0)=2+s(b_0-2)-s(b_0)\ge0$ with equality if and only if adding $2$ to $b_0-2$ produces no carry in base $3$, that is, if and only if $b_0\equiv2\pmod 3$ (and $b_0\ge2$).  Hence for each feasible $L$ the maximum $F(L)$ of $f$ is attained at $X_1=0$, and also at $X_1=1$ exactly when $b_0\equiv2\pmod 3$ and $b_0\ge2$.  Evaluating, for $K\ge0$ (and $K\ge1$ in the second line),
\[
\begin{aligned}F(3K)&=\ell_9(K)\ \text{at }(0,K),\qquad F(3K+1)=\ell_9(K-1)+2\ \text{at }(2,K-1),\\ F(3K+2)&=\ell_9(K)+1\ \text{at }(1,K),\end{aligned}
\]
with the additional maximisers $(3,K-2)$, $(5,K-3)$, $(4,K-2)$ when $b_0=K,K-1,K$ is $\equiv2\pmod 3$ respectively (and large enough).  The value $L=1$ is not feasible.

\emph{Varying $L$.}  Since $\ell_9(K)-\ell_9(K-1)=2+v_3(K)$, we get $F(3K+1)=\ell_9(K)-v_3(K)\le F(3K)$, with equality if and only if $3\nmid K$, and $F(3K+2)=\ell_9(K)+1>F(3K)$, and $F(3K+3)-F(3K+2)=1+v_3(K+1)>0$.  So the maximum of $F$ over $L\le M$ is $F(M)$ if $t=2$, it is $F(3q)=\ell_9(q)$ if $t\in\{0,1\}$, attained also at $L=3q+1$ when $t=1$ and $3\nmid q$.  For $L=3q+1$ the extra maximiser $(5,q-3)$ would need $q\equiv0\pmod3$, which is excluded.  Reading off the leftovers $\rho=n-3L$ for $X_1=0,1$ gives the list.
\end{proof}

\begin{proof}[Proof of Theorem~\ref{thm:p3}]
By Lemmas~\ref{lem:blockexp}, \ref{lem:blocks3} and~\ref{lem:red3}, $v_3(\coeff nF)\ge-\kappa(n)$, and the terms of $\coeff nF$ coming from vectors $V$ with other blocks have valuation at least $-\kappa(n)+1$.  The remaining terms come from pairs $(X,Y)$ with $6X+9Y\le n$.  The ones with $f(X,Y)<\kappa(n)$ also have valuation at least $-\kappa(n)+1$.  Hence
\[
v_3(A_r(n))\ge v_3(n!)-\kappa(n),
\]
and the terms with $f(X,Y)=\kappa(n)$, which are listed in Lemma~\ref{lem:max3}, determine whether equality holds.  By the same lemma, $\kappa(n)=2q+v_3(q!)+[\nu\ge6]$, and $v_3(n!)=4q+v_3(q!)+\lfloor\nu/3\rfloor$, so $v_3(n!)-\kappa(n)=2q+[\nu\ge3]=e(n)$.

The term for $(X,Y)$ is $\coeff{\rho}{I}\,\binom{c_6+X-1}{X}\binom{c_9+Y-1}Y$ with $\rho=n-6X-9Y\le5$.  Since $\rho<6=\min\mathcal B$, $\coeff\rho I=A_r(\rho)/\rho!$, which is $1,1,1,1,\frac34,\frac{11}{20}$ for $r=3$, $1,1,1,1,1,\frac45$ for $r=4$, and $1$ for $r=5$, for $\rho=0,\dots,5$.  Modulo $3$ these are $J_\rho=1,1,1,1,0,1$ ($r=3$), $1,1,1,1,1,2$ ($r=4$), and $1$ ($r=5$), where $0$ means that $\coeff\rho I$ is divisible by $3$.

Let $\omega(m)\in\{\pm1\}$ be the residue of $m!/3^{v_3(m!)}$ modulo $3$.  Since $3k+1\equiv1$ and $3k+2\equiv-1$ modulo $3$, we have $\omega(3k+1)=\omega(3k)$ and $\omega(3k+2)=-\omega(3k)$, and $\omega(0)=\omega(1)=1$, $\omega(2)=\omega(3)=\omega(4)=-1$.  By Lemma~\ref{lem:binom} with $c_6=-1/6$, the coefficient $\binom{c_6+X-1}X$ equals $3^{-\ell_6(X)}u_X$ with $u_X\equiv\prod(6i-1)/(2^X\omega(X))\equiv\omega(X)$ modulo $3$, because $\prod_{i<X}(6i-1)\equiv(-1)^X\equiv2^X$.  Likewise $\binom{c_9+Y-1}Y=3^{-\ell_9(Y)}u'_Y$ with $u'_Y\equiv(-1)^Y\omega(Y)$.  Put $\varsigma(X,Y)=\omega(X)(-1)^Y\omega(Y)$.  A direct computation with the rules for $\omega$ gives, for $q=3k+1$, $\varsigma(2,q-1)=\varsigma(0,q)$, and for $q=3k+2$, $\varsigma(3,q-2)=\varsigma(0,q)$, $\varsigma(2,q-1)=-\varsigma(0,q)$ and $\varsigma(4,q-2)=\varsigma(1,q)$.  (For instance for $q=3k+2$, $\varsigma(0,q)=(-1)^q\omega(3k+2)=-(-1)^q\omega(3k)$, $\varsigma(3,q-2)=\omega(3)(-1)^q\omega(3k)=-(-1)^q\omega(3k)$ and $\varsigma(2,q-1)=\omega(2)(-1)^{q-1}\omega(3k+1)=(-1)^q\omega(3k)$.)

The sum of the terms with $f(X,Y)=\kappa(n)$ is therefore $3^{-\kappa(n)}\tau$ times a quantity congruent modulo $3$ to the following, where $\tau$ is a unit.
\begin{itemize}[itemsep=0pt]
\item If $t=0$ ($\nu\le2$), the quantity is $J_\nu$, or $2J_\nu$ if $q\equiv2$.  Here $J_\nu=1$.
\item If $t=1$ ($\nu=3,4,5$), it is $J_\nu$ for $q\equiv0$, $J_\nu+J_{\nu-3}$ for $q\equiv1$, and $2J_\nu-J_{\nu-3}$ for $q\equiv2$.
\item If $t=2$ ($\nu\ge6$), it is $J_{\nu-6}$, or $2J_{\nu-6}$ if $q\equiv2$.  Here $J_{\nu-6}=1$.
\end{itemize}
The first and last cases never vanish.  In the middle case the values for $(\nu,q\bmod3)$ are as follows.  For $r=3$, $J=1,1,1,1,0,1$ gives $\nu=3$ values $1,2,1$, $\nu=4$ values $0,1,2$ (for $q\equiv0,1,2$), and $\nu=5$ values $1,2,1$.  For $r=4$, $J=1,1,1,1,1,2$ gives $\nu=3$ values $1,2,1$, $\nu=4$ values $1,2,1$, and $\nu=5$ values $2,0,0$.  For $r=5$ all values are $1,2,1$.  A value of $0$ occurs only for $r=3$ with $\nu=4$ and $q\equiv0$, which is $n=9q+4\equiv4\pmod{27}$, and for $r=4$ with $\nu=5$ and $q\equiv1,2$, which is $n\equiv14,23\pmod{27}$.  In every other case the leading terms add up to a unit times $3^{-\kappa(n)}$, so $v_3(\coeff nF)=-\kappa(n)$ and $v_3(A_r(n))=e(n)$.  In the exceptional cases all terms have valuation at least $-\kappa(n)+1$, so $v_3(A_r(n))>e(n)$.
\end{proof}

\subsection{Proofs for Section~\ref{sec:p2}}

\begin{proof}[Proof of Theorem~\ref{thm:p2r4}]
Let $r=4$, $p=2$ and $s=s_2$.  If $n\in\mathcal B$ then $e=v_2(n)\ge1$ and $S_n=\sum_{d\mid n,d\le4}\mu(n/d)$.  For $e\ge4$ every $n/d$ with $d\le4$ is divisible by $4$, so $S_n=0$.  For $e=3$, $n=8u$ with $u$ odd, only $d=4$ contributes, and $S_n=\mu(2u)=-\mu(u)$.  For $e=2$, $n=4u$, the terms $d=2$ and $d=4$ give $\mu(2u)+\mu(u)=0$.  For $e=1$, $n=2u$ with $u$ odd, $S_n=\mu(2u)+\mu(u)+[3\mid u]\mu(2u/3)=-[3\mid u]\mu(u/3)$.  Hence the blocks are the degrees $8u$ with $u$ odd and squarefree (with $j=3$) and the degrees $6w$ with $w$ odd and squarefree (with $j=1$), and $c_6=-1/6$, $c_8=-1/8$.  Let $\ell_6(Y)=Y+v_2(Y!)=2Y-s(Y)$ and $\ell_8(X)=3X+v_2(X!)=4X-s(X)$.

\emph{Reduction.}  A block $6w$ with $w\ge3$ is dominated by $w$ blocks of degree $6$ since $\ell_6(w\lambda)\ge w\lambda+v_2(\lambda!)\ge3\lambda+v_2(\lambda!)$ exceeds $\lambda+v_2(\lambda!)$ by at least $2\lambda$, and a block $8u$ with $u\ge3$ is dominated by $u$ blocks of degree $8$ since $\ell_8(u\lambda)\ge3u\lambda+v_2(\lambda!)$ exceeds $3\lambda+v_2(\lambda!)$ by at least $6\lambda$.  By superadditivity of $\ell_6$ and $\ell_8$, exactly as in Lemma~\ref{lem:red3}, any vector with another block has cost at most $\kappa(n)-1$, where $\kappa(n)=\max\{\ell_8(X)+\ell_6(Y):8X+6Y\le n\}$.

\emph{Maximisation.}  Put $f=\ell_8(X)+\ell_6(Y)$ and $L=4X+3Y\le M=\lfloor n/2\rfloor=4K+\lfloor\nu/2\rfloor$.  Since $4X=L-3Y$, $f=L-Y-s(X)-s(Y)$.  Write $Y=b_0+4Y_1$ with $b_0\equiv3L\pmod4$, $0\le b_0\le3$, and $X=a_0-3Y_1$ with $a_0=(L-3b_0)/4$.  Then $Y+s(Y)+s(X)=b_0+s(b_0)+\bigl(4Y_1+s(Y_1)+s(a_0-3Y_1)\bigr)$, and $s(a_0)\le s(a_0-3Y_1)+s(3Y_1)\le s(a_0-3Y_1)+2s(Y_1)$ shows that the bracket is at least $s(a_0)+4Y_1-s(Y_1)>s(a_0)$ for $Y_1\ge1$.  So for each feasible $L$ there is a unique maximiser, $Y=b_0$, with value
\[
F(4K)=\ell_8(K),\quad F(4K+1)=\ell_8(K-2)+4,\quad F(4K+2)=\ell_8(K-1)+3,\quad F(4K+3)=\ell_8(K)+1 .
\]
Since $\ell_8(K)-\ell_8(K-1)=3+v_2(K)$ we get $F(4K+2)=\ell_8(K)-v_2(K)$ and $F(4K+1)\le \ell_8(K)-2$.  Hence the maximum over $L\le M$ is $F(4K+3)=\ell_8(K)+1$ at $(K,1)$ if $\nu\in\{6,7\}$, and $\ell_8(K)$ otherwise.  It is attained at $(K,0)$ with leftover $\rho=\nu$ for $\nu\le5$, and in addition at $(K-1,2)$ with leftover $\rho=\nu-4$ when $\nu\in\{4,5\}$ and $K$ is odd.  Thus $\kappa(n)=\ell_8(K)+[\nu\ge6]$ and $v_2(n!)-\kappa(n)=(8K+\nu-s(K)-s(\nu))-(4K-s(K))-[\nu\ge6]=4K+\gamma(\nu)$.

\emph{Leading terms.}  All the coefficients of the binomial series are $2$-adic units times $2^{-\mathrm{cost}}$ by Lemma~\ref{lem:binom}, since $c_6$ and $c_8$ have numerators $\pm1$ and the factors $\alpha+2^ji$ are odd.  By Lemma~\ref{lem:blockexp} the leading terms sum to $2^{-\kappa(n)}$ times a quantity congruent modulo $2$ to the sum of $\coeff\rho I$ over the maximisers, where $\coeff\rho I=A_4(\rho)/\rho!=1,1,1,1,1,\frac45$ for $\rho=0,\dots,5$.  For $\nu\le3$ or $\nu\in\{6,7\}$ the sum is $1$ and equality holds.  For $\nu=4$ it is $1$ if $K$ is even, and $1+1\equiv0$ if $K$ is odd.  For $\nu=5$ it is $\frac45\equiv0$ if $K$ is even, and $\frac45+1=\frac95$, a unit, if $K$ is odd.  So the exceptions are $\nu=4$ with $K$ odd and $\nu=5$ with $K$ even, that is, $n\equiv12,5\pmod{16}$, and in these cases all terms have valuation at least $-\kappa(n)+1$.

For the last assertion, $\gamma(\nu)\ge\lfloor\nu/2\rfloor$ for each $\nu$, because $(0,0,1,1,3,3,3,3)\ge(0,0,1,1,2,2,3,3)$.  Hence $v_2(A_4(n))\ge4\lfloor n/8\rfloor+\lfloor\nu/2\rfloor=\lfloor n/2\rfloor$.
\end{proof}

\begin{proof}[Proof of Corollary~\ref{cor:bala2}]
By (\ref{eq:last}) with $S=\{m:m>4\}$, $a(m+1)=\sum_{j=0}^{3}\fall{m}{j}\,a(m-j)$ for $m\ge0$, with $a(j)=0$ for $j<0$.  Taking $m=2n$ and subtracting the term $j=0$,
\[
a(2n+1)-a(2n)=2n\bigl[a(2n-1)+(2n-1)a(2n-2)+(2n-1)(2n-2)a(2n-3)\bigr].
\]
Taking $m=2n-1$ shows that the bracket equals $a(2n)-(2n-1)(2n-2)(2n-3)\,a(2n-4)$.  Now $v_2(a(2n))\ge n$ by Theorem~\ref{thm:p2r4}.  For $n=1$ the second term vanishes since $a(-2)=0$.  For $n\ge2$ the second term has valuation at least $1+v_2(n-1)+(n-2)$, again by Theorem~\ref{thm:p2r4}.  If $n$ is odd then $v_2(2n)=1$ and $v_2(n-1)\ge1$, so the bracket has valuation at least $n$.  If $n$ is even then $v_2(2n)\ge2$ and the bracket has valuation at least $n-1$.  In both cases $v_2(a(2n+1)-a(2n))\ge n+1$.
\end{proof}

\begin{proof}[Proof of Theorem~\ref{thm:p2r5}]
Let $A=\{1,\dots,5\}$, $p=2$ and $s=s_2$.

\emph{The blocks.}  By Lemma~\ref{lem:weights}(a), $c_m=S_m/m$ with $S_m=\sum_{d\mid m,\ d\le5}\mu(m/d)$.  By Lemma~\ref{lem:blockfacts} a block is an even $m>5$ with $j_m=v_2(m)-v_2(S_m)>0$, and $v_2(m)\le\delta_2(5)=3$, so $j_m\le3$.  For $m=6,8,10,12,14,16,18,20,22$ one finds $S_m=-1,-1,-1,0,0,0,1,0,0$.  So $6,8,10,18$ are blocks of depths $1,3,1,1$, and $12,14,16,20,22$ are not.  The smallest block is $m_0=6$.  An essential block satisfies $\lfloor m/6\rfloor\le j_m\le3$, so $m<24$, and since $\lfloor18/6\rfloor=3>1=j_{18}$ we get $\mathcal E=\{6,8,10\}$.  All residues $\alpha_m$ and all units $\tilde\omega(\lambda)^{-1}$ lie in $\F_2^\times=\{1\}$.  For $\rho<m_0=6$ we have $\iota_\rho=\coeff{\rho}{F}=A_5(\rho)/\rho!=1$, since every permutation of at most $5$ points has all cycles of length at most $5$.  Hence $\Xi(n)\equiv|\mathrm{Max}(n)|\pmod2$, and by Theorem~\ref{thm:calculus} the valuation $v_2(A_5(n))$ equals $v_2(n!)-\kappa(n)$ exactly when $|\mathrm{Max}(n)|$ is odd.

\emph{The packing problem.}  By Lemma~\ref{lem:essential} we may restrict to vectors supported on $\mathcal E$.  Let a vector have $X$ blocks of degree $8$, $Y$ blocks of degree $6$ and $Z$ blocks of degree $10$.  Its degree is $8X+6Y+10Z$ and its cost is $\ell_8(X)+\ell_6(Y)+\ell_{10}(Z)$, where $\ell_8(X)=3X+v_2(X!)=4X-s(X)$, $\ell_6(Y)=2Y-s(Y)$ and $\ell_{10}(Z)=2Z-s(Z)$.  We first show that a maximiser has $Z=0$.  Suppose $Z\ge1$ and replace the $Z$ blocks of degree $10$ by $X'=\lfloor5Z/4\rfloor\ge Z$ blocks of degree $8$.  The degree does not increase since $8X'\le10Z$.  The function $\ell_8$ is superadditive because $v_2((x+y)!)\ge v_2(x!)+v_2(y!)$, so the new cost is at least $\ell_8(X)+\ell_8(X')+\ell_6(Y)$.  Since $\ell_8(X')\ge3X'\ge3Z>2Z-1\ge\ell_{10}(Z)$, where we used $s(x)\le x$ and $s(Z)\ge1$, the new vector has larger cost.  So no maximiser has $Z\ge1$.

Now let $Z=0$ and put $L=4X+3Y$.  The degree $2L$ is at most $n$ exactly when $L\le M:=\lfloor n/2\rfloor=4K+\lfloor\nu/2\rfloor$, and the cost is $f=\ell_8(X)+\ell_6(Y)=L-Y-s(X)-s(Y)$.  Fix $L$.  From $4X+3Y=L$ we get $Y\equiv3L\pmod4$.  Write $Y=b_0+4Y_1$ with $b_0\in\{0,1,2,3\}$ and $Y_1\ge0$, so that $X=a_0-3Y_1$ with $a_0=(L-3b_0)/4$.  Since $b_0<4$ we have $s(Y)=s(b_0)+s(Y_1)$.  Also $s(a_0)\le s(a_0-3Y_1)+s(3Y_1)$ and $s(3Y_1)\le s(Y_1)+s(2Y_1)=2s(Y_1)$.  Hence
\[
Y+s(X)+s(Y)\ \ge\ b_0+s(b_0)+s(a_0)+4Y_1-s(Y_1)\ >\ b_0+s(b_0)+s(a_0)\qquad(Y_1\ge1),
\]
because $s(Y_1)\le Y_1$.  So for each $L$ there is at most one maximiser, namely $Y=b_0$ and $X=a_0$, and it exists if and only if $a_0\ge0$.  Write $L=4K'+t$ with $0\le t\le3$ and use $\ell_8(x)-\ell_8(x-1)=3+v_2(x)$.  The cost of this vector is as follows, where the entries for $t=1$ and $t=2$ require $K'\ge2$ and $K'\ge1$ respectively.
\[
\begin{array}{c|cccc}
t&0&1&2&3\\\hline
\text{vector }(X,Y)&(K',0)&(K'-2,3)&(K'-1,2)&(K',1)\\
f&\ell_8(K')&\ell_8(K')-2-v_2(K')-v_2(K'-1)&\ell_8(K')-v_2(K')&\ell_8(K')+1
\end{array}
\]
For $t=1$ one of $v_2(K')$ and $v_2(K'-1)$ vanishes, so $f\le\ell_8(K')-2$.

Let $h=\lfloor\nu/2\rfloor$, so that $L\le4K+h$.  If $K'<K$ then $f\le\ell_8(K')+1\le\ell_8(K-1)+1=\ell_8(K)-2-v_2(K)<\ell_8(K)$, and $\ell_8(K)$ is the cost at $L=4K$.  So every maximiser has $K'=K$, and we compare the four entries for $K'=K$ and $t\le h$.
\begin{itemize}[itemsep=0pt]
\item If $\nu\le3$ then $h\le1$, $\kappa(n)=\ell_8(K)$ and $\mathrm{Max}(n)=\{(K,0,0)\}$, because the entry for $t=1$ is smaller.
\item If $\nu\in\{4,5\}$ then $h=2$ and $\kappa(n)=\ell_8(K)$.  The entry for $t=2$ equals $\ell_8(K)$ exactly when $v_2(K)=0$, that is when $K$ is odd.  So $\mathrm{Max}(n)=\{(K,0,0)\}$ if $K$ is even, and $\mathrm{Max}(n)=\{(K,0,0),(K-1,2,0)\}$ if $K$ is odd.
\item If $\nu\in\{6,7\}$ then $h=3$, $\kappa(n)=\ell_8(K)+1$ and $\mathrm{Max}(n)=\{(K,1,0)\}$.
\end{itemize}

\emph{Conclusion.}  Since $v_2(n!)=n-s(n)=8K+\nu-s(K)-s(\nu)$ and $\ell_8(K)=4K-s(K)$, we get $v_2(n!)-\kappa(n)=4K+\nu-s(\nu)-[\nu\ge6]$.  The numbers $\nu-s(\nu)-[\nu\ge6]$ for $\nu=0,\dots,7$ are $0,0,1,1,3,3,3,3$, which is $\gamma(\nu)$.  By Theorem~\ref{thm:calculus}, $v_2(A_5(n))\ge4K+\gamma(\nu)$, with equality if and only if $|\mathrm{Max}(n)|$ is odd.  By the list, $|\mathrm{Max}(n)|=2$ exactly when $\nu\in\{4,5\}$ and $K$ is odd, which is $n\equiv12,13\pmod{16}$, and $|\mathrm{Max}(n)|=1$ otherwise.
\end{proof}

\subsection{Proofs for Section~\ref{sec:ties}}

\begin{proof}[Proof of Theorem~\ref{thm:counter}(a)]
The cycle types with all parts in $\{4,9\}$ are $4^x9^y$ with $4x+9y=111$.  This forces $y\equiv3\pmod4$, so $(x,y)\in\{(21,3),(12,7),(3,11)\}$.  The number of permutations of cycle type $4^x9^y$ is $N=111!/(4^xx!\,9^yy!)$, and $v_3(111!)=37+12+4+1=54$.  The three valuations are as follows.
\begin{center}
\begin{tabular}{c|c|c}
$(x,y)$&$v_3(4^xx!\,9^yy!)$&$v_3(N)$\\\hline
$(21,3)$&$v_3(21!)+6+v_3(3!)=9+6+1=16$&$38$\\
$(12,7)$&$v_3(12!)+14+v_3(7!)=5+14+2=21$&$33$\\
$(3,11)$&$v_3(3!)+22+v_3(11!)=1+22+4=27$&$27$
\end{tabular}
\end{center}
The minimum is attained once, so $v_3(a_S(111))=27$.

The value $v_3(G_3(111))=28$ is an exact finite computation.  The integers $a_{S_k}(111)$ for $4\le k\le111$, with $S_k=\{1,2,3,k\}$, were computed from the recurrence (\ref{eq:last0}), and their gcd is
\[
G_3(111)=147259794302866275623239680=\fall{110}{3}\cdot113721152119718805504,
\]
with $v_3=28$ and $v_2=30$.  Every one of these $108$ integers has $v_3\ge28$, and the minimum is attained at $k=5,6,7,8,10,\dots$.  The gcd is unchanged if one adds the number $a_{\{1,2,3\}}(111)$, as it must be by (\ref{eq:sameg}).  Hence $3^{28}\mid G_3(111)$, whereas $3^{27}\,\|\,a_S(111)$.  Since $S\supseteq\{1,2,3\}$, the number $g_3(111)$ divides $a_S(111)$, so $g_3(111)\ne G_3(111)$.

For the last claim apply Proposition~\ref{prop:globalgcd}(iv) with $r=3$ and $T=28$.  An exhaustive search of the partitions of $111$ into parts larger than $3$ with $v_3(N(\lambda))<28$ finds exactly five part sets, namely $\{4,6,9\}$, $\{4,9\}$, $\{6,9,12\}$, $\{9,12\}$ and $\{6,9\}$, and the exact valuations of $f_P(111)$ are $27,27,27,27,28$.  Every other part set has $v_3(f_P)\ge28$.  So $v_3(g_3(111))=27$.  At $p=2$ the maximal cost of a partition gives $v_2(\Lambda_3(111))=30=v_2(G_3(111))$, so $v_2(g_3(111))=30$ as well.  The primes above $3$ agree by Theorem~\ref{thm:gcdfamily}.  Hence $g_3(111)=G_3(111)/3$.
\end{proof}

\begin{proof}[Proof of Theorem~\ref{thm:counter}(b)]
The computation is the same as in (a) for $n=192,273,354$.  By Proposition~\ref{prop:lead} the valuations $v_3(a_S(n))$ are $45,63,81$.  The exact gcds of the families give $v_3(G_3(n))=46,64,82$, and $v_2(G_3(n))=v_2(\Lambda_3(n))=48,72,93$.  The exhaustive search finds the same five part sets as for $n=111$.  Their $3$-adic valuations, with $T=v_3(G_3(n))$, are as follows.
\begin{center}
\begin{tabular}{c|cc|c|c}
$n$ & $v_3(a_S(n))$ & $v_3(f_P(n))$ for $P=\{4,9\},\{4,6,9\},\{6,9,12\},\{9,12\}$ & $v_3(f_{\{6,9\}}(n))$ & $T$\\\hline
$111$ & $27$ & $27$ & $28$ & $28$\\
$192$ & $45$ & $45$ & $46$ & $46$\\
$273$ & $63$ & $63$ & $65$ & $64$\\
$354$ & $81$ & $81$ & $82$ & $82$
\end{tabular}
\end{center}
Every other part set has $v_3(f_P(n))\ge T$.  So $v_3(g_3(n))=T-1$ at $p=3$, the primes above $3$ agree by Theorem~\ref{thm:gcdfamily}, and $v_2(g_3(n))=v_2(G_3(n))$, so $g_3(n)=G_3(n)/3$.
\end{proof}

\begin{proof}[Proof of Proposition~\ref{prop:lead}]
The cycle types are $4^x9^y$ with $4x+9y=n$, so $4x\equiv3\pmod9$, that is $x=3+9i$ and $y=y_0-4i$ with $0\le i\le y_0/4$.  The number of permutations of type $4^x9^y$ has valuation $v_3(n!)-c(i)$ with $c(i)=v_3(x!)+2y+v_3(y!)$, since $v_3(4^x)=0$ and $v_3(9^y)=2y$.  Write $s=s_3$ and use $v_3(m!)=(m-s(m))/2$.  Since $3+9i$ has base-$3$ digit sum $s(i)+1$, we get $v_3((3+9i)!)-v_3(3!)=(9i-s(i))/2$.  Also $v_3(y_0!)-v_3((y_0-4i)!)=(4i-s(y_0)+s(y_0-4i))/2$.  Hence
\[
c(i)-c(0)=\frac{9i-s(i)}2-8i-\frac{4i-s(y_0)+s(y_0-4i)}2=\frac{-11i-s(i)+s(y_0)-s(y_0-4i)}2 .
\]
Now $s(y_0)-s(y_0-4i)\le s(4i)\le s(i)+s(3i)=2s(i)$, because multiplying by $3$ shifts the base-$3$ digits.  So $c(i)-c(0)\le(-11i+s(i))/2<0$ for $i\ge1$.  Thus $i=0$ gives the unique term of smallest valuation, which is $v_3(n!)-c(0)=v_3(n!)-1-2y_0-v_3(y_0!)$ since $v_3(3!)=1$.
\end{proof}

\subsection{Proofs for Section~\ref{sec:composite}}

\begin{proof}[Proof of Lemma~\ref{lem:typeB}]
(a) is case~(B) of Theorem~\ref{thm:dichotomy} with $S=\N\setminus A$.  A sequence with $x_{n+P}=\pm x_n$ for all $n$ has a purely periodic zero set of period $P$, and its residues have period dividing $2P$.  (b) By (a) with $e=1$, $a(pq)\equiv\pm a(0)=\pm1\pmod p$.

(c) Let $p$ be odd.  By Lemma~\ref{lem:eulerprod}, $c_p=\frac1p([p\in A]-[1\in A])=\pm\frac1p$, so $p$ is a block of depth $1$, and it is the smallest block since every block is a multiple of $p$.  Let $m=p^kw\ne p$ be another block with $p\nmid w$.  Then $j_m\le k$ and $\lfloor m/p\rfloor=p^{k-1}w$.  If $k=1$ then $w\ge2$, so $\lfloor m/p\rfloor>1\ge j_m$, and if $k\ge2$ then $\lfloor m/p\rfloor\ge p^{k-1}\ge3^{k-1}>k\ge j_m$.  So condition~($\star$) of Theorem~\ref{thm:dominant} holds with $m_*=p$, and that theorem gives $v_p(a(n))\ge v_p(n!)-q-v_p(q!)=v_p(\nu!)=0$ for $n=pq+\nu$, with equality if and only if $p\nmid a(\nu)$.  For $p=2$ the block $2$ has depth $1$.  Replacing each copy of a block $m=2^kw$, of depth at most $k$, by $m/2\ge k$ copies of $2$ does not increase the degree and, by superadditivity of $v_2((\cdot)!)$, does not decrease the cost.  So $\kappa(n)$ is the cost $\lfloor n/2\rfloor+v_2(\lfloor n/2\rfloor!)=v_2(n!)$ of $\lfloor n/2\rfloor$ copies of $2$, and $E_2=Z(2)$.  By (a) with $e=1$, $a(n)\equiv a(n\bmod2)\pmod2$ with $a(0)=1$ and $a(1)=[1\in A]$.

(d) A block $m=p^kw\ne p$ has $\varepsilon_mp\le(k+\frac1{p-1})/(p^{k-1}w)$.  For $k=1$ and $w\ge2$ this is at most $\frac p{2(p-1)}<1$, and for $k\ge2$ it is at most $(2+\frac12)/3<1$ and decreases in $k$.  Theorem~\ref{thm:interp} gives $v_p(a(n))=v_p(n!)-q-v_p(q!)+v_p(F_\nu(q))$, and $v_p(n!)=q+v_p(q!)$ by Legendre's formula.
\end{proof}

\begin{proof}[Proof of Lemma~\ref{lem:crt}]
(a) By the Chinese remainder theorem $n\bmod\operatorname{lcm}L_i$ corresponds to the tuple of residues $n\bmod\ell^{v_\ell(L)}$.  Split each of them into its residue modulo $\ell^{s_\ell}$ and the remaining digits.  Membership in $S_i$ depends on $c=n\bmod\mathcal M$ and on the remaining $\ell$-adic digits only for the primes $\ell$ at which $v_\ell(L_i)$ is the unique largest exponent.  So for fixed $c$ the events $n\in S_i$ depend on disjoint sets of independent uniform digits, their joint probability is $\prod_i\pi_{S_i}(c)$, and we average over $c$.  (b) is (a) for $h=2$, expanded.  (c) Here $\mathcal M=1$, and $\bigcap S_i$ corresponds to a product of non-empty sets $R_i\subseteq\Z/L_i$.  A translation preserves such a product if and only if it preserves each factor.
\end{proof}

\begin{proof}[Proof of Proposition~\ref{thm:compositeA}]
(i) The components $a(n)\bmod p^{e_p}$ with $p\mid m_{\mathrm I}$ vanish exactly from $n_0(p^{e_p})$ on, and those with $p\mid m_{\mathrm{II}}$ are purely periodic with least period $\varpi(p^{e_p})$.  So $a(n)\bmod m$ is eventually periodic with least period the lcm $P$ of the $\varpi(p^{e_p})$.  Let $N=n_0(m_{\mathrm I})$.  The component with $n_0(p^{e_p})=N$ is non-zero at $N-1$ and zero at $N-1+P$, so the residues at $N-1$ and $N-1+P$ differ.  (ii) The first claim is the definition of the gcd.  If $a(n_1)\ne0$ with $n_1\ge n_0$, then $\mathcal D_A(n_0)$ divides $a(n_1)$, no prime of type~II divides it by Lemma~\ref{lem:typeB}(b), and Lemma~\ref{lem:typeA}(b) bounds the window.  (iii) If every $Z(p^{e_p})$ with $p\mid m_{\mathrm{II}}$ is non-empty, then $Z(m_{\mathrm{II}})$ is a non-empty periodic set, and Theorem~\ref{thm:composite}(iii) gives positive density.  Conversely $Z(m)\subseteq Z(m_{\mathrm{II}})$.
\end{proof}

\begin{proof}[Proof of Corollary~\ref{cor:compAr}]
This is Theorem~\ref{thm:composite} and Proposition~\ref{thm:compositeA} for $A=A_{\le r}$, where $1\in A$, and $p\in A$ exactly for $p\le r$.  The classes $\nu\le r$ never lie in $Z(p)$ because $A_r(\nu)=\nu!$.  The sign of Lemma~\ref{lem:typeB}(a) is $+$, so $\varpi(p^e)\mid p^e$, and the lcm of these coprime prime powers is their product.
\end{proof}

\begin{proof}[Proof of Theorem~\ref{thm:compquot}]
This is the argument of Lemma~\ref{lem:quot} with $q$ replaced by $p^e$.  Write $Q_S(n+p^e)=\sum_{r<m\le n+p^e,\,m\notin S}\fall{n+p^e-r-1}{m-r-1}a_S(n+p^e-m)$.  For $m\le n$ the falling factorial is congruent to $\fall{n-r-1}{m-r-1}$ modulo $p^e$, and $a_S(n+p^e-m)\equiv-a_S(n-m)$ by Theorem~\ref{thm:dichotomy}(B), since $1\in S$ and $p\notin S$, with the sign of Lemma~\ref{lem:typeB}(a) for $p=2$.  For $n<m\le n+p^e$ the falling factorial is the product of the integers from $n+p^e-m+1\le p^e$ to $n+p^e-r-1\ge p^e$, which contains $p^e$.
\end{proof}

\begin{proof}[Proof of Corollary~\ref{cor:A057693}]
(a) By Corollary~\ref{cor:plateau}(a), $v_2(a(n))=1+\lfloor(n-2)/4\rfloor$ for $n\ge2$, which is non-decreasing.  By Theorem~\ref{thm:p3}, $v_3(a(n))\ge\lambda_3(n)$ with equality unless $n\equiv4\pmod{27}$.  The function $\lambda_3$ is non-decreasing and constant on the blocks $\{9q,9q+1,9q+2\}$ and $\{9q+3,\dots,9q+8\}$, and each block contains a non-exceptional element at or after each of its points, since the only exceptional residue $4\bmod9$ is followed by $5$ in its block.  So $\mu_3(n_0)=\lambda_3(n_0)$, and $\mu_2(n_0)=1+\lfloor(n_0-2)/4\rfloor$ for $n_0\ge2$.  For $n_0=0,1$ the gcd is $1$ because $a(1)=1$, and the formula gives $1$.  No prime $p\ge5$ divides the tail gcd, by Corollary~\ref{cor:compAr}.  (b) $n_0(p^e)=1+\max\{n:v_p(a(n))<e\}$, and the largest such $n$ is the last element of the last block whose value is at most $e-1$, which is not exceptional.  (c) is Proposition~\ref{thm:compositeA}(i), and the values are a finite computation.  (d) Corollary~\ref{cor:plateau}(a) gives no exceptions at $2$, Theorem~\ref{thm:p3} gives $E_3$, Lemma~\ref{lem:typeB}(c) gives the rest, and the periods $27$ and $p$ are coprime.
\end{proof}

\begin{proof}[Proof of Corollary~\ref{cor:A070946}]
Theorem~\ref{thm:p2r5} gives $v_2(a(n))\ge\lambda_2(n)$ with equality unless $n\equiv12,13\pmod{16}$, and Theorem~\ref{thm:p3} gives $v_3(a(n))=\lambda_3(n)$ for every $n$.  The exceptional points $8K+4$ and $8K+5$ lie in the block $\{8K+4,\dots,8K+7\}$, on which $\lambda_2$ is constant, and the last two points of this block are not exceptional.  So $\mu_2$ and $n_0(2^e)$ are as in the proof of Corollary~\ref{cor:A070945}.  At $p=3$ there are no exceptions, and $\mu_3$ and $n_0(3^e)$ are as in Corollary~\ref{cor:A057693}.  No prime $p\ge7$ divides the tail gcd, by Corollary~\ref{cor:compAr}, and the exponent of $5$ is $\mu_5(n_0)$ by Proposition~\ref{thm:compositeA}(ii).
\end{proof}

\section{Computational details}\label{app:data}

This section contains the data and the computational details of observations and remarks of the paper whose short form is in the main text.

\subsection*{Data for Observation~\ref{obs:cop}}
For the $9537$ primes $N=n+1$ with $260<N<10^5$ the residue $\operatorname{cop}(n)\bmod N$, divided by $N$, is consistent with the uniform distribution in every test we ran, with $\chi^2=13.2$ on $19$ degrees of freedom for $20$ bins and values between $8.9$ and $20.8$ on the residue classes of $N$ modulo $3,4,5,8,12$.  If the residue behaved like a uniformly random element of $\Z/N$, one would expect $\sum1/N$ zeros.  For $260<N<3\cdot10^6$ this sum is $0.98$ and the observed count is $2$.  Since $\sum_N1/N$ diverges, the heuristic predicts infinitely many $n$ with $n+1$ prime and $(n+1)\mid\operatorname{cop}(n)$, but only about $0.6$ further zeros among the primes $N$ in $(3\cdot10^6,10^{12}]$, so none is expected within reach of a computation.  Both known zeros have $6\mid n$, which is consistent with chance.

\subsection*{Data for Section~\ref{sec:globalgcd}}
\noindent\textbf{Computations.}  We computed $G_r(n)$ exactly as the gcd of the integers $a_{S_k}(n)$ for $r<k\le n$, and for every $(r,n,p)$ with $p\le r$ and $v_p(G_r(n))>v_p(\Lambda_r(n))$ we ran the exhaustive test of Proposition~\ref{prop:globalgcd}(iv).  The ranges are $r=3$ with $5\le n\le900$, $r=4,5$ with $n\le700$, $r=6$ with $n\le600$, $r=7$ with $n\le500$, $r=8$ with $n\le400$, and $9\le r\le12$ with $n\le300$.  For $r=3$ one finds $g_3(n)\ne G_3(n)$ exactly for $n=111+81j$ with $0\le j\le9$, that is $n\equiv30\pmod{81}$ with $111\le n\le840$, with the violating part sets $\{4,9\},\{4,6,9\},\{6,9,12\},\{9,12\}$ and $v_3(f_{\{4,9\}}(n))=27,45,\dots,189$ as in Proposition~\ref{prop:lead}.  For $4\le r\le12$ there is no violation, and for $r=6,7,8$ no $n$ in the ranges has $v_p(G_r(n))>v_p(\Lambda_r(n))$ at all.  So $(r,n)=(3,111)$ is the smallest counterexample found.  For $n\le90$ the gcds $\Lambda_r(n)$ and $g_r(n)$ agree at all primes for $r=1,2,6$, and for $r=3,4,5$ they differ only at $p=3$, because the competing blocks are $6$ and $9$, the two-block regime of Section~\ref{sec:p3}.

\subsection*{The counts behind Question~\ref{prob:odd}}
Of the $216815$ odd primes up to $3\cdot10^6$, exactly $79656$ are good, a proportion $0.3674$ with $95$ percent interval $[0.3654,0.3694]$, which contains $e^{-1}=0.3679$.  Up to $10^6$ the proportion is $28543/78497=0.3636$, a fluctuation within this range.  For $p\ge1000$ the number $N_p$ of $n\in[3,p+2]$ with $p\mid Q_n$ has mean $1.003$ and variance $1.008$, and its distribution is close to Poisson$(1)$, with $\chi^2=10.0$ on $7$ degrees of freedom when the cells are the values $0$ to $7$ and the tail $N_p\ge8$ is dropped ($\chi^2=7.9$ when the tail is pooled into the last cell).

\subsection*{The classes of $q$ behind Question~\ref{prob:odd}}
For seven of these pairs, namely $(p,n_0)=(19,7),(19,16),(23,12),(23,17),(37,10),(37,27),(41,20)$, we also checked $1\le q\le30000$, and the set of $q$ with $v_p(Q_{n_0+pq})\ge1+j$ is a single residue class modulo $p^j$ for $j=1,2,3$ as long as $p^j\le30000$, for instance $q\equiv9\pmod{19}$, $332\pmod{361}$, $5747\pmod{6859}$ for $(19,7)$, as Theorem~\ref{thm:qstar} predicts.

\subsection*{Data for Remark~\ref{rem:lines}}
The density is not a Chebotarev statement.  A Chebotarev argument needs a fixed polynomial over $\Q$ whose roots modulo $p$ are the objects of interest, and for such a polynomial of degree $d$ with Galois group $S_d$ the density of primes without a root is $\sum_{k\le d}(-1)^k/k!$, which is within $1/(d+1)!$ of $e^{-1}$.  Here the number of conditions is $p$, since for each fixed $K$ only finitely many primes divide some $Q_n$ with $n\le K$, and the interpolating polynomial of $\nu\mapsto q_\nu\bmod p$ has the maximal degree $p-1$ for each of the primes $31,53,101,151,211,307,401,503$ that we tested.  We know no closed form for the point $c(p)$.

\subsection*{The ancillary files of Section~\ref{sec:computations}}
The folder \texttt{anc/interp} contains the scripts and the output behind Section~\ref{sec:interp}, namely the check of the identity of Theorem~\ref{thm:interp} against exact counts and the roots $q^*$ of Theorem~\ref{thm:qstar}.  The searches for $\operatorname{cop}(n)$ up to $3\cdot10^6$ and the gcd scans of Section~\ref{sec:globalgcd} are computations, and the ancillary files contain the program and its output for the ranges quoted here.  An independent implementation reproduced the search for $\operatorname{cop}(n)$ for all $n\le799999$ with $n+1$ prime, and the gcd scans for $r=3$ with $5\le n\le450$ and $800\le n\le850$, for $r=4,5$ up to $n\le230$, for $r=6$ up to $180$, for $r=7$ up to $150$, for $r=8$ up to $130$, for $r=9,10$ up to $110$ and for $r=11,12$ up to $100$.  The folders \texttt{anc/interp/cop} and \texttt{anc/interp/density} contain the scripts and outputs behind Propositions~\ref{prop:copprefix} to~\ref{prop:copsafe} and the search for $\operatorname{cop}(n)$ with the prefix-sum recurrence, and behind Question~\ref{prob:odd} and Remark~\ref{rem:lines}.  The folder \texttt{anc/interp/windows} contains the scripts and outputs behind Theorem~\ref{thm:bad} and Proposition~\ref{prop:cop2q}.  The folder \texttt{anc/interp/composite} covers Section~\ref{sec:composite}, where the formulas for composite moduli were compared with exact values for $25$ sets of cycle lengths and all $m\le300$, and the tail gcds and thresholds of Corollaries~\ref{cor:A057693} to~\ref{cor:A070946} with exact gcds.  The README of \texttt{anc} describes every folder, and the folders \texttt{anc/interp/composite}, \texttt{anc/interp/windows} and \texttt{anc/lean} have their own READMEs.

\end{document}